\documentclass[12pt,twoside,american,english,british]{article}
\usepackage[T1]{fontenc}
\usepackage[utf8]{inputenc}
\usepackage[a4paper]{geometry}
\usepackage{color}
\usepackage{enumitem}
\usepackage{amsmath}
\usepackage{amsthm}
\usepackage{amssymb}
\usepackage{setspace}
\usepackage{esint}

\makeatletter
\numberwithin{equation}{section}
\numberwithin{figure}{section}
\theoremstyle{plain}
\newtheorem{thm}{\protect\theoremname}[section]
\theoremstyle{definition}
\newtheorem{defn}[thm]{\protect\definitionname}
\theoremstyle{plain}
\newtheorem{lem}[thm]{\protect\lemmaname}

\usepackage[T1]{fontenc}
\usepackage[utf8]{inputenc}
\usepackage[a4paper]{geometry}
\usepackage{color}
\usepackage{dsfont}
\usepackage{amsmath}
\usepackage{amsthm}
\usepackage{amssymb}
\usepackage{setspace}
\usepackage{esint}

\makeatletter
\numberwithin{equation}{section}
\numberwithin{figure}{section}
\theoremstyle{plain}

\@ifundefined{date}{}{\date{}}
\usepackage{babel}
\usepackage{latexsym}
\usepackage{amsthm}
\usepackage{esint}
\usepackage{eucal}
\usepackage{epstopdf}
\usepackage{graphicx}
\usepackage{bigints}
\theoremstyle{plain}

\newtheoremstyle{boldremark}
    {\dimexpr\topsep/2\relax} % space above
    {\dimexpr\topsep/2\relax} % space below
    {}          % body font
    {}          % indent amount
    {\bfseries} % theorem head font
    {.}         % punctuation after theorem head
    {.5em}      % space after theorem head
    {}          % theorem hed spec. (empty = "normal")

\theoremstyle{boldremark}
\newtheorem{brem} [thm] {Remark} % remarks are numbered within sections

\usepackage{fancyhdr}
\usepackage{blindtext}
\usepackage{titlesec}

\titleformat{name=\chapter}[display]
{\normalfont\Huge\itshape}
{\titlerule[1pt]\vspace{-33pt}\filleft%
  \parbox[t]{6em}{%
    \raggedleft%
    \rule{\linewidth}{0.5ex}\newline%
    \chaptertitlename\ \thechapter%
  }%
}
{4pc}
{\normalfont\upshape\bfseries\Huge}
\allowdisplaybreaks

\makeatother

\usepackage{babel}
\addto\captionsbritish{\renewcommand{\definitionname}{Definition}}
\addto\captionsbritish{\renewcommand{\lemmaname}{Lemma}}
\addto\captionsbritish{\renewcommand{\theoremname}{Theorem}}
\addto\captionsenglish{\renewcommand{\definitionname}{Definition}}
\addto\captionsenglish{\renewcommand{\lemmaname}{Lemma}}
\addto\captionsenglish{\renewcommand{\theoremname}{Theorem}}
\providecommand{\definitionname}{Definition}
\providecommand{\lemmaname}{Lemma}
\providecommand{\theoremname}{Theorem}

\usepackage[colorlinks,pdfpagelabels,pdfstartview = FitH,bookmarksopen
= true,bookmarksnumbered = true,linkcolor = blue,plainpages =
false,hypertexnames = false,citecolor = red,pagebackref=false,urlcolor=blue]{hyperref}

\usepackage{amsmath}  % for \iint macro
\usepackage{graphicx} % for \rotatebox macro

\def\YYint#1#2#3{{\setbox0=\hbox{$#1{#2#3}{\iint}$}
    \vcenter{\hbox{$#2#3$}}\kern-.51\wd0}}
 
\usepackage{etoolbox} % va messo nel preambolo
\makeatletter
\pretocmd{\@makefntext}{\setlength\parindent{0pt}\noindent}{}{%
  \PackageWarning{MyDoc}{Patch of \string\@makefntext\ failed}%
}
\makeatother

\makeatother

\usepackage{babel}
\addto\captionsamerican{\renewcommand{\definitionname}{Definition}}
\addto\captionsamerican{\renewcommand{\lemmaname}{Lemma}}
\addto\captionsamerican{\renewcommand{\theoremname}{Theorem}}
\addto\captionsbritish{\renewcommand{\definitionname}{Definition}}
\addto\captionsbritish{\renewcommand{\lemmaname}{Lemma}}
\addto\captionsbritish{\renewcommand{\theoremname}{Theorem}}
\addto\captionsenglish{\renewcommand{\definitionname}{Definition}}
\addto\captionsenglish{\renewcommand{\lemmaname}{Lemma}}
\addto\captionsenglish{\renewcommand{\theoremname}{Theorem}}
\providecommand{\definitionname}{Definition}
\providecommand{\lemmaname}{Lemma}
\providecommand{\theoremname}{Theorem}

\begin{document}
\title{\textbf{Boundedness and contractive estimates}\\
\textbf{for orthotropic, widely degenerate,}\\
\textbf{doubly nonlinear diffusion equations}}
\author{Pasquale Ambrosio\thanks{\textbf{Corresponding author.} Department of Mathematics, Uppsala
University, P.O. Box 524, 751 20 Uppsala, Sweden.\textit{ E-mail address}:
pasquale.ambrosio@math.uu.se} , Matias Vestberg\thanks{Department of Mathematics, Uppsala University, P.O. Box 524, 751 20
Uppsala, Sweden.\textit{ E-mail address}: matias.vestberg@math.uu.se}}
\date{\noindent September 14, 2026}
\maketitle
\begin{abstract}
\begin{singlespace}
\noindent We study the regularity of weak solutions to doubly nonlinear
orthotropic evolution equations of the form 
\[
\partial_{t}(\vert u\vert^{\alpha-1}u)-\sum_{i=1}^{N}\partial_{i}\left[a_{i}(x,t)\,(|\partial_{i}u|-\delta_{i})_{+}^{p-1}\frac{\partial_{i}u}{\vert\partial_{i}u\vert}\right]=f\,\,\,\,\,\,\,\,\,\,\mathrm{in}\,\,\,\Omega_{T}=\Omega\times(0,T),
\]
where $\Omega$ is a bounded open subset of $\mathbb{R}^{N}$ for
$N\geq2$, the coefficients $a_{i}$ are measurable and bounded, $\alpha>0$,
$p\in(1,\infty)$ and $\delta_{1},\ldots,\delta_{N}$ are non-negative
numbers. We show that weak solutions are locally bounded due to their
membership in a suitable De Giorgi-type energy class. We also obtain
contractive estimates and global boundedness in space for solutions
to a Cauchy problem associated with the above PDE. Our analysis extends
analogous results available in the literature for diffusion equations
that either do not exhibit double nonlinearity or are less degenerate
than those considered here. Another main novelty of this paper is
the presence of a source term $f$ on the right-hand side of the equation,
for which we impose suitable integrability assumptions in the space-time
variables.\vspace{0.2cm}
\end{singlespace}
\end{abstract}

\selectlanguage{american}%

\section{Introduction\label{sec: intro}}

\selectlanguage{british}%
$\hspace*{1em}$Let \foreignlanguage{american}{$\Omega$ be a bounded
open subset} of $\mathbb{R}^{N}$, $N\geq2$, and let $T\in(0,\infty)$.
We are concerned \foreignlanguage{american}{with local and global
regularity properties of weak solutions to doubly nonlinear evolution
equations of the form 
\begin{align}
\partial_{t}(\vert u\vert^{\alpha-1}u)-\sum_{i=1}^{N}\partial_{i}\left[a_{i}(x,t)\,(|\partial_{i}u|-\delta_{i})_{+}^{p-1}\frac{\partial_{i}u}{\vert\partial_{i}u\vert}\right]=f\,\,\,\,\,\,\,\,\,\,\mathrm{in}\,\,\,\Omega_{T}:=\Omega\times(0,T),\label{eq:diffusion}
\end{align}
where $\alpha>0$, $p\in(1,\infty)$, $\delta_{1},\ldots,\delta_{N}$
are non-negative real numbers, }and $\left(\,\cdot\,\right)_{+}$
stands for the positive part\foreignlanguage{american}{. Here $a_{i}$,
$i\in\{1,\ldots,N\}$, are measurable coefficients satisfying
\begin{align}
\Lambda^{-1}\,\leq\,a_{i}(x,t)\,\leq\,\Lambda\,\,\,\,\,\,\,\,\,\,\mathrm{a.e.}\,\,\,\mathrm{in}\,\,\,\Omega_{T},\label{eq:coeff_limit}
\end{align}
}for some constant $\Lambda\geq1$.\foreignlanguage{american}{ The
precise assumptions on the source term $f=f(x,t)$ will vary according
to the specific result under consideration, and will concern only
the integrability of $f$ in the space-time variables.}\\
$\hspace*{1em}$\foreignlanguage{american}{In the special case $\alpha=1$
and $f=0$, evolution equations of a form similar to (\ref{eq:diffusion})
have very recently been studied from different perspectives in \cite{Amb-LIP,AmbCia,AmbCiaCup},
as well as in \cite{BBLV-par} and \cite{CiaVeVe}, where, however,
$\delta_{i}=0$ for all $i\in\{1,\ldots,N\}$. In particular, in \cite{AmbCia,AmbCiaCup}
and \cite{CiaVeVe} the authors established, among other results,
the local boundedness of local weak solutions, while the works \cite{Amb-LIP}
and \cite{BBLV-par} are concerned with spatial Lipschitz continuity
of local weak solutions in the case $p\geq2$. In addition, equation
(\ref{eq:diffusion}) with $\alpha=1$, $f=0$, all $a_{i}$ set to
1 and all $\delta_{i}$ equal to zero }is explicitly presented in
the monographs \cite{Lions}, \cite[Example 4.A, Chapter III]{Show}
and \cite[Example 30.8]{Zeid}, among others.\foreignlanguage{american}{}\\
$\hspace*{1em}$\foreignlanguage{american}{To the best of our knowledge,
the PDEs considered in this paper are the first in the literature
to combine, within a unified framework, three of the most studied
classes of nonlinear equations in recent years: the \textit{doubly
nonlinear equations}, the \textit{widely degenerate equations} and
the\textit{ orthotropic} ones.}\\
$\hspace*{1em}$\foreignlanguage{american}{Doubly nonlinear equations
were introduced in the 1960s by Lions \cite{Lions} and subsequently
investigated by Kalashnikov \cite{Kal}. The expression \textit{doubly
nonlinear} refers to the presence of nonlinearities both in the evolution
term and in the diffusion part of the equation. Equations of this
type arise in a broad range of physical contexts, including, for instance,
the flow of nonhomogeneous non-Newtonian fluids and the simultaneous
motion of water in surface channels and underground aquifers, just
to name a few. We refer to \cite[Chapter 4]{AnDiSh} and the references
therein for an account of the applications. Although the theory is
quite well developed, especially in the degenerate and singular supercritical
cases (for the exact definitions of these technical terms, see for
example \cite{FoHeVe}), some important questions still remain open
and continue to be the subject of active research.}\\
$\hspace*{1em}$Widely degenerate equations have emerged as an important
class of nonlinear problems and have received increasing attention
in the literature; see, for instance, \cite{Amb00,Amb2,AmGrPa,BDGPell,Brasco,BraCarSan,CGHP,CoFi,CGGP,Grim,GriRus,Mons,Picc,Russo,Strunk-ell}
in the elliptic setting and \cite{Amb1,Amb3,AmbBau,AmbCuDe,AmbPass,BoDuGiPa,GenPas,Strunk}
in the parabolic context. In these latter works, the model equation
under consideration is the following:
\begin{equation}
\partial_{t}u-\mathrm{div}\left((\vert\nabla u\vert-\lambda)_{+}^{p-1}\frac{\nabla u}{\vert\nabla u\vert}\right)=\,\phi\,,\,\,\,\,\,\,\,\,\mathrm{with}\,\,\lambda>0\,.\label{eq:AmbPass}
\end{equation}
According to the now-standard terminology, this equation is widely
degenerate, in the sense that the diffusion part is uniformly elliptic
only outside the region $\{\vert\nabla u\vert\leq\lambda\}$, while
it behaves asymptotically, that is, for large values of $\vert\nabla u\vert$,
like the parabolic $p$-Laplacian. Therefore, equations of the form
(\ref{eq:AmbPass}) fall within the class of \textit{asymptotically
regular parabolic problems} (for a comprehensive overview of this
topic, see \cite{Amb1,AmbPass,BoDuGiPa} and the references therein).
As already pointed out in \cite{Amb3,BoDuGiPa}, one cannot expect
more than Lipschitz regularity for solutions to equation (\ref{eq:AmbPass}).
In fact, when $\phi=0$, any time-independent $\lambda$-Lipschitz
function solves (\ref{eq:AmbPass}), and even more, it is a solution
of the associated stationary equation.\\
Returning to our setting, equation (\ref{eq:diffusion}) is also widely
degenerate, but with an ``orthotropic'', coordinate-wise degeneracy,
since the ellipticity of the corresponding operator in divergence
form breaks down on the set
\[
\bigcup_{i=1}^{N}\,\{\vert\partial_{i}u\vert\leq\delta_{i}\}\,.
\]
In this paper, the term \textit{orthotropic} refers precisely to the
fact that the diffusion coefficients in (\ref{eq:diffusion}) depend
separately on the absolute values of the components $\partial_{i}u$
of the spatial gradient $\nabla u$, rather than on the Euclidean
norm $\vert\nabla u\vert$, and may vary across the coordinate directions,
with possibly different degeneracy thresholds $\delta_{i}$. The combination
of this orthotropic structure with the type of degeneracy described
just above causes equation (\ref{eq:diffusion}) to exhibit a threshold-type
behavior, with diffusion coefficients that may vanish identically
on regions of nonempty interior whenever the corresponding $\delta_{i}$
are positive, thereby resulting in a loss of the parabolic structure
on sets of possibly positive measure.\foreignlanguage{american}{}\\
$\hspace*{1em}$\foreignlanguage{american}{The orthotropic character
of equation (\ref{eq:diffusion}) moreover places it within the broader
class of anisotropic problems, which were introduced in the 1980s
by Giaquinta \cite{Giaq} and Marcellini \cite{Marc} and have since
been extensively studied under many different aspects; see, for instance,
}\cite{AmbCupMas,BiaCupMas2,BB1,BouBra,BouBraLeo,Cianchi,CiHeSk,CSV,CuMaMa1,CuMaMa2,CuMaMa3,CuMaMa4,FeoPasPos,FusSbo1,FusSbo2,GriRuss2,Russo2,Strof}
in the elliptic setting and \cite{BDM,CiaGuaVes,CHSS,CiaHenSkr,CiaMosVes,CiaVeVe,DeFil,DegTed1,DegTed2,TerTer}
in the parabolic framework.\foreignlanguage{american}{ For a brief
overview of the state of the art, we refer the interested reader in
particular to Section 1 of \cite{CiaVeVe}. In that paper, the authors
establish several qualitative and quantitative properties of weak
solutions to doubly nonlinear anisotropic evolution equations, whose
prototype is} 
\begin{equation}
\partial_{t}(\vert u\vert^{\alpha-1}u)-\sum_{i=1}^{N}\,\partial_{i}\,(\vert\partial_{i}u\vert^{p_{i}-2}\,\partial_{i}u)\,=\,0\,,\label{eq:doublyNL}
\end{equation}
with $\alpha>0$ and $1<p_{1}\leq\cdots\leq p_{N}$. Among other things,
they show that certain regularity properties, such as local boundedness
and the existence of semicontinuous representatives for weak solutions,
are actually embodied in suitable energy estimates rather than in
the mere class of solutions to the equation. Furthermore, they obtain
uniform-in-time contractive bounds for the spatial $L^{1}$ and $L^{\alpha+1}$
norms of weak solutions to a Cauchy problem associated with (\ref{eq:doublyNL}).
Subsequently, they use these contractive bounds to derive a global
boundedness result for such solutions. However, their framework does
not cover our widely degenerate operator, not even in the case $p_{1}=\cdots=p_{N}$,
as already observed in \cite{AmbCiaCup}. There, the authors have
recently extended, in the case $\alpha=1$, the results on local boundedness
and the existence of semicontinuous representatives established in
\cite{CiaVeVe} to weak solutions of anisotropic evolution equations
that are markedly more degenerate than (\ref{eq:doublyNL})\foreignlanguage{american}{,
although they are not doubly nonlinear.}\\
$\hspace*{1em}$In this paper, we continue the research conducted
in the previous works\foreignlanguage{american}{ \cite{AmbCia,AmbCiaCup,CiaVeVe}
}by establishing local boundedness results for weak solutions to (\ref{eq:diffusion}),
together with quantitative bounds for their essential supremum and
infimum (see Section \ref{sec:Local-boundedness}). Moreover, in the
spirit of \cite{CiaVeVe}, we prove contractive estimates for weak
solutions to a Cauchy problem associated with equation (\ref{eq:diffusion}),
and then use these estimates to obtain a global boundedness result
for such solutions (see Section \ref{sec:Cauchy} below). In order
to emphasize the novelty of our results, we stress once again that
the equation (\ref{eq:diffusion}) considered here is substantially
more degenerate than those studied in \cite{CiaVeVe} and, unlike
the equations considered in \cite{AmbCia} and \cite{AmbCiaCup},
is doubly nonlinear whenever $\alpha\neq1$.\\
$\hspace*{1em}$\foreignlanguage{american}{Another main novelty of
the present paper, compared with the above-mentioned works, is that
the equations considered here also allow for a nontrivial source term
$f$ on the right-hand side, whose effect on the estimates derived
in this paper is explicitly quantified. In particular, when $f\equiv0$
and $\delta_{i}=0$ for all $i\in\{1,\ldots,N\}$, we essentially
recover the same local boundedness estimates and contractive bounds
obtained in \cite{CiaVeVe} in the special case $p_{1}=\cdots=p_{N}$.
Thus, our results are fully consistent with those established in \cite{CiaVeVe},
while extending their conclusions to the present, more degenerate
setting.}
\selectlanguage{american}%

\subsection{Plan of the paper}

\selectlanguage{british}%
$\hspace*{1em}$In this subsection, we outline the main issues addressed
throughout the paper, together with the principal steps of our investigation.
Concerning weak solutions to (\ref{eq:diffusion}), we carry out a
detailed analysis of the following aspects:\vspace{-4mm}
\begin{itemize}
\begin{singlespace}
\item mollified weak formulation of the notion of solution via exponential
time mollification, under the assumption that \foreignlanguage{american}{$f\in L^{\beta}(0,T;L_{\mathrm{loc}}^{\beta}(\Omega))$
for some $\beta>1$};\vspace{-2mm}
\item energy estimates for solutions under appropriate integrability assumptions
on $f$;\vspace{-2mm}
\item local boundedness for elements of the whole energy class defined at
the end of Section \ref{sec:energy}.
\end{singlespace}
\end{itemize}
\begin{singlespace}
\noindent Regarding the weak solutions to the Cauchy problem associated
with equation (\ref{eq:diffusion}), we study instead the following
properties:\vspace{-4mm}
\end{singlespace}
\begin{itemize}
\begin{singlespace}
\item uniform-in-time boundedness of the $L^{1}(\mathbb{R}^{N})$ and $L^{\alpha+1}(\mathbb{R}^{N})$
norms, under suitable integrability assumptions on the source term
$f$;\vspace{-2mm}
\item global boundedness of solutions in the case $p>\frac{N(\alpha\,+\,1)}{N\,+\,\alpha\,+\,1}\,$.
\end{singlespace}
\end{itemize}
\begin{singlespace}
\noindent In what follows, we introduce each one of these aspects.
After presenting the definition of a weak solution to (\ref{eq:diffusion}),
in Subsection \ref{subsec:auxilia} we define a time mollification
that dispenses with the usual Steklov averaging technique. We subsequently
employ this type of mollification to derive a mollified weak formulation,
which is in turn used to show that weak solutions to (\ref{eq:diffusion})
satisfy a general energy estimate under suitable integrability assumptions
on $f$.\vspace{-2mm}
\end{singlespace}
\begin{singlespace}

\subsection*{Energy estimates and local boundedness}
\end{singlespace}

\begin{singlespace}
\noindent $\hspace*{1em}$In this paper, we follow an approach dating
back to De Giorgi (see \cite{DeGio}) and subsequently widely developed
in the literature: \foreignlanguage{american}{assuming that 
\begin{align}
{\normalcolor f\in L^{\sigma}(\Omega_{T})\,\,\,\,\,\,\mathrm{for\,\,\,some}\,\,\,\sigma\,>\,\frac{N+p}{p}\,,}\label{assumpt:f}
\end{align}
}
\end{singlespace}

\selectlanguage{american}%
\noindent we show that the local boundedness of weak solutions to
(\ref{eq:diffusion}) is actually encoded in suitable energy estimates\foreignlanguage{british}{,
rather than in the mere class of solutions to the equation. More precisely,
at the end of Section \ref{sec:energy} we define the class of functions
$\mathcal{A}(\alpha,p,\{\delta_{i}\},\Omega_{T},f,\mathcal{C})$,
which we believe to be of fundamental importance in the study of the
local behavior of weak solutions. The fact that weak solutions }to
(\ref{eq:diffusion})\foreignlanguage{british}{ belong to the aforementioned
class under assumption (\ref{assumpt:f}), as explicitly observed
in Remark \ref{def:oss-impor}, allows us to formulate results of
broader scope and applicability within the regularity theory for widely
degenerate and anisotropic operators. More generally, we prove that,
under certain conditions involving $p$, $\alpha$ and $N$, the functions
in the energy class $\mathcal{A}(\alpha,p,\{\delta_{i}\},\Omega_{T},f,\mathcal{C})$
are locally essentially bounded in $\Omega_{T}$. Since weak solutions
to} (\ref{eq:diffusion}) \foreignlanguage{british}{fall within the
above class under assumption (\ref{assumpt:f}), this result applies
directly to them, yielding their local boundedness.}\\
\foreignlanguage{british}{$\hspace*{1em}$Regarding the conditions
on $p$, $\alpha$ and $N$ mentioned above, }in proving the local
boundedness results we need to distinguish between the cases $p>\frac{N(\alpha+1)}{N+\alpha+1}$
and $p\leq\frac{N(\alpha+1)}{N+\alpha+1}\,$. In the latter case,
we require the extra integrability condition 
\begin{align}
u\in L_{\mathrm{loc}}^{m}(\Omega_{T})\,\,\,\,\,\,\mathrm{for\,\,\,some}\,\,\,m\,>\,\frac{N}{p}(\alpha+1-p)\,,\label{extra_integrability}
\end{align}
similarly to subcritical $p$-Laplacian equations (see, for instance,
\cite{DiBene}).\foreignlanguage{british}{ We note that in the case
$\alpha=1$, which corresponds to the usual anisotropic equations,
the two ranges for $p$ are identical to those appearing in \cite{AmbCiaCup}
and \cite{Yu-Lian} when }$p_{1}=\cdots=p_{N}$\foreignlanguage{british}{,
and that the extra integrability condition (\ref{extra_integrability})
coincides with the corresponding integrability assumption in \cite{AmbCiaCup}
and \cite{Yu-Lian} under the same specialization of the exponents
$p_{i}$.}\\
\foreignlanguage{british}{$\hspace*{1em}$Finally, concerning the
integrability assumption on $f$, we point out that condition (\ref{assumpt:f})
is weaker than the corresponding assumption in \cite{SiVe}, where
the integrability exponent chosen for the right-hand side $f$ is
$\sigma p'$, with $\sigma$ as in (\ref{assumpt:f}) and $p'$ denoting
the conjugate exponent of $p$. Similar assumptions are also considered
in \cite{BoeDiVe} and \cite{NaPeVe}, where local boundedness results
are obtained as an intermediate step in the proofs of the existence
of solutions. In the present work, we exploit the higher local integrability
of the solution, ensured by a parabolic Sobolev embedding, to perform
a modified De Giorgi-type iteration for which assumption (\ref{assumpt:f})
turns out to be sufficient.}\\
\foreignlanguage{british}{\smallskip{}
}
\selectlanguage{british}%
\begin{singlespace}

\subsection*{Contractive estimates and global boundedness}
\end{singlespace}

\begin{singlespace}
\noindent $\hspace*{1em}$In Section \ref{sec:Cauchy}, we consider
the following Cauchy problem\foreignlanguage{american}{
\begin{align}
\left\{ \begin{array}{ll}
\partial_{t}(\vert u\vert^{\alpha-1}u)-\sum_{i=1}^{N}\partial_{i}\left[a_{i}(x,t)\,(|\partial_{i}u|-\delta_{i})_{+}^{p-1}\frac{\partial_{i}u}{\vert\partial_{i}u\vert}\right]=f, & \quad\text{in }S_{T}:=\mathbb{R}^{N}\times(0,T),\\[5pt]
u(x,0)=u_{0}(x), & \quad x\in\mathbb{R}^{N}.
\end{array}\right.\label{prob:CauPro}
\end{align}
}For a local weak solution $u\in L^{p}(0,T;W_{\mathrm{loc}}^{1,p}(\mathbb{R}^{N}))\cap C^{0}([0,T];L_{\mathrm{loc}}^{\alpha+1}(\mathbb{R}^{N}))$
of (\ref{prob:CauPro}), we establish contractive-type estimates that
control, uniformly in time, the $L^{1}(\mathbb{R}^{N})$ and $L^{\alpha+1}(\mathbb{R}^{N})$
norms of $u(\cdot,t)$, under suitable integrability assumptions on
the source term $f$ and the initial datum $u_{0}$. When $f\not\equiv0$,
these estimates account for the contribution of the source term to
the control of the aforementioned norms. In the case $f\equiv0$,
the contribution of the source term vanishes and the resulting estimates
reduce to genuine contractive bounds, showing that the mass is non-increasing
when $\alpha\in(0,1)$, i.e.,
\[
\Vert u(\cdot,t)\Vert_{L^{1}(\mathbb{R}^{N})}\,\leq\,\Vert u_{0}\Vert_{L^{1}(\mathbb{R}^{N})}\,\,\,\,\,\,\,\,\mathrm{for\,\,all}\,\,t\in[0,T).
\]
This result is complemented by a corresponding estimate for the $L^{\alpha+1}$-norm
in the full range $\alpha>0$. As anticipated earlier, when $f\equiv0$
our contractive estimates essentially match those obtained in \cite{CiaVeVe}
in the special case $p_{1}=\cdots=p_{N}$.

\noindent $\hspace*{1em}$Finally, for $p>\frac{N(\alpha+1)}{N+\alpha+1}\,$,
we establish a global boundedness result for local weak solutions
of (\ref{prob:CauPro}), together with quantitative bounds for their
essential supremum and infimum (see Theorem \ref{thm:limitatezza-globale}
below). In the case $f\equiv0$, these latter bounds agree with those
obtained from \cite[Theorem 8.2 (1)]{CiaVeVe} by setting all the
exponents $p_{i}$ appearing there equal to $p$.\\
\smallskip{}

\end{singlespace}
\selectlanguage{american}%

\subsection{Structure of the paper}

\selectlanguage{british}%
$\hspace*{1em}$\foreignlanguage{american}{The paper is organized
as follows. In Section \ref{sec:setting}, we collect the preliminary
material, including standard notation, the definition of weak solutions
to (\ref{eq:diffusion}), the exponential mollification and some useful
lemmas. We also show that weak solutions satisfy a mollified weak
formulation, which is more convenient for our purposes. In Section
\ref{sec:energy}, we derive the main energy estimates and, accordingly,
define the functional class }$\mathcal{A}(\alpha,p,\{\delta_{i}\},\Omega_{T},f,\mathcal{C})$.\foreignlanguage{american}{
Next, in Section \ref{sec:Local-boundedness}, we study the local
boundedness of functions belonging to the class }$\mathcal{A}(\alpha,p,\{\delta_{i}\},\Omega_{T},f,\mathcal{C})$\foreignlanguage{american}{
under assumption (\ref{assumpt:f}). Subsequently, in Section \ref{sec:Cauchy},
we establish contractive bounds for the spatial $L^{1}$ and $L^{\alpha+1}$
norms of local weak solutions to the Cauchy problem (\ref{prob:CauPro}),
and use these bounds to derive precise $L^{\infty}$ estimates for
such solutions.\textcolor{red}{}}\\
$\hspace*{1em}$\foreignlanguage{american}{In Sections \ref{sec:setting}--\ref{sec:Local-boundedness},
we work with weak solutions that are continuous in time as maps from
$[0,T]$ into $L^{\alpha+1}(\Omega)$ (see Definition \ref{def:weaksol}).
To establish local regularity results, it would be sufficient to consider
solutions that are continuous in time as maps into $L_{\mathrm{loc}}^{\alpha+1}(\Omega)$
instead. At the end of the paper, we have included an appendix where
we show that, in some relevant cases, this weaker time-continuity
property can in fact be obtained if the solution and the source term
$f$ belong to Lebesgue spaces with Hölder-conjugate exponents.}
\selectlanguage{american}%

\section{Notation and preliminaries\label{sec:setting}}

\selectlanguage{british}%
$\hspace*{1em}$In this paper we shall denote by $C$ or $c$ a general
positive constant that may vary on different occasions, even within
the same line of estimates. Relevant dependencies on parameters and
special constants will be suitably emphasized using parentheses or
subscripts. The norm we use on $\mathbb{R}^{k}$, $k\in\mathbb{N}$,
will be the standard Euclidean one and it will be denoted by $\left|\,\cdot\,\right|$.
In particular, for the vectors $\xi,\eta\in\mathbb{R}^{k}$, we write
$\langle\xi,\eta\rangle$ for the usual inner product and $\left|\xi\right|:=\langle\xi,\xi\rangle^{\frac{1}{2}}$
for the corresponding Euclidean norm.\foreignlanguage{american}{}\\
$\hspace*{1em}$\foreignlanguage{american}{For points in space-time,
we will frequently use abbreviations like $z=(x,t)$ or $z_{o}=(x_{o},t_{o})$,
for spatial variables $x$, $x_{o}\in\mathbb{R}^{N}$ and times $t$,
$t_{o}\in\mathbb{R}$. We also denote by $B_{r}(x_{o})=\left\{ x\in\mathbb{R}^{N}:\left|x-x_{o}\right|<r\right\} $
the $N$-dimensional open ball with radius $r>0$ and center $x_{o}\in\mathbb{R}^{N}$.
Furthermore, we use the notation 
\[
Q_{r}(z_{o}):=B_{r}(x_{o})\times(t_{o}-r^{p},t_{o}),\,\,\,\,\,\,\,\,\,z_{o}=(x_{o},t_{o})\in\mathbb{R}^{N}\times\mathbb{R},\,\,r>0,
\]
for the backward parabolic cylinder with vertex $(x_{o},t_{o})$ and
width $r$.}\\
\foreignlanguage{english}{$\hspace*{1em}$}For a real-valued function
$v=v(x,t)$, with $x=(x_{1},\ldots,x_{N})\in\mathbb{R}^{N}$ and \foreignlanguage{american}{$t\in\mathbb{R}$},
we write $\partial_{i}v$, $i\in\{1,\ldots,N\}$, for the partial
derivative of $v$ with respect to the spatial variable $x_{i}$,
and $\partial_{t}v$ for its partial derivative with respect to time.
Moreover, we use the notation $\nabla v:=(\partial_{1}v,\ldots,\partial_{N}v)$
for the spatial gradient of $v$.\\
\foreignlanguage{english}{$\hspace*{1em}$If $E\subseteq\mathbb{R}^{k}$
is a Lebesgue-measurable set, we will denote by $\chi_{E}$ its characteristic
function and by $\vert E\vert$ its $k$-dimensional Lebesgue measure.}\\
$\hspace*{1em}$\foreignlanguage{american}{Throughout the paper, for
$\gamma>0$ and $a\in\mathbb{R}$ we adopt the convention that $\vert a\vert^{\gamma-1}a=0$
if $a=0,$ even though in this case technically the first factor is
ill-defined if $\gamma<1$. Furthermore, for notational convenience,
we will frequently write $a^{\gamma}=\vert a\vert^{\gamma-1}a$, even
when $a<0$, following a convention that is by now standard in the
literature.}\\
\\
$\hspace*{1em}$Let $\mathcal{F}:\Omega_{T}\times\mathbb{R}^{N}\to\mathbb{R}$
be the function defined by
\begin{equation}
\mathcal{F}(x,t,\xi):=\sum_{i=1}^{N}\,\frac{a_{i}(x,t)}{p}\,(\vert\xi_{i}\vert-\delta_{i})_{+}^{p}\,.\label{eq:def:F}
\end{equation}
For notational convenien\foreignlanguage{english}{ce, we set 
\begin{equation}
A(x,t,\eta):=\,\nabla_{\xi}\,\mathcal{F}(x,t,\eta)\,,\,\,\,\,\,\,\,\,\eta\in\mathbb{R}^{N},\label{eq:vector_field}
\end{equation}
so that equation (\ref{eq:diffusion}) can be rewritten as 
\[
\partial_{t}(\vert u\vert^{\alpha-1}u)-\mathrm{div}\,[A(x,t,\nabla u)]=f\,,
\]
where $\mathrm{div}$ denotes the spatial divergence operator. Assuming,
as a basic integrability condition, that $f\in L_{\mathrm{loc}}^{1}(\Omega_{T})$,
we define a weak solution of (\ref{eq:diffusion}) as follows.}
\selectlanguage{american}%
\begin{defn}
\label{def:weaksol} A function $u\in L^{p}(0,T;W^{1,p}(\Omega))\cap C^{0}([0,T];L^{\alpha+1}(\Omega))$
is a weak solution to (\ref{eq:diffusion}) if 
\begin{align}
 & \iint_{\Omega_{T}}\left(\langle A(x,t,\nabla u),\nabla\varphi\rangle-\vert u\vert^{\alpha-1}u\,\partial_{t}\varphi\right)dx\,dt\,=\,\iint_{\Omega_{T}}f\varphi\,dx\,dt\label{eq:weak_form}
\end{align}
for all $\varphi\in C_{0}^{\infty}(\Omega_{T})$.
\end{defn}

\selectlanguage{british}%
\noindent \begin{brem}\foreignlanguage{american}{We observe that,
in order to prove local regularity results, it would be sufficient
in Definition \ref{def:weaksol} to require only corresponding \textit{local}
integrability properties for the weak solutions. However, since this
is a mere technicality, we directly assume global integrability. Likewise,
in Definition \ref{def:weaksol} one could dispense with explicitly
requiring the time-continuity of a weak solution $u$ as a map from
$[0,T]$ into $L^{\alpha+1}(\Omega)$, provided that $u$ and the
source term $f$ satisfy integrability assumptions of sufficiently
high order, which in turn imply precisely that $u\in C^{0}([0,T];L_{\mathrm{loc}}^{\alpha+1}(\Omega))$.
In the present setting, the precise integrability assumptions to be
imposed on $u$ and $f$ depend on the relation between the parameters
$N$, $\alpha$ and $p$, and we refer the interested reader to the
\hyperref[sec:app:time-cont]{Appendix} for the details.\end{brem}}
\selectlanguage{american}%

\subsection{Auxiliary tools\label{subsec:auxilia}}

\selectlanguage{british}%
$\hspace*{1em}$\foreignlanguage{american}{We now recall some elementary
lemmas that will be used later, and start by defining a mollification
in time as in \cite{KiLi} (see also \cite{BoeDuMa}). For $T>0$,
$t\in[0,T]$, $h\in(0,T)$ and $v\in L^{1}(\Omega_{T})$, we set 
\begin{align}
v_{h}(x,t):=\,\frac{1}{h}\int_{0}^{t}e^{\frac{s-t}{h}}\,v(x,s)\,ds\,.\label{def:moll}
\end{align}
Moreover, we define the reversed analogue by 
\begin{align*}
v_{\bar{h}}(x,t):=\,\frac{1}{h}\int_{t}^{T}e^{\frac{t-s}{h}}\,v(x,s)\,ds\,.
\end{align*}
For details regarding the properties of the exponential mollification,
we refer to \cite[Lemma 2.2]{BoeDuMa}, \cite[Lemma 2.2]{KiLi} and
\cite[Lemma 2.9]{St}. The properties of the mollification that we
will use have been gathered for convenience into the following lemma.}
\selectlanguage{american}%
\begin{lem}
\label{lem:expmolproperties}Assume that $v\in L^{1}(\Omega_{T})$
and let $p\in[1,\infty)$. Then the mollification $v_{h}$ defined
in $\eqref{def:moll}$ has the following properties: 
\begin{enumerate}
\item[\foreignlanguage{american}{$\mathrm{(i)}$}] if $v\in L^{p}(\Omega_{T})$, then $v_{h}\in L^{p}(\Omega_{T})$,
\[
\Vert v_{h}\Vert_{L^{p}(\Omega_{T})}\,\leq\,\Vert v\Vert_{L^{p}(\Omega_{T})}
\]
and $v_{h}\to v$ in $L^{p}(\Omega_{T})$. A similar estimate also
holds with $v_{\bar{h}}$ on the left-hand side. 
\item[\foreignlanguage{american}{$\mathrm{(ii)}$}] In the above situation, $v_{h}$ has a weak time derivative $\partial_{t}v_{h}$
on $\Omega_{T}$ given by 
\begin{align*}
\partial_{t}v_{h}\,=\,\frac{v-v_{h}}{h}\,,
\end{align*}
whereas for $v_{\bar{h}}$ we have 
\begin{align*}
\partial_{t}v_{\bar{h}}\,=\,\frac{v_{\bar{h}}-v}{h}\,.
\end{align*}
\item[\foreignlanguage{american}{$\mathrm{(iii)}$}] If $v$ has a weak partial derivative in space, then so do $v_{h}$
and $v_{\bar{h}}$ and 
\begin{align*}
\partial_{j}(v_{h})=(\partial_{j}v)_{h}\,,\hspace{5mm}\,\,\,\partial_{j}(v_{\bar{h}})=(\partial_{j}v)_{\bar{h}}\,.
\end{align*}
\item[\foreignlanguage{american}{$\mathrm{(iv)}$}] If $v\in L^{p}(0,T;L^{p}(\Omega))$, then $v_{h},v_{\bar{h}}\in C^{0}([0,T];L^{p}(\Omega))$.
\end{enumerate}
\end{lem}

\noindent \begin{brem}\label{enu:expmol_local_integ} The exponential
time mollification \eqref{def:moll} is well defined also if the function
$v$ is only in $L^{1}(0,T;L_{\mathrm{loc}}^{1}(\Omega))$ and the
properties $\mathrm{(i)}$ and $\mathrm{(iv)}$ in the previous lemma
evidently have corresponding versions for functions that are only
locally integrable in space.\end{brem}

\selectlanguage{british}%
\noindent $\hspace*{1em}$Let $\alpha\in(0,\infty)$. For $v,w\in\mathbb{R}$
we define the quantity

\selectlanguage{american}%
\begin{align*}
\mathfrak{b}_{\alpha}[v,w]:=\,\frac{\alpha}{\alpha+1}\, & (|v|^{\alpha+1}-|w|^{\alpha+1})-w\,(|v|^{\alpha-1}v-|w|^{\alpha-1}w)\,.
\end{align*}

\noindent The next three lemmas provide us with some useful estimates
for the quantity $\mathfrak{b}_{\alpha}[v,w]$. Note that some estimates
hold for all $\alpha>0$, whereas others are valid only in one of
the ranges $\alpha\in(0,1)$ and $\alpha\geq1$.
\begin{lem}
\label{lem:bdry_term_estimates_all_alpha} Let $\alpha>0$. Then there
exists a positive constant $c$, depending only on $\alpha$, such
that for all $v,w\in\mathbb{R}$,
\begin{align}
\frac{1}{c}\,\big||w|^{\frac{\alpha-1}{2}}w-|v|^{\frac{\alpha-1}{2}}v\big|^{2}\,\leq\,\mathfrak{b}_{\alpha}[v,w]\,\leq\,c\,\big||w|^{\frac{\alpha-1}{2}}w-|v|^{\frac{\alpha-1}{2}}v\big|^{2}.\label{est:b-all-alpha}
\end{align}
\end{lem}

\noindent \begin{proof}[\bfseries{Proof}]If $\alpha\in(0,1)$ the
claim follows from \cite[Lemma 2.3 (i)]{BoeDuKoSc} with the choice
$m=1/\alpha$. If $\alpha\geq1$ the claim follows again from \cite[Lemma 2.3 (i)]{BoeDuKoSc}
with the choice $m=\alpha$.\end{proof}
\begin{lem}
\label{lem:bdry_term_estimates_alpha_small} Let $\alpha\in(0,1)$.
Then there exists a positive constant $c$, depending only on $\alpha$,
such that for all $v,w\in\mathbb{R}$, 
\begin{enumerate}
\item[\foreignlanguage{american}{$\mathrm{(i)}$}] $\frac{1}{c}\,|w-v|^{2}\leq(|w|+|v|)^{1-\alpha}\,\mathfrak{b}_{\alpha}[v,w]\leq c\,|w-v|^{2}$
;
\item[\foreignlanguage{american}{$\mathrm{(ii)}$}]  $\mathfrak{b}_{\alpha}[v,w]\leq c\,|v-w|^{1+\alpha}$.
\end{enumerate}
\end{lem}

\noindent \begin{proof}[\bfseries{Proof}]The estimates follow from
properties (ii) and (iii) in \cite[Lemma 2.3]{BoeDuKoSc} with the
choices $m=1/\alpha$, $u=|v|^{\alpha-1}\,v$ and $a=|w|^{\alpha-1}\,w$.\end{proof}

\noindent The following estimates have been proved in \cite[Lemma 3.6]{Vestb}.
\begin{lem}
\label{lem:bdry_term_estimates_alpha_large} Let $\alpha\geq1$. Then
there exists a positive constant $c$, depending only on $\alpha$,
such that for all $v,w\in\mathbb{R}$,
\begin{enumerate}
\item[\foreignlanguage{american}{$\mathrm{(i)}$}]  $\mathfrak{b}_{\alpha}[v,w]\leq c\,\big||v|^{\alpha-1}\,v-|w|^{\alpha-1}\,w\big|^{\frac{\alpha+1}{\alpha}}$;
\item[\foreignlanguage{american}{$\mathrm{(ii)}$}]  $|v-w|^{\alpha+1}\leq\,c\,\mathfrak{b}_{\alpha}[v,w]$.
\end{enumerate}
\end{lem}

\selectlanguage{british}%
\noindent $\hspace*{1em}$\foreignlanguage{american}{The following
inequalities will be used frequently in our calculations; we refer
to \cite[Lemma 3.7]{CiaVeVe} for a proof.}
\selectlanguage{american}%
\begin{lem}
\label{lem:elementary_real} Let $\gamma\geq1$. For all $a,b\in\mathbb{R}$
we have 
\begin{align}
\frac{1}{c}\,|a-b|^{\gamma}\,\leq\,\big||a|^{\gamma-1}a-|b|^{\gamma-1}b\big|\,\leq\,c\,(|a|^{\gamma-1}+|b|^{\gamma-1})|a-b|\label{est:exponent_inside}
\end{align}
for a positive constant $c=c(\gamma)$.
\end{lem}

\selectlanguage{british}%
\noindent $\hspace*{1em}$\foreignlanguage{american}{The next result
is a direct consequence of Lemma 3.9 in \cite{Vestb}.}
\selectlanguage{american}%
\begin{lem}
\label{lem:equivalent_time_cont} Let $\alpha>0$ and $u\in L_{\mathrm{loc}}^{1}(\Omega_{T})$.
Then $u\in C^{0}([0,T];L_{\mathrm{loc}}^{\alpha+1}(\Omega))$ if and
only if $|u|^{\alpha-1}u\in C^{0}([0,T];L_{\mathrm{loc}}^{\frac{\alpha+1}{\alpha}}(\Omega))$.
\end{lem}

\selectlanguage{british}%
\noindent $\hspace*{1em}$For further needs, we now recall the following
standard Sobolev embedding, which can be obtained from Proposition
3.1 in Chapter I of \cite{DiBene} by taking $m=\alpha+1$. For an
alternative proof of this type of embeddings, see also \cite[Lemma 1]{Vestthesis}.
\selectlanguage{american}%
\begin{lem}
\label{lem:parabolic-sobolev}\foreignlanguage{british}{ Let }$\Omega\subset\mathbb{R}^{N}$
be a bounded open set, $p\in(1,\infty)$ and $\alpha\in(0,\infty)$.
Then the space $L^{p}(0,T;W_{0}^{1,p}(\Omega))\cap L^{\infty}(0,T;L^{\alpha+1}(\Omega))$
is contained in $L^{p_{\alpha+1}}(\Omega_{T})$, where $p_{\alpha+1}:=p(1+\tfrac{\alpha+1}{N})$.\\
Moreover, there exists a constant $c$ depending only on $N$, $\alpha$,
$p$ such that, for every $v\in L^{p}(0,T;W_{0}^{1,p}(\Omega))\cap L^{\infty}(0,T;L^{\alpha+1}(\Omega))$,
\[
\iint_{\Omega_{T}}\vert v\vert^{p_{\alpha+1}}\,dx\,dt\,\leq\,c\left[\underset{\tau\,\in\,(0,T)}{\mathrm{ess}\,\sup}\int_{\Omega}|v(x,\tau)|^{\alpha+1}\,dx\right]^{\frac{p}{N}}\iint_{\Omega_{T}}|\nabla v|^{p}\,dx\,dt\,.
\]
\end{lem}

\selectlanguage{british}%
\noindent $\hspace*{1em}$The next lemma plays a key role in the De
Giorgi-type iterations; we refer to \foreignlanguage{american}{\cite[Lemma 7.1]{Gi}
for a proof.}
\selectlanguage{american}%
\begin{lem}
\label{lem:fastconvg} \foreignlanguage{british}{Let $\delta>0$ and
let $\{Y_{j}\}_{j\,\in\,\mathbb{N}_{0}}$ be a sequence of non-negative
real numbers, satisfying the recursive inequalities 
\[
Y_{j+1}\,\leq\,C\,b^{j}\,Y_{j}^{1+\delta}
\]
where $C>0$ and $b>1$. If $Y_{0}\leq C^{-\,\frac{1}{\delta}}\,b^{-\,\frac{1}{\delta^{2}}}$,
then 
\[
\lim_{j\to\infty}Y_{j}=0\,.
\]
}
\end{lem}

\selectlanguage{british}%
\noindent $\hspace*{1em}$\foreignlanguage{american}{We conclude this
subsection with the following result, which will be useful in establishing
the local boundedness of weak solutions to (\ref{eq:diffusion}).}
\selectlanguage{american}%
\begin{lem}
\label{lem:basic_delta} Let $N\geq2$, $p\in[1,\infty)$ and $\delta_{i}\geq0$
for all $i\in\{1,\ldots,N\}$. Define
\[
\delta_{\max}:=\,\max\,\{\delta_{1},...,\delta_{N}\}.
\]
Then there exists a constant $c=c(N,p)>0$ such that, for every $\xi\in\mathbb{R}^{N}$,
\begin{align}
|\xi|^{p}\,\leq\,c\sum_{i=1}^{N}(|\xi_{i}|-\delta_{i})_{+}^{p}\,+\,c\,\delta_{\max}^{p}\,.\label{est:basic_delta}
\end{align}
\end{lem}

\noindent \begin{proof}[\bfseries{Proof}]Note that for all $i\in\{1,\ldots,N\}$,
\begin{align*}
|\xi_{i}|^{p}\,\leq\,[(|\xi_{i}|-\delta_{i})_{+}+\delta_{i}]^{p}.
\end{align*}
Summing over $i=1,\ldots,N$ and using the equivalence of norms in
$\mathbb{R}^{N}$, we have for some constant $c_{1}=c_{1}(N,p)>0$
that\begin{align*}
c_{1}\,|\xi|^{p}\,&\leq\,\sum_{i=1}^{N}|\xi_{i}|^{p}\,\leq\,\sum_{i=1}^{N}\,[(|\xi_{i}|-\delta_{i})_{+}+\delta_{i}]^{p}\,\leq\,2^{p-1}\sum_{i=1}^{N}\,[(|\xi_{i}|-\delta_{i})_{+}^{p}+\delta_{i}^{p}]\\
&\leq\,2^{p-1}\sum_{i=1}^{N}(|\xi_{i}|-\delta_{i})_{+}^{p}\,+\,2^{p-1}N\delta_{\max}^{p}\,,
\end{align*}from which the assertion follows.\end{proof}

\subsection{Mollified weak formulation}

\selectlanguage{british}%
$\hspace*{1em}$In this subsection, we show that weak solutions to
(\ref{eq:diffusion}) satisfy a mollified weak formulation. The first
part of the proof follows that of Lemma $4.3$ in \cite{CiaVeVe},
where, however, $f=0$. We nevertheless include it here for the reader's
convenience.
\selectlanguage{american}%
\begin{lem}
\label{lem:mollified} Assume that \eqref{eq:coeff_limit} holds,
let $f\in L^{\beta}(0,T;L_{\mathrm{loc}}^{\beta}(\Omega))$ for some
$\beta>1$, and let $u$ be a weak solution to $(\ref{eq:diffusion})$
in the sense of Definition \ref{def:weaksol}. Then we have 
\begin{align}
\iint_{\Omega_{T}}\left(\langle[A(x,t,\nabla u)]_{h},\nabla\phi\rangle+\,\partial_{t}[|u|^{\alpha-1}u]_{h}\,\phi\right)dx\,dt\,-\int_{\Omega}\tilde{u}_{\alpha}\,\phi_{\bar{h}}(x,0)\,dx\,=\iint_{\Omega_{T}}f_{h}\,\phi\,dx\,dt\label{weak:expmollified}
\end{align}
for all $h\in(0,T)$ and all $\phi\in C^{\infty}(\Omega\times[0,T])$
with support contained in $K\times[0,\tau]$, where $K\subset\Omega$
is compact and $\tau\in(0,T)$. Here $A:\Omega_{T}\times\mathbb{R}^{N}\to\mathbb{R}^{N}$
is the vector field defined in $(\ref{eq:vector_field})$, while $\tilde{u}_{\alpha}(x,0)$
denotes the value at time zero of the continuous representative of
$|u|^{\alpha-1}u$ as a map $[0,T]\to L^{\frac{\alpha+1}{\alpha}}(K)$.
\end{lem}

\noindent \begin{proof}[\bfseries{Proof}]Let $\phi$ be as in the
statement of the lemma. For $\varepsilon\in(0,T)$, consider the piecewise
smooth function 
\begin{align*}
\eta_{\varepsilon}(t):=\begin{cases}
\,\frac{t}{\varepsilon} & \mathrm{if}\,\,t\in[0,\varepsilon],\\
\,1 & \mathrm{if}\,\,t\in(\varepsilon,T],
\end{cases}
\end{align*}
and use \eqref{eq:weak_form} with the test function $\varphi=\eta_{\varepsilon}\,\phi_{\bar{h}}$,
$h\in(0,T)$. Taking the limit as $\varepsilon\to0$ and using Fubini's
theorem together with \eqref{eq:coeff_limit}, we see that the integral
of the elliptic term converges to the integral of $\langle[A(x,\cdot,\nabla u)]_{h},\nabla\phi\rangle$.
Now, note that
\begin{align*}
\iint_{\Omega_{T}}\vert u\vert^{\alpha-1}u\,\partial_{t}(\eta_{\varepsilon}\,\phi_{\bar{h}})\,dx\,dt\,=\,\iint_{\Omega_{T}}\vert u\vert^{\alpha-1}u\,\eta_{\varepsilon}\,\frac{\phi_{\bar{h}}-\phi}{h}\,dx\,dt\,+\,\varepsilon^{-1}\int_{0}^{\varepsilon}\int_{\Omega}\vert u\vert^{\alpha-1}u\,\phi_{\bar{h}}\,dx\,dt\,.
\end{align*}
In the first term on the right-hand side, we can pass to the limit
as $\varepsilon\to0$, use Fubini's theorem to move the mollification
over to $u$ and apply Lemma \ref{lem:expmolproperties} (ii) to obtain
the integral of $\partial_{t}[|u|^{\alpha-1}u]_{h}\,\phi$. It remains
to analyze the behavior of the last term as $\varepsilon\to0$. We
can rewrite this term as\begin{align}\label{eq:2nd_term}
\varepsilon^{-1}\int_{0}^{\varepsilon}\int_{K}\vert u\vert^{\alpha-1}u\,\phi_{\bar{h}}\,dx\,dt\,= &\,\,\,\varepsilon^{-1}\int_{0}^{\varepsilon}\int_{K}(\vert u\vert^{\alpha-1}u)(x,t)\,\phi_{\bar{h}}(x,0)\,dx\,dt\nonumber\\
&+\,\varepsilon^{-1}\int_{0}^{\varepsilon}\int_{K}(\vert u\vert^{\alpha-1}u)(x,t)\,[\phi_{\bar{h}}(x,t)-\phi_{\bar{h}}(x,0)]\,dx\,dt\,.
\end{align}The second term on the right-hand side of \foreignlanguage{british}{\eqref{eq:2nd_term}
}converges to zero as $\varepsilon\rightarrow0$, since $\phi_{\bar{h}}$
is uniformly continuous in $t$ and $\Vert\vert u\vert^{\alpha-1}u(\cdot,t)\Vert_{L^{\frac{1}{\alpha}+1}(K)}$
is bounded independently of $t$. The first term on the right-hand
side \foreignlanguage{british}{of \eqref{eq:2nd_term} }converges
to the second integral on the left-hand side of \eqref{weak:expmollified},
since $\vert u\vert^{\alpha-1}u\in C^{0}([0,T];L^{\frac{1}{\alpha}+1}(K))$
by Lemma \ref{lem:equivalent_time_cont} and $\phi_{\bar{h}}(\cdot,0)\in L^{\alpha+1}(\Omega)$.

\selectlanguage{british}%
\noindent $\hspace*{1em}$\foreignlanguage{american}{As for the integral
involving $f$, observe that
\begin{equation}
\iint_{\Omega_{T}}f\,\phi_{\bar{h}}\,(\eta_{\varepsilon}-1)\,dx\,dt\,=\int_{0}^{\varepsilon}\int_{K}f\,\phi_{\bar{h}}\,(\eta_{\varepsilon}-1)\,dx\,dt\,,\label{eq:trivial}
\end{equation}
since $\eta_{\varepsilon}=1$ on $(\varepsilon,T]$. By the integrability
assumption on $f$, the fact that $\phi_{\bar{h}}\in L^{\infty}(\Omega\times[0,T])$,
the estimate $\vert\eta_{\varepsilon}-1|\leq1$ and the pointwise
convergence $\eta_{\varepsilon}\to1$ as $\varepsilon\to0$ on $(0,T],$
we may apply the Dominated Convergence Theorem to the last integral.
We thus obtain
\[
\int_{0}^{\varepsilon}\int_{K}f\,\phi_{\bar{h}}\,(\eta_{\varepsilon}-1)\,dx\,dt\,\longrightarrow0\,\,\,\,\,\,\,\,\mathrm{as}\,\,\varepsilon\rightarrow0\,,
\]
which, together with (\ref{eq:trivial}), yields
\[
\iint_{\Omega_{T}}f\,\eta_{\varepsilon}\,\phi_{\bar{h}}\,dx\,dt\,\longrightarrow\iint_{\Omega_{T}}f\,\phi_{\bar{h}}\,dx\,dt\,=\iint_{\Omega_{T}}f_{h}\,\phi\,dx\,dt\,\,\,\,\,\,\,\,\,\mathrm{as}\,\,\varepsilon\rightarrow0\,,
\]
where the last equality is obtained by a further application of Fubini's
theorem. This completes the proof.\end{proof}}
\selectlanguage{american}%

\section{Energy estimates\label{sec:energy}}

\selectlanguage{british}%
$\hspace*{1em}$\foreignlanguage{american}{In this section, we establish
a general energy estimate that will serve to derive the classical
energy estimates in Lemma \ref{lem:energy-classical} below and will
play a key role in the subsequent sections. This estimate, in the
form needed for our purposes, is provided by the following lemma.}
\selectlanguage{american}%
\begin{lem}
\label{lem:energy-general} Let $\Omega$ be an open, possibly unbounded,
subset\foreignlanguage{british}{ of $\mathbb{R}^{N}$, with $N\geq2$,
and let $T>0$}. Let $F:\mathbb{R}\rightarrow\mathbb{R}$ be a non-decreasing,
Lipschitz and piecewise $C^{1}$ function satisfying
\begin{align}
F(s)=0\,\,\,\,\mathit{whenever}\,\,\,\,F'(s)=0\,.\label{Assunzione}
\end{align}
Let $g:\mathbb{R}\to\mathbb{R}$ be the unique function such that
\begin{align*}
F(s)=g(|s|^{\alpha-1}s)\,,
\end{align*}
and let $G:\mathbb{R}\to\mathbb{R}$ be any integral function of $g$.
Moreover, assume that \eqref{eq:coeff_limit} holds, let $u$ be a
weak solution of \eqref{eq:diffusion} in $\Omega_{T}$, and suppose
that, for some $\lambda>1$, 
\begin{equation}
F\circ u\,\in\,L^{\lambda}(0,T;L_{\mathrm{loc}}^{\lambda}(\Omega))\,\,\,\,\,\,\,\,\,\,\,\mathit{and}\,\,\,\,\,\,\,\,\,\,\,f\,\in\,L^{\lambda'}(0,T;L_{\mathrm{loc}}^{\lambda'}(\Omega))\,,\label{eq:esponenti_coniugati}
\end{equation}
where $\lambda':=\lambda/(\lambda-1)$ is the conjugate exponent of
$\lambda$. Then, for every $\eta\in C_{0}^{\infty}(\Omega;[0,\infty))$
and every $\varphi\in C^{\infty}([0,T];[0,\infty))$, we have for
all $0\leq\tau_{1}<\tau_{2}\leq T$ the estimate\begin{align}\label{eq:energy-general}
&\int_{\Omega}\eta^{p}\varphi\,G(|u|^{\alpha-1}u)(x,\tau_{2})\,dx\,+\,\frac{1}{\gamma}\,\sum_{i=1}^{N}\iint_{\Omega\times[\tau_{1},\tau_{2}]}(|\partial_{i}u|-\delta_{i})_{+}^{p}\,\eta^{p}\,\varphi\,F'(u)\,dx\,dt\nonumber\\
&\,\,\,\leq\int_{\Omega}\eta^{p}\varphi\,G(|u|^{\alpha-1}u)(x,\tau_{1})\,dx\,+\iint_{\Omega\times[\tau_{1},\tau_{2}]}\eta^{p}\,\partial_{t}\varphi\,G(|u|^{\alpha-1}u)\,dx\,dt\\
&\,\,\,\,\,\,\,\,\,\,+\iint_{\Omega\times[\tau_{1},\tau_{2}]}f\eta^{p}\varphi F(u)\,dx\,dt\,+\,\gamma\sum_{i=1}^{N}\iint_{\Omega\times[\tau_{1},\tau_{2}]}|F(u)|^{p}\,(F'(u))^{1-p}\,|\partial_{i}\eta|^{p}\,\varphi\,\chi_{\{|\partial_{i}u|\,>\,\delta_{i}\}}\,dx\,dt\,,\nonumber
\end{align}where $\gamma$ is a positive constant depending only on $p$ and
$\Lambda$.
\end{lem}

\noindent \begin{proof}[\bfseries{Proof}]Consider first $0<\tau_{1}<\tau_{2}<T$.
Let $\delta\in\left(0,\frac{\tau_{2}\,-\,\tau_{1}}{2}\right)$ and
test the mollified weak formulation \eqref{weak:expmollified} with
\[
\phi=F(u(x,t))\,\eta^{p}(x)\,\xi(t)\,,
\]
where $\eta$ is as in the statement of the lemma and $\xi(t):=\varphi(t)\,\psi(t)$,
being $\varphi$ as above while $\psi$ is the trapezoidal function
\begin{align*}
\psi(t)=\begin{cases}
\,0 & \mathrm{if}\,\,t<\tau_{1}\,,\\
\,\delta^{-1}(t-\tau_{1}) & \mathrm{if}\,\,t\in[\tau_{1},\tau_{1}+\delta]\,,\\
\,1 & \mathrm{if}\,\,t\in(\tau_{1}+\delta,\tau_{2}-\delta)\,,\\
\,1-\delta^{-1}(t-\tau_{2}+\delta) & \mathrm{if}\,\,t\in[\tau_{2}-\delta,\tau_{2}]\,,\\
\,0 & \mathrm{if}\,\,t\geq\tau_{2}\,.
\end{cases}
\end{align*}
Due to \eqref{Assunzione}, the chain rule for weak derivatives ensures
that $\phi$ is an admissible test function in the mollified weak
formulation, which takes the form\begin{align}\label{eq:form_debole}
&\iint_{\Omega_{T}}\left(\langle[A(x,t,\nabla u)]_{h},\nabla\phi\rangle+\,\partial_{t}[|u|^{\alpha-1}u]_{h}\,\phi\right)dx\,dt\nonumber\\
&\,\,\,\,\,\,\,-\int_{0}^{T}\int_{\Omega}\tilde{u}_{\alpha}(x,0)\,F(u)\,\eta^{p}(x)\,\xi(t)\,\frac{1}{h}\,e^{-\,\frac{t}{h}}\,dx\,dt\,=\iint_{\Omega_{T}}f_{h}\,\phi\,dx\,dt\,,
\end{align}where $h\in(0,T)$ and $\tilde{u}_{\alpha}(x,0)$ denotes the value
at time zero of the continuous representative of $|u|^{\alpha-1}u$
as a map $[0,T]\to L^{\frac{\alpha+1}{\alpha}}(\mathrm{supp}\,\eta)$.
Note that the last integral in \foreignlanguage{british}{\eqref{eq:form_debole}}
is finite by virtue of (\ref{eq:esponenti_coniugati}), Lemma \ref{lem:expmolproperties}
and Remark \ref{enu:expmol_local_integ}. Now, arguing as in \cite[estimate (5.5)]{CiaVeVe},
we have
\begin{equation}
\partial_{t}[|u|^{\alpha-1}u]_{h}\,\phi\,\geq\,\eta^{p}\,\xi\,\partial_{t}\left(G([|u|^{\alpha-1}u]_{h})\right).\label{eq:CiaVeVe1}
\end{equation}
Combining \foreignlanguage{british}{\eqref{eq:form_debole} }and \eqref{eq:CiaVeVe1},
and moving the time derivative to the test function, we end up with\begin{align}\label{eq:disug_molli}
&\iint_{\Omega_{T}}\left(\langle[A(x,t,\nabla u)]_{h},\nabla\phi\rangle-\,\eta^{p}\,\partial_{t}\xi\,G([|u|^{\alpha-1}u]_{h})\right)dx\,dt\nonumber\\
&\,\,\,\,\,\,\,-\iint_{\Omega_{T}}\tilde{u}_{\alpha}(x,0)\,F(u)\,\eta^{p}(x)\,\xi(t)\,\frac{1}{h}\,e^{-\,\frac{t}{h}}\,dx\,dt\,\leq\iint_{\Omega_{T}}f_{h}\,\phi\,dx\,dt\,.
\end{align}We now pass to the limit as $h\rightarrow0$. The first integral on
the last line of \foreignlanguage{british}{\eqref{eq:disug_molli}
}vanishes in this limit by the Dominated Convergence Theorem. This
can be seen by noting that $\tilde{u}_{\alpha}(\cdot,0)\,F(u)$ is
locally integrable and the term $\xi(t)\,\frac{1}{h}\,e^{-\,\frac{t}{h}}$
remains bounded independently of $h$, since $\xi=0$ for times less
than $\tau_{1}$. In the remaining terms of \foreignlanguage{british}{\eqref{eq:disug_molli}},
passing to the limit as $h\rightarrow0$ poses no difficulty and we
recover the corresponding terms without time mollification. Thus we
obtain
\begin{equation}
\iint_{\Omega_{T}}\left(\langle A(x,t,\nabla u),\nabla\phi\rangle-\,\eta^{p}\,\partial_{t}\xi\,G(|u|^{\alpha-1}u)\right)dx\,dt\,\leq\iint_{\Omega_{T}}f\,F(u)\,\eta^{p}\,\xi\,dx\,dt\,.\label{eq:post-limit}
\end{equation}
In order to estimate the integral of the elliptic term, we first note
that we can exclude the set of points where $F'\circ u$ is ill-defined.
Indeed, there is an at most countable set $S\subset\mathbb{R}$ where
$F'$ is not defined. Since $\nabla u$ vanishes almost everywhere
on each level set of $u$, it follows that $\nabla u=0$ almost everywhere
on $u^{-1}(S)$. Now observe that for all $i\in\{1,\ldots,N\},$
\begin{equation}
\vert A_{i}(x,t,\nabla u)\vert\,=\,a_{i}(x,t)\,(|\partial_{i}u|-\delta_{i})_{+}^{p-1}\,\leq\,\Lambda\,\vert\nabla u\vert^{p-1},\label{eq:structure_conditions}
\end{equation}
and hence $A(x,t,\nabla u)$ also vanishes almost everywhere on $u^{-1}(S)$.
Setting $E:=\Omega_{T}\,\backslash\,u^{-1}(S)$, we thus have
\begin{equation}
\iint_{\Omega_{T}}\langle A(x,t,\nabla u),\nabla\phi\rangle\,dx\,dt\,=\,\iint_{E}\langle A(x,t,\nabla u),\nabla\phi\rangle\,dx\,dt\,.\label{eq:integrale_ellittico}
\end{equation}
For any point of $E$ where $F'\circ u\neq0$ we can use the definition
of $A(x,t,\nabla u)$, \eqref{eq:coeff_limit} and Young's inequality
to obtain the following estimate:\begin{align*}
&\langle A(x,t,\nabla u),\nabla\phi\rangle \,=\,F'(u)\,\eta^{p}\,\xi\,\langle A(x,t,\nabla u),\nabla u\rangle\,+\,p\,F(u)\,\eta^{p-1}\,\xi\,\langle A(x,t,\nabla u),\nabla\eta\rangle\\
&\,\,\,\,\,=\,\sum_{i=1}^{N}a_{i}\,(|\partial_{i}u|-\delta_{i})_{+}^{p-1}\,\vert\partial_{i}u\vert\,F'(u)\,\eta^{p}\,\xi\,+\,p\,F(u)\,\eta^{p-1}\,\xi\,\sum_{i=1}^{N}a_{i}\,(|\partial_{i}u|-\delta_{i})_{+}^{p-1}\frac{\partial_{i}u}{\vert\partial_{i}u\vert}\,\partial_{i}\eta\\
&\,\,\,\,\,\geq\,\xi\,\Lambda^{-1}\sum_{i=1}^{N}(|\partial_{i}u|-\delta_{i})_{+}^{p}\,F'(u)\,\eta^{p}\,-\,p\Lambda\,\xi\,\sum_{i=1}^{N}(|\partial_{i}u|-\delta_{i})_{+}^{p-1}\,\vert F(u)\vert\,\eta^{p-1}\,\vert\partial_{i}\eta\vert\\
&\,\,\,\,\,\geq\,(2\Lambda)^{-1}\,\xi\,\sum_{i=1}^{N}(|\partial_{i}u|-\delta_{i})_{+}^{p}\,F'(u)\,\eta^{p}\,-\,c(p,\Lambda)\,\xi\,\sum_{i=1}^{N}\vert F(u)\vert^{p}\,(F'(u))^{1-p}\,\vert\partial_{i}\eta\vert^{p}\,\chi_{\{\vert\partial_{i}u\vert\,>\,\delta_{i}\}}\,.
\end{align*}At the points of $E$ where $F'\circ u=0$ we obtain the same estimate
as above, but without the terms involving $F(u)$, due to \eqref{Assunzione}.
Thus, choosing $\gamma=\max\,\{2\Lambda,c(p,\Lambda)\}$ we have at
every point of $E$ that 
\[
\frac{1}{\gamma}\,\xi\,\sum_{i=1}^{N}(|\partial_{i}u|-\delta_{i})_{+}^{p}\,F'(u)\,\eta^{p}\,\leq\,\langle A(x,t,\nabla u),\nabla\phi\rangle\,+\,\gamma\,\xi\,\sum_{i=1}^{N}\vert F(u)\vert^{p}\,(F'(u))^{1-p}\,\vert\partial_{i}\eta\vert^{p}\,\chi_{\{\vert\partial_{i}u\vert\,>\,\delta_{i}\}}\,,
\]
where the last term is interpreted as zero whenever $F\circ u=0$.
Note that at the points of $E$ where $F\circ u\neq0$, the last term
is well-defined due to \eqref{Assunzione}. The first two terms of
the previous estimate are integrable, and the last term is non-negative,
so we can integrate over $E$ to obtain\begin{align}\label{eq:integrals_E}
&\frac{1}{\gamma}\iint_{E}\xi\,\sum_{i=1}^{N}(|\partial_{i}u|-\delta_{i})_{+}^{p}\,F'(u)\,\eta^{p}\,dx\,dt\nonumber\\
&\,\,\,\,\,\,\,\leq \iint_{E}\langle A(x,t,\nabla u),\nabla\phi\rangle\,dx\,dt\nonumber\\
&\,\,\,\,\,\,\,\,\,\,\,\,\,\,+\,\gamma\iint_{E\,\cap\,\{F\,\circ\,u\,\neq\,0\}}\xi\,\sum_{i=1}^{N}\vert F(u)\vert^{p}\,(F'(u))^{1-p}\,\vert\partial_{i}\eta\vert^{p}\,\chi_{\{\vert\partial_{i}u\vert\,>\,\delta_{i}\}}\,dx\,dt\,,
\end{align}where the last integral could potentially be infinite. In the first
integral, we may replace $E$ by $\Omega_{T}$ with the understanding
that $\nabla u=0$ almost everywhere on the set where $F'(u)$ is
ill-defined, so that the integrand can be interpreted as zero on this
set. Similarly, in the last integral we may replace $E\cap\{F\circ u\neq0\}$
by $\Omega_{T}$, since except on a set of measure zero, the last
integrand is well-defined when $\nabla u\neq0$ and $F\circ u\neq0$.
With these modifications, we combine \foreignlanguage{british}{\eqref{eq:integrals_E}}
with \eqref{eq:post-limit} and \eqref{eq:integrale_ellittico} to
obtain\begin{align}\label{eq:pre-limite}
&\frac{1}{\gamma}\iint_{\Omega_{T}}\xi\,\sum_{i=1}^{N}(|\partial_{i}u|-\delta_{i})_{+}^{p}\,F'(u)\,\eta^{p}\,dx\,dt\nonumber\\
&\,\,\,\,\,\,\,\leq\,\gamma\iint_{\Omega_{T}}\xi\,\sum_{i=1}^{N}\vert F(u)\vert^{p}\,(F'(u))^{1-p}\,\vert\partial_{i}\eta\vert^{p}\,\chi_{\{\vert\partial_{i}u\vert\,>\,\delta_{i}\}}\,dx\,dt\nonumber\\
&\,\,\,\,\,\,\,\,\,\,\,\,\,\,+\iint_{\mathrm{supp}\,\eta\,\times\,(0,T)}f\,F(u)\,\eta^{p}\,\xi\,dx\,dt\,+\iint_{\Omega_{T}}\eta^{p}\,\partial_{t}\xi\,G(|u|^{\alpha-1}u)\,dx\,dt\,.
\end{align}Note that the function $\xi=\varphi\psi$ converges pointwise from
below to $\varphi\chi_{(\tau_{1},\tau_{2})}$ as $\delta\to0$. Therefore,
by the Monotone Convergence Theorem, the first two integrals in the
above estimate tend to the corresponding terms in \foreignlanguage{british}{\eqref{eq:energy-general}}
as $\delta\to0$.\\
Now we turn our attention to the third integral in \foreignlanguage{british}{\eqref{eq:pre-limite}}.
Using the properties of $\eta$, $\varphi$ and $\psi$, we obtain
\begin{equation}
\vert f\,F(u)\,\eta^{p}\,\xi\vert\,\leq\,\Vert\eta^{p}\Vert_{L^{\infty}(\mathrm{supp}\,\eta)}\,\Vert\varphi\Vert_{L^{\infty}([0,T])}\,\vert f\,F(u)\vert\,.\label{eq:dominating}
\end{equation}
By assumption (\ref{eq:esponenti_coniugati}), the right-hand side
of (\ref{eq:dominating}) belongs to $L^{1}(\mathrm{supp}\,\eta\times(0,T))$.
Therefore, we may apply the Dominated Convergence Theorem to conclude
that the third integral in \foreignlanguage{british}{\eqref{eq:pre-limite}}
converges to the corresponding term in \foreignlanguage{british}{\eqref{eq:energy-general}}
as $\delta\to0$.\\
To handle the last integral in \foreignlanguage{british}{\eqref{eq:pre-limite}},
we note that, by the definition of $\xi$, we have\begin{align}\label{eq:pre-limite02}
&\iint_{\Omega_{T}}\eta^{p}\,\partial_{t}\xi\,G(|u|^{\alpha-1}u)\,dx\,dt\nonumber\\
&\,\,\,\,\,\,\,=\iint_{\Omega_{T}}\eta^{p}\,\partial_{t}\varphi\,\psi\,G(|u|^{\alpha-1}u)\,dx\,dt\,+\,\frac{1}{\delta}\int_{\tau_{1}}^{\tau_{1}+\delta}\int_{\Omega}\eta^{p}\varphi\,G(|u|^{\alpha-1}u)\,dx\,dt\nonumber\\
&\,\,\,\,\,\,\,\,\,\,\,\,\,\,-\,\frac{1}{\delta}\int_{\tau_{2}-\delta}^{\tau_{2}}\int_{\Omega}\eta^{p}\varphi\,G(|u|^{\alpha-1}u)\,dx\,dt\,.
\end{align} In the first term on the right-hand side, we may pass to the limit
as $\delta\to0$ inside the integral by the Dominated Convergence
Theorem, replacing $\psi$ with $\chi_{(\tau_{1},\tau_{2})}$ and
recovering the corresponding term in \foreignlanguage{british}{\eqref{eq:energy-general}}.
The fact that $u\in C^{0}([0,T];L^{\alpha+1}(\Omega))$, together
with Lemma \ref{lem:equivalent_time_cont}, can be used to prove that
the map $t\mapsto\eta^{p}\,G(|u|^{\alpha-1}u)(\cdot,t)$ is continuous
into $L^{1}(\Omega)$. From this it follows that
\[
\frac{1}{\delta}\int_{\tau_{1}}^{\tau_{1}+\delta}\int_{\Omega}\eta^{p}\varphi\,G(|u|^{\alpha-1}u)\,dx\,dt\,\longrightarrow\int_{\Omega}\eta^{p}\varphi\,G(|u|^{\alpha-1}u)(x,\tau_{1})\,dx\,\,\,\,\,\,\,\,\mathrm{as}\,\,\delta\rightarrow0\,.
\]
The last integral in \foreignlanguage{british}{\eqref{eq:pre-limite02}}
can be treated in a similar way. Thus, letting $\delta\to0$ in \foreignlanguage{british}{\eqref{eq:pre-limite}},
we obtain estimate \foreignlanguage{british}{\eqref{eq:energy-general}}.
We have proved the assertion in the case $0<\tau_{1}<\tau_{2}<T$.
The result in the full range $0\leq\tau_{1}<\tau_{2}\le T$ now follows
from the continuity properties of $|u|^{\alpha-1}u$, assumption (\ref{eq:esponenti_coniugati}),
and the appropriate convergence theorems, by considering the limits
as $\tau_{1}\to0$ and $\tau_{2}\to T$.\end{proof}

\selectlanguage{british}%
\noindent $\hspace*{1em}$\foreignlanguage{american}{When the function
$F$ appearing in the previous lemma is chosen appropriately, we obtain
the following classical energy estimates.}
\selectlanguage{american}%
\begin{lem}
\label{lem:energy-classical} Let $\Omega$ be an open, possibly unbounded,
subset\foreignlanguage{british}{ of $\mathbb{R}^{N}$, with $N\geq2$,
and let $T>0$}. Moreover, assume that \eqref{eq:coeff_limit} holds,
let $u$ be a weak solution of \eqref{eq:diffusion} in $\Omega_{T}$,
and suppose that
\begin{equation}
f\,\in\,L^{p'_{\alpha+1}}(0,T;L_{\mathrm{loc}}^{p'_{\alpha+1}}(\Omega))\,,\label{eq:f_new_assump}
\end{equation}
where $p'_{\alpha+1}$ denotes the conjugate exponent of $p_{\alpha+1}:=p(1+\tfrac{\alpha+1}{N})$.
Then, for every $k\in\mathbb{R},$ for every $\eta\in C_{0}^{\infty}(\Omega;[0,\infty))$,
and every $\varphi\in C^{\infty}([0,T];[0,\infty))$ vanishing in
a neighborhood of the origin, we have\begin{align}\label{eq:energy-classical}
&\sup_{\tau\,\in\,[0,T]}\int_{\Omega\times\{\tau\}}(u^{\frac{\alpha+1}{2}}-k^{\frac{\alpha+1}{2}})_{\pm}^{2}\,\eta^{p}\varphi\,dx\,+\sum_{i=1}^{N}\iint_{\Omega_{T}}(|\partial_{i}u|-\delta_{i})_{+}^{p}\,\eta^{p}\,\varphi\,\chi_{\{(u\,-\,k)_{\pm}\,>\,0\}}\,dx\,dt\nonumber\\
&\,\,\,\,\,\,\,\leq\,\mathcal{C}\iint_{\Omega_{T}}(u-k)_{\pm}^{p}\,|\nabla\eta|^{p}\,\varphi\,dx\,dt\,+\,\mathcal{C}\iint_{\Omega_{T}}(u^{\frac{\alpha+1}{2}}-k^{\frac{\alpha+1}{2}})_{\pm}^{2}\,\eta^{p}\,(\partial_{t}\varphi)_{+}\,dx\,dt\nonumber\\
&\,\,\,\,\,\,\,\,\,\,\,\,\,\,+\,\mathcal{C}\iint_{\Omega_{T}}(u-k)_{\pm}\,|f|\,\eta^{p}\,\varphi\,dx\,dt\,,
\end{align}for a positive constant $\mathcal{C}$ depending only on $N$, $\alpha$,
$p$ and $\Lambda$.
\end{lem}

\noindent \begin{proof}[\bfseries{Proof}]We prove that estimate \foreignlanguage{british}{\eqref{eq:energy-classical}}
holds with the positive part and briefly indicate at the end how the
proof can be modified to treat the negative part. We apply Lemma \ref{lem:energy-general}
with the function
\[
F(s)\,=\,(s-k)_{+}\,,
\]
which clearly satisfies the assumptions of the lemma. In particular,
assumption (\ref{eq:esponenti_coniugati}) holds with $\lambda=p_{\alpha+1}$,
since Lemma \ref{lem:parabolic-sobolev} readily implies that
\[
u\,\in\,L^{p_{\alpha+1}}(0,T;L_{\mathrm{loc}}^{p_{\alpha+1}}(\Omega))\,,
\]
and hence $F(u)=(u-k)_{+}$ also belongs to $L^{p_{\alpha+1}}(0,T;L_{\mathrm{loc}}^{p_{\alpha+1}}(\Omega))$.
With the above choice of $F$ we have
\[
F'(s)\,=\,\chi_{\{s\,>\,k\}}\,,\,\,\,\,\,\,\,\,\,\,\,\,g(s)\,=\,F(\vert s\vert^{\frac{1}{\alpha}\,-\,1}s)\,=\,(\vert s\vert^{\frac{1}{\alpha}\,-\,1}s-k)_{+}\,,
\]
and moreover we choose $G$ as the integral function
\[
G(\tau):=\int_{\vert k\vert^{\alpha-1}k}^{\tau}g(s)\,ds\,.
\]
An easy calculation reveals that
\[
G(\vert u\vert^{\alpha-1}u)\,=\,\mathfrak{b}_{\alpha}[u,k]\,\chi_{\{u\,>\,k\}}\,.
\]
Therefore, in this case estimate \foreignlanguage{british}{\eqref{eq:energy-general}}
takes the form\begin{align*}
&\int_{\Omega\times\{\tau_{2}\}}\mathfrak{b}_{\alpha}[u,k]\,\chi_{\{u\,>\,k\}}\,\eta^{p}\,\varphi\,dx\,+\,\frac{1}{\gamma}\,\sum_{i=1}^{N}\int_{\tau_{1}}^{\tau_{2}}\int_{\Omega}(\vert\partial_{i}u\vert-\delta_{i})_{+}^{p}\,\eta^{p}\,\varphi\,\chi_{\{u\,>\,k\}}\,dx\,dt\\
&\,\,\,\,\,\,\,\leq\,\int_{\Omega\times\{\tau_{1}\}}\mathfrak{b}_{\alpha}[u,k]\,\chi_{\{u\,>\,k\}}\,\eta^{p}\,\varphi\,dx\,+\,\int_{\tau_{1}}^{\tau_{2}}\int_{\Omega}\mathfrak{b}_{\alpha}[u,k]\,\chi_{\{u\,>\,k\}}\,\eta^{p}\,\partial_{t}\varphi\,dx\,dt\\
&\,\,\,\,\,\,\,\,\,\,\,\,\,\,+\,\gamma N\int_{\tau_{1}}^{\tau_{2}}\int_{\Omega}(u-k)_{+}^{p}\,\vert\nabla\eta\vert^{p}\,\varphi\,dx\,dt\,+\,\int_{\tau_{1}}^{\tau_{2}}\int_{\Omega}(u-k)_{+}\,f\,\eta^{p}\,\varphi\,dx\,dt\,.
\end{align*}Choosing $\tau_{1}=0$ and recalling that $\varphi(0)=0$, we infer
from the previous estimate that\begin{align}\label{eq:ultima-est-energy}
&\int_{\Omega\times\{\tau_{2}\}}\mathfrak{b}_{\alpha}[u,k]\,\chi_{\{u\,>\,k\}}\,\eta^{p}\,\varphi\,dx\,+\,\sum_{i=1}^{N}\int_{0}^{\tau_{2}}\int_{\Omega}(\vert\partial_{i}u\vert-\delta_{i})_{+}^{p}\,\eta^{p}\,\varphi\,\chi_{\{u\,>\,k\}}\,dx\,dt\nonumber\\
&\,\,\,\,\,\,\,\leq\,c\int_{0}^{\tau_{2}}\int_{\Omega}\mathfrak{b}_{\alpha}[u,k]\,\chi_{\{u\,>\,k\}}\,\eta^{p}\,(\partial_{t}\varphi)_{+}\,dx\,dt\,+\,c\int_{0}^{\tau_{2}}\int_{\Omega}(u-k)_{+}^{p}\,\vert\nabla\eta\vert^{p}\,\varphi\,dx\,dt\nonumber\\
&\,\,\,\,\,\,\,\,\,\,\,\,\,\,+\,c\int_{0}^{\tau_{2}}\int_{\Omega}(u-k)_{+}\,\vert f\vert\,\eta^{p}\,\varphi\,dx\,dt\,,
\end{align}where $c$ is a positive constant depending only on $N$, $p$ and
$\Lambda$. Note that $\mathfrak{b}_{\alpha}[u,k]\,\chi_{\{u\,>\,k\}}$
may be replaced by $(\vert u\vert^{\frac{\alpha-1}{2}}u-\vert k\vert^{\frac{\alpha-1}{2}}k)_{+}^{2}$
on both sides by (\ref{est:b-all-alpha}). We can also estimate the
right-hand side upwards by replacing $\tau_{2}$ with $T$. Since
both terms on the left-hand side of the resulting estimate are non-negative,
each of them can be bounded by the right-hand side. Taking the supremum
over $\tau_{2}\in[0,T]$ and adding the resulting inequalities, we
end up with \foreignlanguage{british}{\eqref{eq:energy-classical}}.

\selectlanguage{british}%
\noindent $\hspace*{1em}$Finally, in the case of the negative part,
one chooses $F(s)=-\,(s-k)_{-}$ and the remainder of the proof is
analogous.\foreignlanguage{american}{\end{proof}}

\selectlanguage{american}%
\noindent \begin{brem}\label{Osserva_su_f}The conclusion of Lemma
\ref{lem:energy-classical} holds, in particular, if $f$ satisfies
condition (\ref{assumpt:f}). Indeed, if (\ref{assumpt:f}) holds,
then
\[
\sigma'\,<\,\frac{N+p}{N}\,<\,\frac{pN+p(\alpha+1)}{N}\,=\,p_{\alpha+1}\,,
\]
and therefore
\[
L^{\sigma}(\Omega_{T})\,\subset\,L^{p'_{\alpha+1}}(0,T;L_{\mathrm{loc}}^{p'_{\alpha+1}}(\Omega))\,.
\]
Thus, the assumption on $f$ in Lemma \ref{lem:energy-classical}
is satisfied.\end{brem}
\selectlanguage{british}%
\begin{defn}[\textbf{De Giorgi classes $\mathcal{A}(\alpha,p,\{\delta_{i}\},\Omega_{T},f,\mathcal{C})$}]
\noindent Let $f:\Omega_{T}\to\mathbb{R}$ be a measurable function.
We say that a function $u:\Omega_{T}\to\mathbb{R}$ belongs to the
class $\mathcal{A}(\alpha,p,\{\delta_{i}\},\Omega_{T},f,\mathcal{C})$
if 
\[
u\,\in\,L^{p}(0,T;W^{1,p}(\Omega))\cap C^{0}([0,T];L^{\alpha+1}(\Omega))
\]
and for all $k\in\mathbb{R}$ and all $\eta,\varphi$ as in Lemma
\ref{lem:energy-classical}, the function $u$ satisfies the inequalities
\eqref{eq:energy-classical}.\smallskip{}

\noindent \begin{brem}\label{def:oss-impor}By Lemma \ref{lem:energy-classical},
under assumptions (\ref{eq:coeff_limit}) and (\ref{eq:f_new_assump}),
weak solutions to (\ref{eq:diffusion}) in the sense of Definition
\ref{def:weaksol} belong to $\mathcal{A}(\alpha,p,\{\delta_{i}\},\Omega_{T},f,\mathcal{C})$.
In particular, in view of Remark \ref{Osserva_su_f}, the same conclusion
holds if $f$ satisfies the stronger condition (\ref{assumpt:f}).\end{brem}\smallskip{}

\noindent $\hspace*{1em}$Starting from the above definition and Lemmas
\ref{lem:energy-general} and \ref{lem:energy-classical}, in the
next sections we establish local boundedness for functions in $\mathcal{A}(\alpha,p,\{\delta_{i}\},\Omega_{T},f,\mathcal{C})$
under assumption (\ref{assumpt:f}), as well as contractive estimates
and global boundedness for integrable weak solutions to the Cauchy
problem associated with equation (\ref{eq:diffusion}).
\end{defn}

\selectlanguage{american}%

\section{Local boundedness\label{sec:Local-boundedness}}

\selectlanguage{british}%
\noindent $\hspace*{1em}$In this section we prove the local boundedness
of weak solutions to equation (\ref{eq:diffusion}), assuming that
the datum $f$ satisfies condition (\ref{assumpt:f}). Our proofs
are based on De Giorgi-type iterations combining the energy estimates
established in Lemma \ref{lem:energy-classical} with the Sobolev
embedding of Lemma \ref{lem:parabolic-sobolev}. Throughout the section,
we will use the backward parabolic cylinders defined in Section \ref{sec:setting},
which turn out to be convenient in our setting.\\
$\hspace*{1em}$As mentioned in the introduction, there are two ranges
for $p$ that require somewhat different arguments, and in the range
corresponding to small values of $p$ we also require some extra integrability
of the weak solutions. For clarity, the two cases have been treated
in separate subsections.

\subsection{The case $p>\frac{N(\alpha\,+\,1)}{N\,+\,\alpha\,+\,1}$ }

\noindent $\hspace*{1em}$In this section we focus on the case in
which $p$ satisfies the following lower bound:

\selectlanguage{american}%
\begin{align}
p\,>\,\frac{N(\alpha+1)}{N+\alpha+1}\,,\label{p_lower_bnd}
\end{align}
which can also be expressed as $p_{\alpha+1}>\alpha+1$. We recall
that 
\[
p_{\alpha+1}\,=\,p\left(1+\,\frac{\alpha+1}{N}\right)
\]
and note, for later use, that $p_{\alpha+1}>p$ trivially. The argument
consists of two parts. First, we obtain local boundedness without
an explicit bound (Theorem \ref{theo:local_bdd_large_bar_p} below).
Once local boundedness has been established, we derive quantitative
estimates for both the essential supremum and the essential infimum
(see Theorems \ref{thm:loc_bdd_p_large_LP} and \ref{thm:better_est}).
\begin{thm}
\label{theo:local_bdd_large_bar_p} Let $u\in\mathcal{A}(\alpha,p,\{\delta_{i}\},\Omega_{T},f,\mathcal{C})$,
where $\alpha\in(0,\infty)$, $p\in(1,\infty)$ and $f$ satisfies
$(\ref{assumpt:f})$. Suppose that \eqref{p_lower_bnd} holds. Then
$u$ is locally essentially bounded. In particular, if $u$ is a weak
solution of \eqref{eq:diffusion} under assumptions $\eqref{eq:coeff_limit}$
and $(\ref{assumpt:f})$, then $u$ belongs to $L_{\mathrm{loc}}^{\infty}(\Omega_{T})$.
\end{thm}

\noindent \begin{proof}[\bfseries{Proof}]Let $u\in\mathcal{A}(\alpha,p,\{\delta_{i}\},\Omega_{T},f,\mathcal{C})$,
with $\alpha$, $p$ and $f$ as in the statement of the theorem.
We begin by proving that $u$ is locally bounded from above. Let $z_{o}=(x_{o},t_{o})\in\Omega_{T}$
and $r>0$ be such that $Q_{r}(z_{o})\Subset\Omega_{T}$. For a fixed
$\nu\in(0,1)$, we define a sequence of space-time cylinders as follows:
\begin{align}
r_{j}:=\,r(\nu+(1-\nu)2^{-j}),\,\,\,\,\,\,\,\,Q_{j}:=\,Q_{r_{j}}(z_{o})=B_{j}\times T_{j}\,,\,\,\,\,\,\,\,\,\,\,\,\,\mathrm{for}\,\,j\in\mathbb{N}_{0}\,.\label{r_j_Q_j-def}
\end{align}
Here, $B_{j}:=B_{r_{j}}(x_{o}),$ while $T_{j}:=(t_{o}-r_{j}^{p},t_{o})$.
Choose cut-off functions $\eta_{j}\in C_{0}^{\infty}(B_{j};[0,1])$
such that $\eta_{j}=1$ on $B_{j+1}$ and 
\begin{align}
|\nabla\eta_{j}|\,\leq\,\frac{c\,2^{j+1}}{r(1-\nu)}\,.\label{bound:grad-eta}
\end{align}
Similarly, we take $\varphi_{j}\in C^{\infty}(T_{j};[0,1])$ vanishing
in a neighborhood of the lower endpoint of $T_{j}$, such that $\varphi_{j}=1$
on $T_{j+1}$ and 
\begin{align}
|\varphi'_{j}|\,\leq\,\frac{c\,2^{jp}}{r^{p}(1-\nu)^{p}}\,.\label{bound:der-varphi}
\end{align}
Moreover, we define the sequences
\begin{align}
Y_{j}:=\iint_{Q_{j}}(u-k_{j})_{+}^{p_{\alpha+1}}\,dx\,dt\,,\,\,\,\,\,\,\,\,k_{j}:=\,k(2-2^{-j}),\quad\widehat{k}_{j}:=\,\frac{1}{2}(k_{j}+k_{j+1})\,,\,\,\,\,\,\,\,\,\,\,\mathrm{for}\,\,j\in\mathbb{N}_{0}\,,\label{def:k_j_alpha_small}
\end{align}
where $k\geq1$ is a number to be chosen later. Note that $Y_{j}$
is finite for every $j\in\mathbb{N}_{0}$, since $Q_{j}\subseteq Q_{0}=Q_{r}(z_{o})\Subset\Omega_{T}$
and, by Lemma \ref{lem:parabolic-sobolev}, $u\in L^{p_{\alpha+1}}(Q_{r}(z_{o}))$.\\
In the following calculation, we estimate $Y_{j+1}$ by first introducing
the cut-off function in space $\eta_{j}^{p}$. Then, we use Lemma
\ref{lem:parabolic-sobolev} with $v=(u-k_{j+1})_{+}\,\eta_{j}^{p}$
and introduce the cut-off function in time $\varphi_{j}$. All in
all, the calculation takes the form\begin{align}\label{eq:notsolongcalc}   
Y_{j+1}&=\iint_{Q_{j+1}}(u-k_{j+1})_{+}^{p_{\alpha+1}}\,dx\,dt\,\leq\int_{T_{j+1}}\int_{B_{j}}[(u-k_{j+1})_{+}\,\eta_{j}^{p}]^{p_{\alpha+1}}\,dx\,dt\nonumber\\
&\leq c\Big[\sup_{T_{j+1}}\int_{B_{j}}(u-k_{j+1})_{+}^{\alpha+1}\,\eta_{j}^{p(\alpha+1)}\,dx\Big]^{\frac{p}{N}}\int_{T_{j+1}}\int_{B_{j}}|\nabla[(u-k_{j+1})_{+}\,\eta_{j}^{p}]|^{p}\,dx\,dt\nonumber\\
&\leq c\Big[\sup_{T_{j}}\int_{B_{j}}(u-k_{j+1})_{+}^{\alpha+1}\,\eta_{j}^{p}\,\varphi_{j}\,dx\Big]^{\frac{p}{N}}\nonumber\\
&\,\,\,\,\,\,\,\,\,\,\times\iint_{Q_{j}}[|\nabla(u-k_{j+1})_{+}|^{p}\,\eta_{j}^{p}\,\varphi_{j}\,+\,(u-k_{j+1})_{+}^{p}\,|\nabla\eta_{j}|^{p}\,\varphi_{j}]\,dx\,dt\,,
\end{align}where $c$ is a positive constant depending only on $N$, $\alpha$
and $p$. Note that in the previous computations we have also used
that $0\leq\eta_{j}\leq1$. To estimate the supremum over $T_{j}$,
we observe that\begin{align}\label{eq:est:timesupterm}
(u-k_{j+1})_{+}^{\alpha+1}\,&\leq\,u^{\alpha+1}\,\chi_{\{u\,>\,k_{j+1}\}}\,=\,(u^{\frac{\alpha+1}{2}}-\widehat{k}_{j}^{\frac{\alpha+1}{2}}+\widehat{k}_{j}^{\frac{\alpha+1}{2}})^{2}\,\chi_{\{u\,>\,k_{j+1}\}}\nonumber\\
&\leq\,2\,(u^{\frac{\alpha+1}{2}}-\widehat{k}_{j}^{\frac{\alpha+1}{2}})_{+}^{2}\,+\,2\,\widehat{k}_{j}^{\alpha+1}\,\chi_{\{u\,>\,k_{j+1}\}}\nonumber\\
&\leq\,2\,(u^{\frac{\alpha+1}{2}}-\widehat{k}_{j}^{\frac{\alpha+1}{2}})_{+}^{2}\,+\,2\,\widehat{k}_{j}^{\alpha+1}\,\frac{(u^{\frac{\alpha+1}{2}}-\widehat{k}_{j}^{\frac{\alpha+1}{2}})^{2}}{(k_{j+1}^{\frac{\alpha+1}{2}}-\widehat{k}_{j}^{\frac{\alpha+1}{2}})^{2}}\,\chi_{\{u\,>\,k_{j+1}\}}\nonumber\\
&\leq\,c\,4^{j}\,(u^{\frac{\alpha+1}{2}}-\widehat{k}_{j}^{\frac{\alpha+1}{2}})_{+}^{2}\,,
\end{align}where the last inequality follows from the Mean Value Theorem applied
to the denominator in the penultimate line. Indeed, there exists $\zeta_{j}\in(\widehat{k}_{j},k_{j+1})\subset[k,2k]$
such that
\begin{equation}
k_{j+1}^{\frac{\alpha+1}{2}}-\widehat{k}_{j}^{\frac{\alpha+1}{2}}=\,\frac{\alpha+1}{2}\,\zeta_{j}^{\frac{\alpha-1}{2}}\,(k_{j+1}-\widehat{k}_{j})\,\geq\,c(\alpha)\,k^{\frac{\alpha+1}{2}}\,2^{-j}\,.\label{eq:Lagrange}
\end{equation}
Combining \foreignlanguage{british}{\eqref{eq:notsolongcalc}} with\foreignlanguage{british}{
\eqref{eq:est:timesupterm}, we have\begin{align}\label{eq:est:haest}
Y_{j+1}\,&\leq\, c\,b^{j}\Big[\sup_{T_{j}}\int_{B_{j}}(u^{\frac{\alpha+1}{2}}-\widehat{k}_{j}^{\frac{\alpha+1}{2}})_{+}^{2}\,\eta_{j}^{p}\,\varphi_{j}\,dx\Big]^{\frac{p}{N}}\nonumber\\
&\,\,\,\,\,\,\,\,\,\,\,\,\,\,\,\,\,\times\iint_{Q_{j}}[|\nabla(u-\widehat{k}_{j})_{+}|^{p}\,\eta_{j}^{p}\,\varphi_{j}\,+\,(u-\widehat{k}_{j})_{+}^{p}\,|\nabla\eta_{j}|^{p}\,\varphi_{j}]\,dx\,dt\,,
\end{align}where $b>1$ is a constant depending only on $N$ and $p$. Note that,
in the last integral, we were able to replace $k_{j+1}$ with $\widehat{k}_{j},$
since $\widehat{k}_{j}<k_{j+1}$. We now} estimate the first part
of this integral using (\ref{est:basic_delta}) with $\xi=\nabla u$,
the energy estimate \foreignlanguage{british}{\eqref{eq:energy-classical}},
and the upper bound for $\eta_{j}$ and $\varphi_{j}$ as follows:\begin{align*}  
&\iint_{Q_{j}}|\nabla(u-\widehat{k}_{j})_{+}|^{p}\,\eta_{j}^{p}\,\varphi_{j}\,dx\,dt\,=\iint_{Q_{j}}|\nabla u|^{p}\,\chi_{\{u\,>\,\widehat{k}_{j}\}}\,\eta_{j}^{p}\,\varphi_{j}\,dx\,dt\\
&\,\,\,\,\,\,\,\leq c\iint_{Q_{j}}\left[\sum_{i=1}^{N}(|\partial_{i}u|-\delta_{i})_{+}^{p}\,\eta_{j}^{p}\,\varphi_{j}\,\chi_{\{u\,>\,\widehat{k}_{j}\}}+\,\delta_{\max}^{p}\,\chi_{\{u\,>\,\widehat{k}_{j}\}}\,\eta_{j}^{p}\,\varphi_{j}\right]dx\,dt\\
&\,\,\,\,\,\,\,\leq c\iint_{Q_{j}}[(u-\widehat{k}_{j})_{+}^{p}\,|\nabla\eta_{j}|^{p}\,\varphi_{j}\,+\,(u^{\frac{\alpha+1}{2}}-\widehat{k}_{j}^{\frac{\alpha+1}{2}})_{+}^{2}\,\eta_{j}^{p}\,(\partial_{t}\varphi_{j})_{+}]\,dx\,dt\\
&\,\,\,\,\,\,\,\,\,\,\,\,\,\,+\,c\iint_{Q_{j}}|f|\,(u-\widehat{k}_{j})_{+}\,\eta_{j}^{p}\,\varphi_{j}\,dx\,dt\,+\,c\,\delta_{\max}^{p}\,|E_{j}|\,,
\end{align*}where $c=c(N,\alpha,p,\Lambda)>0$, $\delta_{\mathrm{max}}:=\max\,\{\delta_{1},...,\delta_{N}\}$,
and we have set
\begin{equation}
E_{j}:=\,\{(x,t)\in Q_{j}:u(x,t)>\widehat{k}_{j}\}\,.\label{eq:sopralivello}
\end{equation}
Since the \foreignlanguage{british}{supremum over $T_{j}$ in \eqref{eq:est:haest}}
can likewise be controlled using the energy estimate \foreignlanguage{british}{\eqref{eq:energy-classical}},
we obtain\begin{align}\label{eq:integ_after_energy}
Y_{j+1}\,&\leq\,c\,b^{j}\,\Big[\iint_{Q_{j}}(u-\widehat{k}_{j})_{+}^{p}\,|\nabla\eta_{j}|^{p}\,dx\,dt\,+\iint_{Q_{j}}(u^{\frac{\alpha+1}{2}}-\widehat{k}_{j}^{\frac{\alpha+1}{2}})_{+}^{2}\,(\partial_{t}\varphi_{j})_{+}\,dx\,dt\nonumber\\
&\qquad\,\,\,\,\,\,+\iint_{Q_{j}}|f|\,(u-\widehat{k}_{j})_{+}\,\eta_{j}^{p}\,\varphi_{j}\,dx\,dt\,+\,\delta_{\max}^{p}\,|E_{j}|\Big]^{\frac{N+p}{N}}.
\end{align}Applying Young's inequality to the first factor in the first integrand
on the right-hand side of \foreignlanguage{british}{\eqref{eq:integ_after_energy}},
we have
\begin{equation}
(u-\widehat{k}_{j})_{+}^{p}\,\leq\,(u-\widehat{k}_{j})_{+}^{p_{\alpha+1}}\,+\,\chi_{\{u\,>\,\widehat{k}_{j}\}}\,.\label{eq:est:term1}
\end{equation}
Recalling that (\ref{p_lower_bnd}) is equivalent to $p_{\alpha+1}>\alpha+1$,
we estimate the first factor in the second integrand on the right-hand
side of \foreignlanguage{british}{\eqref{eq:integ_after_energy}}
as follows:\begin{align}\label{eq:est:term2}
(u^{\frac{\alpha+1}{2}}-\widehat{k}_{j}^{\frac{\alpha+1}{2}})_{+}^{2}\,&\leq\,u^{\alpha+1}\,\chi_{\{u\,>\,\widehat{k}_{j}\}}\nonumber\\
&\leq\,c\,(u-\widehat{k}_{j})_{+}^{\alpha+1}\,+\,c\,\widehat{k}_{j}^{\alpha+1}\,\chi_{\{u\,>\,\widehat{k}_{j}\}}\nonumber\\
&\leq\,c\,(u-\widehat{k}_{j})_{+}^{p_{\alpha+1}}\,+\,c\,\widehat{k}_{j}^{\alpha+1}\,\chi_{\{u\,>\,\widehat{k}_{j}\}}\,,
\end{align}where, in the last line, we have used Young's inequality and the fact
that $\widehat{k}_{j}>k\geq1$. Now we apply Hölder's inequality and
the fact that $\eta_{j},\varphi_{j}\leq1$ to find that
\begin{align}
\iint_{Q_{j}}|f|\,(u-\widehat{k}_{j})_{+}\,\eta_{j}^{p}\,\varphi_{j}\,dx\,dt\,\leq\,\Vert f\Vert_{L^{\sigma}(Q_{0})}\left[\iint_{Q_{j}}(u-\widehat{k}_{j})_{+}^{\sigma'}\,dx\,dt\right]^{1-\frac{1}{\sigma}},\label{est:term3}
\end{align}
where $\sigma':=\frac{\sigma}{\sigma-1}$ is the conjugate exponent
of $\sigma$. At this stage, note that the range of $\sigma$ in \eqref{assumpt:f}
ensures that
\begin{align}
\sigma'\,<\,\frac{N+p}{N}\,<\,\frac{pN+p(\alpha+1)}{N}\,=\,p_{\alpha+1}\,.\label{sigma-prime_ineq}
\end{align}
Therefore, we may apply Young's inequality again to obtain
\begin{align}
(u-\widehat{k}_{j})_{+}^{\sigma'}\,\leq\,(u-\widehat{k}_{j})_{+}^{p_{\alpha+1}}\,+\,\chi_{\{u\,>\,\widehat{k}_{j}\}}\,.\label{sigma-prime-and-NeilYoung}
\end{align}
Combining (\ref{est:term3}) and (\ref{sigma-prime-and-NeilYoung}),
and recalling the definition of $E_{j}$ in (\ref{eq:sopralivello}),
we have
\begin{equation}
\iint_{Q_{j}}|f|\,(u-\widehat{k}_{j})_{+}\,\eta_{j}^{p}\,\varphi_{j}\,dx\,dt\,\leq\,\Vert f\Vert_{L^{\sigma}(Q_{0})}\,(Y_{j}\,+\,|E_{j}|)^{1-\frac{1}{\sigma}}.\label{eq:term3nuovo}
\end{equation}
Joining estimates \eqref{eq:integ_after_energy}$-$\eqref{eq:est:term2}
and \eqref{eq:term3nuovo}, using \eqref{bound:grad-eta} and \eqref{bound:der-varphi},
and recalling the definition of $Y_{j}$ and the fact that $\widehat{k}_{j}>k\geq1$,
we obtain
\[
Y_{j+1}\,\leq\,c\,b^{j}\left[Y_{j}\,+\,(k^{\alpha+1}+\delta_{\mathrm{max}}^{p})\,\vert E_{j}\vert\,+\,\Vert f\Vert_{L^{\sigma}(Q_{0})}\,Y_{j}^{1-\frac{1}{\sigma}}\,+\,\Vert f\Vert_{L^{\sigma}(Q_{0})}\,\vert E_{j}\vert^{1-\frac{1}{\sigma}}\right]^{\frac{N+p}{N}},
\]
where the constant $c$ now depends on $N,\alpha,p,\Lambda,\sigma,r$
and $\nu$. At this point, we observe that
\begin{equation}
\vert E_{j}\vert\,=\iint_{E_{j}}\frac{(\widehat{k}_{j}-k_{j})^{p_{\alpha+1}}}{(\widehat{k}_{j}-k_{j})^{p_{\alpha+1}}}\,dx\,dt\,\leq\iint_{E_{j}}\frac{(u-k_{j})_{+}^{p_{\alpha+1}}}{(\widehat{k}_{j}-k_{j})^{p_{\alpha+1}}}\,dx\,dt\,\leq\,c_{1}\,\frac{2^{j\,p_{\alpha+1}}}{k^{p_{\alpha+1}}}\,Y_{j}\,,\label{eq:superlevel}
\end{equation}
where $c_{1}=c_{1}(N,\alpha,p)>0.$ Combining the two previous estimates
and using that $p_{\alpha+1}>\alpha+1$ and $k\geq1$, we get
\begin{align}
Y_{j+1} & \,\leq\,c\,b^{j}\left[Y_{j}\,+\,\Vert f\Vert_{L^{\sigma}(Q_{0})}\,Y_{j}^{1-\frac{1}{\sigma}}\right]^{\frac{N+p}{N}},\label{shtevensh}
\end{align}
where $b>1$ now also depends on $\alpha$, while $c$ additionally
depends on $\delta_{\mathrm{max}}$. In order to have the same exponent
for $Y_{j}$ in both terms in the square brackets, we observe that
\begin{align*}
Y_{j}\,=\,Y_{j}^{\frac{1}{\sigma}}\,Y_{j}^{1-\frac{1}{\sigma}}\,\leq\left(\iint_{Q_{0}}u_{+}^{p_{\alpha+1}}\,dx\,dt\right)^{\frac{1}{\sigma}}Y_{j}^{1-\frac{1}{\sigma}}.
\end{align*}
Combining this estimate with \eqref{shtevensh}, we obtain
\begin{equation}
Y_{j+1}\,\leq\,c\,b^{j}\left[\iint_{Q_{0}}(u_{+}^{p_{\alpha+1}}+\vert f\vert^{\sigma})\,dx\,dt\right]^{\frac{N+p}{N\sigma}}Y_{j}^{1+\beta},\label{eq:appoggio}
\end{equation}
where 
\begin{equation}
\beta\,=\,\frac{p}{N}\left(1-\,\frac{N+p}{\sigma p}\right)>0\,,\label{eq:def:beta}
\end{equation}
since $\sigma>(N+p)/p$ by assumption. Noting that the integral in
(\ref{eq:appoggio}) is finite due to the integrability properties
of $u$ and $f$, we can write
\begin{align}
Y_{j+1}\,\leq\,C_{u,f}\,b^{j}\,Y_{j}^{1+\beta}\,,\label{eq:rec_ineq}
\end{align}
where
\[
C_{u,f}\,=\,c\left[\iint_{Q_{0}}(u_{+}^{p_{\alpha+1}}+\vert f\vert^{\sigma})\,dx\,dt\right]^{\frac{N+p}{N\sigma}}<+\infty\,.
\]
If $C_{u,f}=0$, we immediately deduce that $u\leq0$ almost everywhere
in $Q_{r}(z_{o})$. If instead $C_{u,f}>0$, then, by Lemma \ref{lem:fastconvg},
the sequence $\{Y_{j}\}$ converges to zero provided that
\begin{align}
\iint_{Q_{r}(z_{o})}(u-k)_{+}^{p_{\alpha+1}}\,dx\,dt\,=\,Y_{0}\,\leq\,C_{u,f}^{-\,\frac{1}{\beta}}\,b^{-\,\frac{1}{\beta^{2}}}\,.\label{cond:Y_0}
\end{align}
The integral on the left-hand side tends to zero as $k\rightarrow\infty$
by the Dominated Convergence Theorem, since $u\in L^{p_{\alpha+1}}(Q_{r}(z_{o}))$,
whereas the right-hand side is independent of $k$. Therefore, condition
(\ref{cond:Y_0}) is satisfied for some sufficiently large $k\geq1$.
For such a choice of $k$, we have
\begin{align*}
\iint_{Q_{\nu r}(z_{o})}(u-2k)_{+}^{p_{\alpha+1}}\,dx\,dt\,\leq\,Y_{j}\to0\,\,\,\,\,\,\,\,\mathrm{as}\,\,j\to\infty\,,
\end{align*}
\foreignlanguage{british}{which implies that }$u\leq2k\,$ a.e. in
$Q_{\nu r}(z_{o})$\foreignlanguage{british}{.} We have thus proved
that $u$ is locally bounded from above.\\
\foreignlanguage{british}{$\hspace*{1em}$}Finally, with analogous
arguments, but using the energy estimate \foreignlanguage{british}{\eqref{eq:energy-classical}}
with the negative part, we also get a lower bound for $u$.\end{proof}\smallskip{}

\selectlanguage{british}%
\noindent $\hspace*{1em}$\foreignlanguage{american}{Having established
the local boundedness of functions in $\mathcal{A}(\alpha,p,\{\delta_{i}\},\Omega_{T},f,\mathcal{C})$
under condition (\ref{assumpt:f}), we now derive explicit bounds
for their essential supremum and infimum.}
\selectlanguage{american}%
\begin{thm}
\noindent \label{thm:loc_bdd_p_large_LP} Let $u\in\mathcal{A}(\alpha,p,\{\delta_{i}\},\Omega_{T},f,\mathcal{C})$,
where $\alpha\in(0,\infty)$, $p\in(1,\infty)$ and $f$ satisfies
$(\ref{assumpt:f})$. Suppose that \eqref{p_lower_bnd} holds. Then,
for every cylinder $Q_{r}(z_{o})\Subset\Omega_{T}$ and every $s\in(0,1)$,
we have the explicit upper bound\begin{align}\label{eq:sup-bound-explicit-slow-diff}
\underset{Q_{sr}(z_{o})}{\mathrm{ess}\,\sup}\,\,u\,&\leq\,c\left(\big[1+(r(1-s))^{-(N+p)}\big]\iint_{Q_{r}(z_{o})}u_{+}^{P}\,dx\,dt\right)^{\frac{p}{p(N+\alpha+1)\,-\,NP}}\nonumber\\
&\,\,\,\,\,\,\,+c\left(1+\iint_{Q_{r}(z_{o})}|f|^{\sigma}\,dx\,dt\right)^{\frac{p}{p(N+\alpha+1)\,-\,NP\,+\,pP\min\,\{0,\frac{\sigma}{P'}-1\}}},
\end{align}where $P:=\max\,\{\alpha+1,p\}$ and $c$ is a positive constant depending
only on $N$, $\alpha$, $p$, $\Lambda$, $\sigma$ and $\max\,\{\delta_{1},\ldots,\delta_{N}\}$.
An analogous lower bound holds for the essential infimum, with $u_{-}$
replacing $u_{+}$ on the right-hand side of \foreignlanguage{british}{\eqref{eq:sup-bound-explicit-slow-diff}}.
In particular, under assumptions $\eqref{eq:coeff_limit}$ and $(\ref{assumpt:f})$,
these estimates hold for all weak solutions to \eqref{eq:diffusion}
in the sense of Definition \ref{def:weaksol}.
\end{thm}

\noindent \begin{proof}[\bfseries{Proof}]Let $u\in\mathcal{A}(\alpha,p,\{\delta_{i}\},\Omega_{T},f,\mathcal{C})$,
with $\alpha$, $p$ and $f$ as in the statement of the theorem.
We limit ourselves to the derivation of the upper bound in \foreignlanguage{british}{\eqref{eq:sup-bound-explicit-slow-diff}},
since arguments analogous to those presented below yield a lower bound
for the essential infimum in terms of the $L^{P}$-norm of $u_{-}\,$.
We use the radii $r_{j}$, space-time cylinders $Q_{j}$, and cut-off
functions $\eta_{j}$ and $\varphi_{j}$ as defined in \eqref{r_j_Q_j-def}$-$\eqref{bound:der-varphi}.
Furthermore, we set 
\begin{align}
\gamma:=\,\frac{p_{\alpha+1}}{P}\,=\,\frac{p}{P}\Big(1+\frac{\alpha+1}{N}\Big)\,.\label{def:gamma}
\end{align}
In the case $P=p$ we directly see that $\gamma>1$. In the case $P=\alpha+1$
we can instead use \eqref{p_lower_bnd} to conclude that $\gamma>1$.
We now define the sequences
\begin{align*}
X_{j}:=\iint_{Q_{j}}(u-k_{j})_{+}^{P}\,dx\,dt\,,\,\,\,\,\,\,\,\,k_{j}:=\,k(1-2^{-j}),\quad\widehat{k}_{j}:=\,\frac{1}{2}(k_{j}+k_{j+1})\,,\,\,\,\,\,\,\,\,\,\,\mathrm{for}\,\,j\in\mathbb{N}_{0}\,,
\end{align*}
where $k\geq1$ is a number to be chosen later. Note that $X_{j}$
is finite for every $j\in\mathbb{N}_{0}$, since $u\in L^{p}(0,T;W^{1,p}(\Omega))\cap C^{0}([0,T];L^{\alpha+1}(\Omega))$.
By Hölder's inequality, we have 
\begin{align}
X_{j+1}\,\leq\left[\iint_{Q_{j+1}}(u-k_{j+1})_{+}^{p_{\alpha+1}}\,dx\,dt\right]^{\frac{1}{\gamma}}|Q_{j+1}\cap\{u>k_{j+1}\}|^{1-\,\frac{1}{\gamma}}\,.\label{X-jplusone_estim1}
\end{align}
Observing that the integral in square brackets has the same form as
the term $Y_{j+1}$ in the proof of Theorem \ref{theo:local_bdd_large_bar_p},
we may argue exactly as in the first part of that proof to recover
the estimate \foreignlanguage{british}{\eqref{eq:integ_after_energy}}
with the current definitions of $\widehat{k}_{j}$, $Y_{j+1}$ and
$E_{j}$. In particular, the last inequality in \eqref{eq:est:timesupterm}
still holds because, by the Mean Value Theorem, there exists $\zeta_{j}\in(\widehat{k}_{j},k_{j+1})\subset\left[\frac{k}{5},k\right]$
such that (\ref{eq:Lagrange}) holds with the present definitions
of $\widehat{k}_{j}$ and $k_{j+1}$. Combining \foreignlanguage{british}{\eqref{eq:integ_after_energy}}
with (\ref{X-jplusone_estim1}), we then obtain\begin{align}\label{eq:X_jplusone_after_energy}
X_{j+1}\,\leq\,c\,b^{j}\,&\Big[\iint_{Q_{j}}(u-\widehat{k}_{j})_{+}^{p}\,|\nabla\eta_{j}|^{p}\,dx\,dt\,+\iint_{Q_{j}}(u^{\frac{\alpha+1}{2}}-\widehat{k}_{j}^{\frac{\alpha+1}{2}})_{+}^{2}\,(\partial_{t}\varphi_{j})_{+}\,dx\,dt\nonumber\\
&+\iint_{Q_{j}}|f|\,(u-\widehat{k}_{j})_{+}\,dx\,dt\,+\,\delta_{\max}^{p}\,|E_{j}|\Big]^{\frac{N+p}{\gamma N}}\,|Q_{j+1}\cap\{u>k_{j+1}\}|^{1-\,\frac{1}{\gamma}}\,,
\end{align}where $c=c(N,\alpha,p,\Lambda)>0$, $b=b(N,\alpha,p)>1$, and $E_{j}$
is defined as in (\ref{eq:sopralivello}) with the current definition
of $\widehat{k}_{j}$. Note that, in the previous computation, we
have also used that $\eta_{j},\varphi_{j}\leq1$. Arguing as in \foreignlanguage{british}{\eqref{eq:est:term1}}
and \eqref{eq:est:term2}, we can estimate the terms appearing in
the first two integrals of \foreignlanguage{british}{\eqref{eq:X_jplusone_after_energy}}
as follows:
\begin{align}
(u-\widehat{k}_{j})_{+}^{p} & \,\leq\,(u-\widehat{k}_{j})_{+}^{P}\,+\,\chi_{\{u\,>\,\widehat{k}_{j}\}}\,,\label{est:term1new}\\
(u^{\frac{\alpha+1}{2}}-\widehat{k}_{j}^{\frac{\alpha+1}{2}})_{+}^{2} & \,\leq\,c\,(u-\widehat{k}_{j})_{+}^{P}\,+\,c\,k^{\alpha+1}\,\chi_{\{u\,>\,\widehat{k}_{j}\}}\,,\label{est:term2new}
\end{align}
where we have also used that $\widehat{k}_{j}<k$ and $k\geq1$. As
for the integral involving $f$, by Hölder's inequality we have
\begin{align*}
\iint_{Q_{j}}|f|\,(u-\widehat{k}_{j})_{+}\,dx\,dt\, & \leq\,\Vert f\Vert_{L^{\sigma}(Q_{0})}\,\Big[\iint_{Q_{j}}(u-\widehat{k}_{j})_{+}^{\sigma'}\,dx\,dt\Big]^{\frac{1}{\sigma'}}\\
 & \leq\,\Vert f\Vert_{L^{\sigma}(Q_{0})}\,\Big[\Vert u_{+}\Vert_{L^{\infty}(Q_{0})}^{\max\,\{0,\sigma'-P\}}\iint_{Q_{j}}(u-\widehat{k}_{j})_{+}^{\min\,\{\sigma',P\}}\,dx\,dt\Big]^{\frac{1}{\sigma'}}\\
 & =\,\Vert f\Vert_{L^{\sigma}(Q_{0})}\,\Vert u_{+}\Vert_{L^{\infty}(Q_{0})}^{\max\,\left\{ 0,1\,-\,\frac{P}{\sigma'}\right\} }\Big[\iint_{Q_{j}}(u-\widehat{k}_{j})_{+}^{\min\,\{\sigma',P\}}\,dx\,dt\Big]^{\frac{1}{\sigma'}}\\
 & \leq\,\Vert f\Vert_{L^{\sigma}(Q_{0})}\,\Vert u_{+}\Vert_{L^{\infty}(Q_{0})}^{\max\,\left\{ 0,1\,-\,\frac{P}{\sigma'}\right\} }\Big[\iint_{Q_{j}}\left((u-\widehat{k}_{j})_{+}^{P}\,+\,\chi_{\{u\,>\,\widehat{k}_{j}\}}\right)dx\,dt\Big]^{\frac{1}{\sigma'}}\\
 & \leq\,\Vert f\Vert_{L^{\sigma}(Q_{0})}\,\Vert u_{+}\Vert_{L^{\infty}(Q_{0})}^{\max\,\left\{ 0,1\,-\,\frac{P}{\sigma'}\right\} }\,(X_{j}\,+\,\vert E_{j}\vert)^{\frac{1}{\sigma'}}\,,
\end{align*}
with the standard convention that 
\[
\Vert u_{+}\Vert_{L^{\infty}(Q_{0})}^{\max\,\left\{ 0,1\,-\,\frac{P}{\sigma'}\right\} }=1
\]
when both $\Vert u_{+}\Vert_{L^{\infty}(Q_{0})}$ and $\max\,\{0,1-P/\sigma'\}$
are equal to zero\footnote{However, when $\Vert u_{+}\Vert_{L^{\infty}(Q_{0})}=0$ we have $u\leq0$
almost everywhere in $Q_{r}(z_{o})$, and therefore estimate \eqref{eq:sup-bound-explicit-slow-diff}
is trivially satisfied.}. We note that, in the penultimate line of the previous estimate,
one may need to apply Young's inequality to increase the exponent
of $(u-\widehat{k}_{j})_{+}$. Moreover, arguing as in \eqref{eq:superlevel},
we see that
\begin{align}
|E_{j}|\,\leq\,c\,2^{jP}\,k^{-P}X_{j}\,,\label{est:superlevel2}
\end{align}
and therefore
\begin{align}
\iint_{Q_{j}}|f|\,(u-\widehat{k}_{j})_{+}\,dx\,dt\, & \leq\,c\,b^{j}\,\Vert f\Vert_{L^{\sigma}(Q_{0})}\,\Vert u_{+}\Vert_{L^{\infty}(Q_{0})}^{\max\,\left\{ 0,1\,-\,\frac{P}{\sigma'}\right\} }\,X_{j}^{1\,-\,\frac{1}{\sigma}},\label{est:term3final}
\end{align}
where the constants $c>0$ and $b>1$ now also depend on $\sigma$.
We now use \eqref{bound:grad-eta}, \eqref{bound:der-varphi}, \eqref{est:term1new}
and \eqref{est:term2new} in the estimate \foreignlanguage{british}{\eqref{eq:X_jplusone_after_energy}}
for $X_{j+1}$, and handle the resulting measures of the superlevel
sets by \eqref{est:superlevel2}. Then, taking \eqref{est:term3final}
and the fact that $k\geq1$ into account as well, we obtain
\begin{align}
X_{j+1}\,\leq\,c\,b^{j}\Big[\xi(p,r,\nu)\,X_{j}\,+\,\Vert f\Vert_{L^{\sigma}(Q_{0})}\,\Vert u_{+}\Vert_{L^{\infty}(Q_{0})}^{\max\,\left\{ 0,1\,-\,\frac{P}{\sigma'}\right\} }X_{j}^{1\,-\,\frac{1}{\sigma}}\Big]^{\frac{N+p}{\gamma N}}k^{-P(1\,-\,\frac{1}{\gamma})}X_{j}^{1\,-\,\frac{1}{\gamma}},\label{est:X_jplusone}
\end{align}
where $c=c(N,\alpha,p,\Lambda,\sigma,\delta_{\max})>0$, and we have
set
\begin{equation}
\xi(p,r,\nu):=\,r^{-p}(1-\nu)^{-p}+1\,.\label{eq:def:xi}
\end{equation}

\noindent Observing that
\begin{align*}
X_{j}\,=\,X_{j}^{\frac{1}{\sigma}}\,X_{j}^{1\,-\,\frac{1}{\sigma}}\leq\,\Vert u_{+}\Vert_{L^{P}(Q_{0})}^{\frac{P}{\sigma}}\,X_{j}^{1\,-\,\frac{1}{\sigma}},
\end{align*}
from (\ref{est:X_jplusone}) we deduce
\begin{align*}
X_{j+1}\,\leq\,C_{u,f}\,b^{j}\,X_{j}^{1\,+\,\beta}\,,
\end{align*}
where
\begin{align*}
C_{u,f}\,=\,c\Big[\xi(p,r,\nu)\,\Vert u_{+}\Vert_{L^{P}(Q_{0})}^{\frac{P}{\sigma}}\,+\,\Vert f\Vert_{L^{\sigma}(Q_{0})}\,\Vert u_{+}\Vert_{L^{\infty}(Q_{0})}^{\max\,\left\{ 0,1\,-\,\frac{P}{\sigma'}\right\} }\Big]^{\frac{N+p}{\gamma N}}k^{-P(1\,-\,\frac{1}{\gamma})}
\end{align*}
and
\begin{align*}
\beta\,=\,\frac{1}{\gamma N}\Big(p\,-\,\frac{N+p}{\sigma}\Big)>0\,,
\end{align*}
due to the lower bound for $\sigma$ in (\ref{assumpt:f}). If $C_{u,f}=0$,
then $u_{+}=0$ almost everywhere in $Q_{r}(z_{o})$ and estimate
\eqref{eq:sup-bound-explicit-slow-diff} is trivially satisfied. If
instead $C_{u,f}>0$, then, by Lemma \ref{lem:fastconvg}, the sequence
$\{X_{j}\}$ converges to zero provided that 
\begin{align}
\iint_{Q_{0}}u_{+}^{P}\,dx\,dt\,=\,X_{0}\,\leq\,C_{u,f}^{-\,\frac{1}{\beta}}\,b^{-\,\frac{1}{\beta^{2}}},\label{eq:da_garantire}
\end{align}
which can equivalently be stated as 
\begin{align}
k\,\geq\,c\,\Vert u_{+}\Vert_{L^{P}(Q_{0})}^{\frac{\beta\gamma}{\gamma-1}}\Big[\xi(p,r,\nu)\,\Vert u_{+}\Vert_{L^{P}(Q_{0})}^{\frac{P}{\sigma}}\,+\,\Vert f\Vert_{L^{\sigma}(Q_{0})}\,\Vert u_{+}\Vert_{L^{\infty}(Q_{0})}^{\max\,\left\{ 0,1\,-\,\frac{P}{\sigma'}\right\} }\Big]^{\frac{N+p}{NP(\gamma-1)}}.\label{k:lowerbound1}
\end{align}
Estimating the right-hand side of \eqref{k:lowerbound1} from above
using the elementary inequality
\begin{equation}
(v+w)^{\tau}\,\leq\,2^{\tau}(v^{\tau}+w^{\tau})\,,\,\,\,\,\,\,\,\,\,\,\,\,v,w\,\geq\,0\,,\,\,\tau>0\,,\label{eq:elem_ineq}
\end{equation}
we see that \eqref{k:lowerbound1} holds, for instance, if 
\begin{align}
k\,\geq\,c\,\Vert u_{+}\Vert_{L^{P}(Q_{0})}^{\frac{\beta\gamma}{\gamma-1}}\,\Vert f\Vert_{L^{\sigma}(Q_{0})}^{\frac{N+p}{NP(\gamma-1)}}\,\Vert u_{+}\Vert_{L^{\infty}(Q_{0})}^{\theta}\,+\,c\,\,\xi(p,r,\nu)^{\frac{N+p}{NP(\gamma-1)}}\,\Vert u_{+}\Vert_{L^{P}(Q_{0})}^{\frac{p}{N(\gamma-1)}}\,,\label{k:lowerbound2}
\end{align}
where 
\begin{align*}
\theta\,=\,\max\,\left\{ 0,\left(1-\frac{P}{\sigma'}\right)\frac{N+p}{NP(\gamma-1)}\right\} .
\end{align*}
Thus, if $k\geq1$ and if \eqref{k:lowerbound2} holds, then 
\begin{align}
\iint_{Q_{\nu r}(z_{o})}(u-k)_{+}^{P}\,dx\,dt\,\leq\,X_{j}\to0\,\,\,\,\,\,\,\,\mathrm{as}\,\,j\to\infty\,,\label{eq:conv_bigp}
\end{align}
and hence $u\leq k\,$ a.e. in $Q_{\nu r}(z_{o})$.\\
\foreignlanguage{british}{$\hspace*{1em}$}Let us first consider the
case $\sigma'\leq P$. Then $\theta=0$ and the $L^{\infty}$-norm
of $u_{+}$ does not appear on the right-hand side of \eqref{k:lowerbound2}.
Therefore, observing that $\frac{\beta\gamma}{\gamma-1}<\frac{p}{N(\gamma-1)}$,
we can estimate the first term on the right-hand side of \eqref{k:lowerbound2}
upwards using Young's inequality. We thus conclude that a sufficient
condition for (\ref{eq:da_garantire}) to hold is that
\begin{align}
k\,\geq\,c\,\Vert f\Vert_{L^{\sigma}(Q_{0})}^{\frac{\sigma p}{NP(\gamma-1)}}\,+\,c\,\,\xi(p,r,\nu)^{\frac{N+p}{NP(\gamma-1)}}\,\Vert u_{+}\Vert_{L^{P}(Q_{0})}^{\frac{p}{N(\gamma-1)}}\,+1\,.\label{eq:KKK}
\end{align}
Using the definition of $\gamma$ in (\ref{def:gamma}) to rewrite
the right-hand side of (\ref{eq:KKK}), and noting that 
\begin{align*}
\min\,\left\{ 0,\frac{\sigma}{P'}-1\right\} =0\,\,\,\,\,\,\,\,\,\,\mathrm{when}\,\,\,\sigma'\leq P,
\end{align*}
we conclude that estimate \eqref{eq:sup-bound-explicit-slow-diff}
holds.

\selectlanguage{british}%
\noindent $\hspace*{1em}$\foreignlanguage{american}{Consider now
the case $\sigma'>P$. Then (\ref{k:lowerbound2}), (\ref{eq:conv_bigp})
and the fact that $k\ge1$ imply the estimate\begin{align}\label{eq:est:recursive1}
\underset{Q_{\nu r}(z_{o})}{\mathrm{ess}\,\sup}\,\,u_{+}\,&\leq\,c\,\Big(\underset{Q_{r}(z_{o})}{\mathrm{ess}\,\sup}\,\,u_{+}\Big)^{\theta}\,\Vert u_{+}\Vert_{L^{P}(Q_{r}(z_{o}))}^{\frac{\beta\gamma}{\gamma-1}}\,\Vert f\Vert_{L^{\sigma}(Q_{r}(z_{o}))}^{\frac{N+p}{NP(\gamma-1)}}\nonumber\\
&\,\,\,\,\,\,\,+\,c\,\,\xi(p,r,\nu)^{\frac{N+p}{NP(\gamma-1)}}\,\Vert u_{+}\Vert_{L^{P}(Q_{r}(z_{o}))}^{\frac{p}{N(\gamma-1)}}+1\,,
\end{align}where $\theta>0$. We now show that $\theta<1$ also holds. Using
the definition of $\gamma$ in \eqref{def:gamma}, we see that\begin{align*}
\theta\,&=\,\Big(1-\frac{P}{\sigma'}\Big)\frac{N+p}{NP\gamma-NP}\,<\,\Big(1-\frac{NP}{N+p}\Big)\frac{N+p}{NP\gamma-NP}\\
&=\,\Big(1-\frac{NP}{N+p}\Big)\frac{N+p}{p(N+\alpha+1)-NP}\,,
\end{align*}where we have used the fact that $\sigma'<(N+p)/N$. By definition,
we have $P\in\{p,\alpha+1\}$. If $P=p$, we obtain
\begin{align*}
\theta\,<\,\Big(1-\frac{Np}{N+p}\Big)\frac{N+p}{p(\alpha+1)}\,=\,\frac{N+p-Np}{p(\alpha+1)}\,<\,\frac{1}{\alpha+1}\,<\,1\,.
\end{align*}
If instead $P=\alpha+1$, we have\begin{align*}
\theta\, & <\Big(1-\frac{N(\alpha+1)}{N+p}\Big)\frac{N+p}{p(N+\alpha+1)-N(\alpha+1)}\,=\,\frac{N+p-N(\alpha+1)}{p(N+\alpha+1)-N(\alpha+1)}\\
&\leq\,\frac{N+\alpha+1-N(\alpha+1)}{p(N+\alpha+1)-N(\alpha+1)}\,<\,1\,,
\end{align*}where the last inequality follows from the assumption $p>1$. At this
point, we can apply Young's inequality with exponents $1/\theta$
and $1/(1-\theta)$ to the first term on the right-hand side of \eqref{eq:est:recursive1},
thus obtaining\begin{align}\label{eq:est:recursive2}
\underset{Q_{\nu r}(z_{o})}{\mathrm{ess}\,\sup}\,\,u_{+}\,&\leq\,\varepsilon\,\,\,\underset{Q_{r}(z_{o})}{\mathrm{ess}\,\sup}\,\,u_{+}\,+\,\tilde{c}(\varepsilon)\,\Vert u_{+}\Vert_{L^{P}(Q_{r}(z_{o}))}^{\frac{\beta\gamma}{(\gamma-1)\,(1-\theta)}}\,\Vert f\Vert_{L^{\sigma}(Q_{r}(z_{o}))}^{\frac{N+p}{NP\,(\gamma-1)\,(1-\theta)}}\nonumber\\
&\,\,\,\,\,\,\,+\,c\,\,\xi(p,r,\nu)^{\frac{N+p}{NP(\gamma-1)}}\,\Vert u_{+}\Vert_{L^{P}(Q_{r}(z_{o}))}^{\frac{p}{N(\gamma-1)}}+1\,,
\end{align}where $\varepsilon>0$ will be chosen later. Here we emphasize that
the positive constant $\tilde{c}(\varepsilon)$ depends on $\varepsilon$,
in addition to $N$, $\alpha$, $p$, $\Lambda$, $\sigma$ and $\max\,\{\delta_{1},\ldots,\delta_{N}\}$.
In order to get a more manageable estimate, we will use Young's inequality
again and the fact that the exponent of $\Vert u_{+}\Vert_{L^{P}(Q_{r}(z_{o}))}$
in the second term on the right-hand side of \eqref{eq:est:recursive2}
is smaller than the exponent in the third term. To verify this, note
that
\begin{align*}
\frac{\beta\gamma}{(\gamma-1)(1-\theta)}\,=\,\frac{a-N}{a+N(\gamma-1)-\frac{N\,+\,p}{P}}\,=:\,g(a),\qquad\mathrm{with}\,\,\,a\,=\,\frac{N+p}{\sigma'}\,\in\left(N,\frac{N+p}{P}\right).
\end{align*}
We can rewrite $g(a)$ as follows: 
\begin{align*}
g(a)\,=\,1-\,\frac{(N\gamma-\frac{N\,+\,p}{P})}{a+N(\gamma-1)-\frac{N\,+\,p}{P}}\,.
\end{align*}
The numerator in the last expression can be expressed as 
\begin{align*}
N\gamma-\,\frac{N+p}{P}\,=\,\frac{1}{P}\,(NP\gamma-N-p)\,=\,\frac{1}{P}\left[p(N+\alpha+1)-N-p\right]\,>\,0\,.
\end{align*}
Thus, $g$ is an increasing function and 
\begin{align*}
\frac{\beta\gamma}{(\gamma-1)(1-\theta)}\,=\,g(a)\,<\,g\left(\frac{N+p}{P}\right)\,=\,\frac{\frac{p}{P}+N(\frac{1}{P}-1)}{N(\gamma-1)}\,<\,\frac{1}{N(\gamma-1)}\,<\,\frac{p}{N(\gamma-1)}\,.
\end{align*}
Therefore, we can apply Young's inequality with exponents $Q=p(1-\theta)/(N\beta\gamma)$
and $Q'$ to the second term on the right-hand side of \eqref{eq:est:recursive2}
to obtain\begin{align}\label{eq:est:recursive3}
\underset{Q_{\nu r}(z_{o})}{\mathrm{ess}\,\sup}\,\,u_{+}\,&\leq\,\varepsilon\,\,\,\underset{Q_{r}(z_{o})}{\mathrm{ess}\,\sup}\,\,u_{+}\,+\,\tilde{c}(\varepsilon)\,\xi(p,r,\nu)^{\frac{N+p}{NP(\gamma-1)}}\,\Vert u_{+}\Vert_{L^{P}(Q_{r}(z_{o}))}^{\frac{p}{N(\gamma-1)}}\nonumber\\
&\,\,\,\,\,\,\,+\,\tilde{c}(\varepsilon)\,\Vert f\Vert_{L^{\sigma}(Q_{r}(z_{o}))}^{\frac{p(N+p)}{NP\,(\gamma-1)\,[p(1-\theta)-N\beta\gamma]}}\,+1\,,
\end{align}where we have also used that $\xi(p,r,\nu)>1$. Now, let $s\in(0,1)$
and define an increasing sequence of radii as follows:
\begin{align*}
\rho_{j}\,:=\,sr+r(1-s)(1-2^{-j})\,,\,\,\,\,\,\,\,\,\,\,\,j\in\mathbb{N}_{0}\,.
\end{align*}
Note that $\rho_{0}=sr$ and $\rho_{j}\to r$ as $j\rightarrow\infty$.
We will apply \eqref{eq:est:recursive3} with $\rho_{j+1}$ in place
of $r$ and $\nu=\nu_{j}:=\frac{\rho_{j}}{\rho_{j+1}}$. First, we
see that 
\begin{align*}
\xi(p,\rho_{j+1},\nu_{j})\,=\,[\rho_{j+1}(1-\nu_{j})]^{-p}+1\,=\,[r(1-s)\,2^{-(j+1)}]^{-p}+1\,\leq\,\xi(p,r,s)\,2^{p(j+1)}.
\end{align*}
Therefore, from \eqref{eq:est:recursive3} we deduce\begin{align*}
M_{j}:=\,\underset{Q_{\rho_{j}}(z_{o})}{\mathrm{ess}\,\sup}\,\,u_{+}\,&\leq\,\varepsilon\,M_{j+1}\,+\,\tilde{c}(\varepsilon)\,\mathsf{b}^{j}\,\xi(p,r,s)^{\frac{N+p}{NP\,(\gamma-1)}}\,\Vert u_{+}\Vert_{L^{P}(Q_{r}(z_{o}))}^{\frac{p}{N(\gamma-1)}}\\
&\,\,\,\,\,\,\,+\,\tilde{c}(\varepsilon)\,\Vert f\Vert_{L^{\sigma}(Q_{r}(z_{o}))}^{\frac{p(N+p)}{NP\,(\gamma-1)\,[p(1-\theta)-N\beta\gamma]}}\,+1\\
&=:\,\varepsilon\,M_{j+1}\,+\,\tilde{c}(\varepsilon)\,\mathsf{b}^{j}\,C_{u}\,+\,\tilde{c}(\varepsilon)\,C_{f}\,+\,1\,,
\end{align*}for some constant $\mathsf{b}>1$ depending only on the data. Iterating
the previous estimate, we end up with
\begin{align}
\underset{Q_{sr}(z_{o})}{\mathrm{ess}\,\sup}\,\,u_{+}\,=\,M_{0}\,\leq\,\varepsilon^{n}M_{n}\,+\,\tilde{c}(\varepsilon)\,C_{u}\,\sum_{j=0}^{n-1}(\varepsilon\mathsf{b})^{j}\,+\,[\tilde{c}(\varepsilon)\,C_{f}+1]\,\sum_{j=0}^{n-1}\varepsilon^{j}\,,\label{est:iterated}
\end{align}
for any $n\in\mathbb{N}$. Now choose $\varepsilon=1/(2\mathsf{b})\in(0,1)$.
Then the sums in the preceding inequality converge as $n\to\infty$.
Since the sequence $\{M_{n}\}$ is bounded due to the local boundedness
of $u$, the first term on the right-hand side of \eqref{est:iterated}
tends to zero in the limit as $n\to\infty$. Hence, letting $n\to\infty$
in \eqref{est:iterated}, we obtain
\begin{align*}
\underset{Q_{sr}(z_{o})}{\mathrm{ess}\,\sup}\,\,u_{+}\,\leq\,c\left(\xi(p,r,s)^{\frac{N+p}{NP\,(\gamma-1)}}\,\Vert u_{+}\Vert_{L^{P}(Q_{r}(z_{o}))}^{\frac{p}{N(\gamma-1)}}\,+\,\Vert f\Vert_{L^{\sigma}(Q_{r}(z_{o}))}^{\frac{p(N+p)}{NP\,(\gamma-1)\,[p(1-\theta)-N\beta\gamma]}}\,+1\right).
\end{align*}
Finally, by lengthy but elementary computations involving the definitions
of $\gamma$, $\beta$ and $\theta$, we can rewrite the exponents
appearing in the last expression to confirm the validity of \eqref{eq:sup-bound-explicit-slow-diff}.
Here we also exploit the fact that, for $\sigma'>P$,
\begin{align*}
\min\,\left\{ 0,\frac{\sigma}{P'}-1\right\} =\,\frac{\sigma}{P'}-1\,.
\end{align*}
This concludes the proof.\end{proof}}

\begin{singlespace}
\noindent $\hspace*{1em}$\foreignlanguage{american}{By iterating
the estimates of Theorem \ref{thm:loc_bdd_p_large_LP}, we can actually
derive sharper bounds for the essential supremum and essential infimum
of $u$. We first note that, by the definition of $P$ and inequality
\eqref{p_lower_bnd}, we have
\[
\lambda_{o}(N,\alpha,p):=\,P-\,\frac{N}{p}\left[p\left(1+\,\frac{\alpha+1}{N}\right)-P\right]\,<\,P\,.
\]
Moreover, from the definitions of $\lambda_{o}(N,\alpha,p)$ and $P$,
we see that
\[
\lambda_{o}(N,\alpha,p)\,=\,P-(\alpha+1)+N\left(\frac{P}{p}\,-1\right)\geq\,0\,.
\]
Now we are in a position to prove the following result.}
\end{singlespace}
\selectlanguage{american}%
\begin{thm}
\label{thm:better_est} Let $u\in\mathcal{A}(\alpha,p,\{\delta_{i}\},\Omega_{T},f,\mathcal{C})$,
where $\alpha\in(0,\infty)$, $p\in(1,\infty)$ and $f$ satisfies
$(\ref{assumpt:f})$. Suppose that \eqref{p_lower_bnd} holds and
let $\lambda\in(\lambda_{o}(N,\alpha,p),P]$. Then, for every cylinder
$Q_{r}(z_{o})\Subset\Omega_{T}$ and every $\nu\in(0,1)$, we have\begin{align}\label{eq:est:loc_bdd_best}
\underset{Q_{\nu r}(z_{o})}{\mathrm{ess}\,\sup}\,\,u\,&\leq\,c\left(\big[1+(r(1-\nu))^{-(N+p)}\big]\iint_{Q_{r}(z_{o})}u_{+}^{\lambda}\,dx\,dt\right)^{\frac{1}{\lambda\,-\,\lambda_{o}(N,\alpha,p)}}\nonumber\\
&\,\,\,\,\,\,\,+\,c\left(1+\iint_{Q_{r}(z_{o})}|f|^{\sigma}\,dx\,dt\right)^{\frac{p}{p(N+\alpha+1)\,-\,NP\,+\,pP\min\,\{0,\frac{\sigma}{P'}-1\}}},
\end{align}for some positive constant $c$ depending only on $N$, $\alpha$,
$p$, $\Lambda$, $\sigma$, $\max\,\{\delta_{1},\ldots,\delta_{N}\}$
and $\lambda$. An analogous lower bound holds for the essential infimum,
with $u_{-}$ replacing $u_{+}$ on the right-hand side of \foreignlanguage{british}{\eqref{eq:est:loc_bdd_best}}.
In particular, under assumptions $\eqref{eq:coeff_limit}$ and $(\ref{assumpt:f})$,
these estimates hold for all weak solutions to \eqref{eq:diffusion}
in the sense of Definition \ref{def:weaksol}.
\end{thm}

\noindent \begin{proof}[\bfseries{Proof}]Let $\lambda$ be as in
the statement of the theorem. Note that the case $\lambda=P$ is already
contained in Theorem \ref{thm:loc_bdd_p_large_LP}, so henceforth
we assume that $\lambda<P$. We define
\[
r_{j}:=\,\nu r\,+(1-\nu)\,r\,(1-2^{-j})\,,\hspace{7mm}Q_{j}:=\,Q_{r_{j}}(z_{o})\,,\,\,\,\,\,\,\,\,\,\,\,\,\mathrm{for}\,\,j\in\mathbb{N}_{0}\,,
\]
and
\[
K_{f}:=\left(1+\iint_{Q_{r}(z_{o})}|f|^{\sigma}\,dx\,dt\right)^{\frac{p}{p(N+\alpha+1)\,-\,NP\,+\,pP\min\,\{0,\frac{\sigma}{P'}-1\}}}.
\]
Using \eqref{eq:sup-bound-explicit-slow-diff} with $r=r_{j+1}$,
$s=\frac{r_{j}}{r_{j+1}}$, and noting that $\frac{p}{p(N\,+\,\alpha\,+\,1)\,-\,PN}=\frac{1}{P\,-\,\lambda_{o}}$,
we obtain\begin{align*}
M_{j}&:=\,\underset{Q_{j}}{\mathrm{ess}\,\sup}\,\,u_{+}\\
&\leq\,c\left(\big[1+(r(1-\nu))^{-(N+p)}\,2^{(N+p)(j+1)}\big]\iint_{Q_{j+1}}u_{+}^{P}\,dx\,dt\right)^{\frac{1}{P\,-\,\lambda_{o}(N,\alpha,p)}}+c\,K_{f}\\
&\leq\,c\,b^{j}\,M_{j+1}^{\frac{P\,-\,\lambda}{P\,-\,\lambda_{o}(N,\alpha,p)}}\left(\big[1+(r(1-\nu))^{-(N+p)}\big]\iint_{Q_{j+1}}u_{+}^{\lambda}\,dx\,dt\right)^{\frac{1}{P\,-\,\lambda_{o}(N,\alpha,p)}}+c\,K_{f}\,,
\end{align*}where $c$ is a positive constant depending only on $N$, $\alpha$,
$p$, $\Lambda$, $\sigma$ and $\max\,\{\delta_{1},\ldots,\delta_{N}\}$,
while $b>1$ depends only on $N$, $\alpha$ and $p$. Since $\lambda_{o}(N,\alpha,p)<\lambda<P$,
we see that the exponent of $M_{j+1}$ in the last expression belongs
to the interval $(0,1)$. Thus we can apply Young's inequality with
$\varepsilon>0$ to obtain
\begin{equation}
M_{j}\,\le\,\varepsilon\,M_{j+1}\,+\,\tilde{c}(\varepsilon,\lambda)\,\tilde{b}^{j}\left(\big[1+(r(1-\nu))^{-(N+p)}\big]\iint_{Q_{r}(z_{o})}u_{+}^{\lambda}\,dx\,dt\right)^{\frac{1}{\lambda\,-\,\lambda_{o}(N,\alpha,p)}}+c\,K_{f}\,,\label{eq:iteranda_est}
\end{equation}
where we have also used the fact that $Q_{j+1}$ is contained in $Q_{r}(z_{o})$.
Here we emphasize that the positive constant $\tilde{c}(\varepsilon,\lambda)$
depends on $\varepsilon$ and $\lambda$, in addition to $N$, $\alpha$,
$p$, $\Lambda$, $\sigma$ and $\max\,\{\delta_{1},\ldots,\delta_{N}\}$.
The constant $\tilde{b}>1$ depends only on $N$, $\alpha$, $p$
and $\lambda$. Proceeding as in the proof of Theorem \ref{thm:loc_bdd_p_large_LP},
we now iterate (\ref{eq:iteranda_est}) and choose $\varepsilon=1/(2\tilde{b})$.
This eventually yields the following estimate:
\[
\underset{Q_{\nu r}(z_{o})}{\mathrm{ess}\,\sup}\,\,u_{+}\,\leq\,C\left(\big[1+(r(1-\nu))^{-(N+p)}\big]\iint_{Q_{r}(z_{o})}u_{+}^{\lambda}\,dx\,dt\right)^{\frac{1}{\lambda\,-\,\lambda_{o}(N,\alpha,p)}}+C\,K_{f}\,,
\]
where $C$ is a positive constant depending only on $N$, $\alpha$,
$p$, $\Lambda$, $\sigma$, $\max\,\{\delta_{1},\ldots,\delta_{N}\}$
and $\lambda$. This bound clearly implies \foreignlanguage{british}{\eqref{eq:est:loc_bdd_best}}.\\
\foreignlanguage{british}{$\hspace*{1em}$}Finally, with analogous
arguments, we also get a lower bound for the essential infimum of
$u$ in terms of the $L^{\lambda}$-norm of $u_{-}\,$.\end{proof}

\subsection{The case $p\protect\leq\frac{N(\alpha\,+\,1)}{N\,+\,\alpha\,+\,1}$}

\selectlanguage{british}%
$\hspace*{1em}$\foreignlanguage{american}{We now turn our attention
to the range
\begin{align}
p\,\leq\,\frac{N(\alpha+1)}{N+\alpha+1}\,,\label{range:small_p}
\end{align}
and recall that we also require (\ref{extra_integrability}) in this
case. As in the previous subsection, we will first obtain local boundedness
without an explicit bound (Theorem \ref{thm:p_small1} below). Once
this has been done, we can employ alternative methods to derive quantitative
estimates for both the essential supremum and the essential infimum
(see Theorem \ref{thm:p_small2}).}
\selectlanguage{american}%
\begin{thm}
\label{thm:p_small1} Let $u\in\mathcal{A}(\alpha,p,\{\delta_{i}\},\Omega_{T},f,\mathcal{C})$,
where $\alpha\in(0,\infty)$, $p\in(1,\infty)$ and $f$ satisfies
$(\ref{assumpt:f})$. Suppose that \eqref{range:small_p} holds and
that $u$ satisfies the extra integrability condition $(\ref{extra_integrability})$.
Then $u$ is locally essentially bounded. In particular, if $u$ is
a weak solution of \eqref{eq:diffusion} under assumptions $\eqref{eq:coeff_limit}$,
$(\ref{assumpt:f})$ and $(\ref{extra_integrability})$, then $u$
belongs to $L_{\mathrm{loc}}^{\infty}(\Omega_{T})$.
\end{thm}

\noindent \begin{proof}[\bfseries{Proof}]Let $u\in\mathcal{A}(\alpha,p,\{\delta_{i}\},\Omega_{T},f,\mathcal{C})$,
with $\alpha$, $p$ and $f$ as in the statement of the theorem.
We only prove that $u$ is locally bounded from above, since the argument
yielding a lower bound is analogous. We choose radii $r_{j}$, space-time
cylinders $Q_{j}$ and cut-off functions $\eta_{j}$ and $\varphi_{j}$
exactly as in the proof of Theorem \ref{theo:local_bdd_large_bar_p}.
Moreover, we define the sequences $Y_{j}$, $k_{j}$ and $\widehat{k}_{j}$
as in \eqref{def:k_j_alpha_small}. As before, we obtain estimate
\foreignlanguage{british}{\eqref{eq:integ_after_energy}}. The second
integral on the right-hand side of \foreignlanguage{british}{\eqref{eq:integ_after_energy}}
can no longer be handled as in the proof of Theorem \ref{theo:local_bdd_large_bar_p},
since we are now in the regime $p_{\alpha+1}\leq\alpha+1$ due to
(\ref{range:small_p}). However, in the special case $p_{\alpha+1}=\alpha+1$
we can directly write 
\begin{align}
\iint_{Q_{j}}(u^{\frac{\alpha+1}{2}}-\widehat{k}_{j}^{\frac{\alpha+1}{2}})_{+}^{2}\,dx\,dt\,\leq & \iint_{Q_{j}}u^{p_{\alpha+1}}\,\chi_{\{u\,>\,\widehat{k}_{j}\}}\,dx\,dt\nonumber \\
\leq & \,\,c\iint_{Q_{j}}\left[(u-\widehat{k}_{j})_{+}^{p_{\alpha+1}}\,+\,\widehat{k}_{j}^{p_{\alpha+1}}\,\chi_{\{u\,>\,\widehat{k}_{j}\}}\right]dx\,dt\nonumber \\
\leq & \,\,c\,2^{j\,p_{\alpha+1}}\,Y_{j}\,,\label{eq:est:problematic_term_special_case}
\end{align}
where $c=c(N,\alpha,p)>0$ and, in the last line, we have used (\ref{eq:superlevel})
along with $\widehat{k}_{j}<2k$. If instead $p_{\alpha+1}<\alpha+1$,
we introduce the parameters 
\begin{align*}
q:=\,\frac{m-p_{\alpha+1}}{\alpha+1-p_{\alpha+1}}>1\,,\quad\quad\vartheta:=\,\frac{m(\alpha+1)-m\,p_{\alpha+1}}{m(\alpha+1)-(\alpha+1)\,p_{\alpha+1}}\,\in(0,1)\,.
\end{align*}
The ranges for $q$ and $\vartheta$ can be easily verified by observing
that
\begin{align}
p_{\alpha+1}\,<\,\alpha+1\,<\,m\,,\label{m-ineq}
\end{align}
where the last inequality is obtained by suitably combining (\ref{range:small_p})
with the lower bound for $m$ in \eqref{extra_integrability}. Using
the same lower bound, we also see that 
\begin{align}
q\,=\,1\,+\,\frac{m-(\alpha+1)}{\alpha+1-p_{\alpha+1}}\,>\,1\,+\,\frac{\frac{N}{p}(\alpha+1-p)-(\alpha+1)}{\alpha+1-p_{\alpha+1}}\,=\,\frac{N+p}{p}\,.\label{range:q}
\end{align}

\noindent Thus, we can use Hölder's inequality to estimate\begin{align}\label{eq:est:problematic_term_general}
&\iint_{Q_{j}}(u^{\frac{\alpha+1}{2}}-\widehat{k}_{j}^{\frac{\alpha+1}{2}})_{+}^{2}\,dx\,dt\,\leq\iint_{Q_{j}}u^{\alpha+1}\,\chi_{\{u\,>\,\widehat{k}_{j}\}}\,dx\,dt\,=\iint_{Q_{j}}u^{(\alpha+1)\vartheta}\,u^{(\alpha+1)(1-\vartheta)}\,\chi_{\{u\,>\,\widehat{k}_{j}\}}\,dx\,dt\nonumber\\
&\,\,\,\,\,\,\,\leq\Big[\iint_{Q_{j}}u^{(\alpha+1)\vartheta q}\,\chi_{\{u\,>\,\widehat{k}_{j}\}}\,dx\,dt\Big]^{\frac{1}{q}}\,\Big[\iint_{Q_{j}}u^{(\alpha+1)(1-\vartheta)q'}\,\chi_{\{u\,>\,\widehat{k}_{j}\}}\,dx\,dt\Big]^{\frac{1}{q'}}\nonumber\\
&\,\,\,\,\,\,\,=\Big[\iint_{Q_{j}}u^{m}\,\chi_{\{u\,>\,\widehat{k}_{j}\}}\,dx\,dt\Big]^{\frac{1}{q}}\,\Big[\iint_{Q_{j}}u^{p_{\alpha+1}}\,\chi_{\{u\,>\,\widehat{k}_{j}\}}\,dx\,dt\Big]^{1\,-\,\frac{1}{q}}\nonumber\\
&\,\,\,\,\,\,\,\leq\,c\,2^{j\,p_{\alpha+1}}\Big[\iint_{Q_{0}}u_{+}^{m}\,dx\,dt\Big]^{\frac{1}{q}}\,Y_{j}^{1\,-\,\frac{1}{q}}\,,
\end{align}where, in the last line, we have used the last two inequalities in
(\ref{eq:est:problematic_term_special_case}). Note that the constant
$c$ now also depends on $m$. Taking into account (\ref{eq:est:problematic_term_special_case}),
\eqref{range:q} and \foreignlanguage{british}{\eqref{eq:est:problematic_term_general}},
we deduce that in the full range $p_{\alpha+1}\leq\alpha+1$ we have
\[
\iint_{Q_{j}}(u^{\frac{\alpha+1}{2}}-\widehat{k}_{j}^{\frac{\alpha+1}{2}})_{+}^{2}\,dx\,dt\,\leq\,C_{u}\,2^{j\,p_{\alpha+1}}\,Y_{j}^{1\,-\,\frac{1}{q}}\,,
\]
where $C_{u}$ is a positive constant that may depend on $N$, $\alpha$,
$p$, $m$ and $\Vert u_{+}\Vert_{L^{m}(Q_{r}(z_{o}))}$, and where
\[
\frac{N+p}{p}\,<\,q\,\leq\,\infty\,,
\]
with $q=\infty$ if $p_{\alpha+1}=\alpha+1$ (in which case $\frac{1}{q}$
is understood to be zero). For the other terms on the right-hand side
of \foreignlanguage{british}{\eqref{eq:integ_after_energy}}, we can
use the same estimates as in the proof of Theorem \ref{theo:local_bdd_large_bar_p},
which leads to 
\begin{align*}
Y_{j+1} & \,\leq\,C_{u}\,b^{j}\Big[Y_{j}\,+\,Y_{j}^{1\,-\,\frac{1}{q}}\,+\,\Vert f\Vert_{L^{\sigma}(Q_{0})}\,Y_{j}^{1\,-\,\frac{1}{\sigma}}\Big]^{\frac{N+p}{N}},
\end{align*}
where $b>1$ is a constant depending only on $N$, $\alpha$ and $p$,
while $C_{u}$ now also depends on $\Lambda$, $\sigma$, $\max\,\{\delta_{1},\ldots,\delta_{N}\}$,
$r$ and $\nu$. In order to have the same exponent for $Y_{j}$ in
all three terms inside the square brackets, we define the parameter
\begin{align*}
\omega:=\,\min\left\{ 1-\,\frac{1}{q}\,,1-\,\frac{1}{\sigma}\right\} \,\in\,\left(\frac{N}{N+p}\,,1\right),
\end{align*}
and observe that
\begin{align*}
Y_{j}\,=\,Y_{j}^{1\,-\,\omega}\,Y_{j}^{\omega}\,\leq\left(\iint_{Q_{0}}u_{+}^{p_{\alpha+1}}\,dx\,dt\right)^{1\,-\,\omega}Y_{j}^{\omega}
\end{align*}
and, similarly,
\[
Y_{j}^{\max\,\left\{ 1\,-\,\tfrac{1}{q}\,,1\,-\,\tfrac{1}{\sigma}\right\} }\,\leq\left(\iint_{Q_{0}}u_{+}^{p_{\alpha+1}}\,dx\,dt\right)^{\max\,\left\{ 1\,-\,\tfrac{1}{q}\,,1\,-\,\tfrac{1}{\sigma}\right\} \,-\,\omega}\,Y_{j}^{\omega}\,,
\]
whenever $Y_{j}\neq0$. Therefore, arguing as in \eqref{shtevensh}$-$\eqref{eq:rec_ineq},
we eventually obtain
\begin{align*}
Y_{j+1}\,\leq\,C_{u,f}\,b^{j}\,Y_{j}^{\frac{\omega(N+p)}{N}},
\end{align*}
where $C_{u,f}$ is a positive constant depending only on $N$, $\alpha$,
$p$, $\Lambda$, $\sigma$, $m$, $\max\,\{\delta_{1},\ldots,\delta_{N}\}$,
$r$, $\nu$, $\Vert u_{+}\Vert_{L^{p_{\alpha+1}}(Q_{r}(z_{o}))}$
and $\Vert f\Vert_{L^{\sigma}(Q_{r}(z_{o}))}$. Since 
\begin{align*}
\frac{\omega(N+p)}{N}>1\,,
\end{align*}
we can pursue the same reasoning as in the proof of Theorem \ref{theo:local_bdd_large_bar_p},
thereby concluding that $u\leq2k\,$ a.e. in $Q_{\nu r}(z_{o})$ for
some sufficiently large $k\geq1$. This implies that $u$ is essentially
bounded from above on compact subsets of $\Omega_{T}$ by a standard
covering argument.\end{proof}

\selectlanguage{british}%
\noindent $\hspace*{1em}$\foreignlanguage{american}{Having established
the local boundedness of functions $u\in\mathcal{A}(\alpha,p,\{\delta_{i}\},\Omega_{T},f,\mathcal{C})$
under condition (\ref{assumpt:f}) and the extra integrability assumption
(\ref{extra_integrability}), we now derive explicit estimates for
their essential supremum and infimum.}
\selectlanguage{american}%
\begin{thm}
\label{thm:p_small2}Let $u\in\mathcal{A}(\alpha,p,\{\delta_{i}\},\Omega_{T},f,\mathcal{C})$,
where $\alpha\in(0,\infty)$, $p\in(1,\infty)$ and $f$ satisfies
$(\ref{assumpt:f})$. Suppose that \eqref{range:small_p} holds and
that $u$ satisfies the extra integrability condition $(\ref{extra_integrability})$.
Then, for every cylinder $Q_{r}(z_{o})\Subset\Omega_{T}$ and every
$\nu\in(0,1)$, we have the explicit upper bound
\begin{align}
\underset{Q_{\nu r}(z_{o})}{\mathrm{ess}\,\sup}\,\,u\,\leq\,c\left([1+(r(1-\nu))^{-(N+p)}]\iint_{Q_{r}(z_{o})}u_{+}^{m}\,dx\,dt\,+\iint_{Q_{r}(z_{o})}|f|^{\sigma}\,dx\,dt\right)^{\frac{p}{\mu}}+c\,,\label{est:essup_small-p}
\end{align}
where
\[
\mu\,=\,mp\,-\,N(\alpha+1-p)\,+\,(N+p)\,\min\left\{ 0,\alpha-\frac{m}{\sigma}\right\} \,>\,0
\]
and $c$ is a positive constant depending only on $N$, $\alpha$,
$p$, $\Lambda$, $\sigma$, $m$ and $\max\,\{\delta_{1},\ldots,\delta_{N}\}$.
An analogous lower bound holds for the essential infimum, with $u_{-}$
replacing $u_{+}$ on the right-hand side of $\eqref{est:essup_small-p}$.
In particular, if $u$ is a weak solution of \eqref{eq:diffusion}
under assumptions $\eqref{eq:coeff_limit}$, $(\ref{assumpt:f})$,
$(\ref{extra_integrability})$ and \eqref{range:small_p}, then $u$
satisfies these bounds.
\end{thm}

\noindent \begin{proof}[\bfseries{Proof}]Let $u\in\mathcal{A}(\alpha,p,\{\delta_{i}\},\Omega_{T},f,\mathcal{C})$,
with $\alpha$, $p$ and $f$ as in the statement of the theorem.
We limit ourselves to the derivation of the upper bound in \eqref{est:essup_small-p},
since arguments analogous to those presented below yield a lower bound
for the essential infimum in terms of the $L^{m}$-norm of $u_{-}\,$.
We use the space-time cylinders $Q_{j}$ defined in \eqref{r_j_Q_j-def}
and introduce the sequences 
\begin{align*}
Z_{j}:=\iint_{Q_{j}}(u-k_{j})_{+}^{m}\,dx\,dt\,,\,\,\,\,\,\,\,\,k_{j}:=\,k(1-2^{-j}),\quad\widehat{k}_{j}:=\,\frac{1}{2}(k_{j}+k_{j+1})\,,\,\,\,\,\,\,\,\,\,\,\mathrm{for}\,\,j\in\mathbb{N}_{0}\,,
\end{align*}
where $k\geq1$ is a number to be chosen later. Note that $Z_{j}$
is finite for every $j\in\mathbb{N}_{0}$, due to the integrability
condition \eqref{extra_integrability}. In view of \eqref{m-ineq},
we may write
\begin{align*}
Z_{j+1} & =\iint_{Q_{j+1}}(u-k_{j+1})_{+}^{m\,-\,p_{\alpha+1}}\,(u-k_{j+1})_{+}^{p_{\alpha+1}}\,dx\,dt\\
 & \leq\iint_{Q_{j+1}}u_{+}^{m\,-\,p_{\alpha+1}}\,(u-k_{j+1})_{+}^{p_{\alpha+1}}\,dx\,dt\\
 & \leq\Vert u_{+}\Vert_{L^{\infty}(Q_{0})}^{m\,-\,p_{\alpha+1}}\iint_{Q_{j+1}}(u-k_{j+1})_{+}^{p_{\alpha+1}}\,dx\,dt\,.
\end{align*}
We now observe that the last integral coincides with $Y_{j+1}$ as
defined in \eqref{def:k_j_alpha_small}, except for the different
choice of $k_{j+1}$. Therefore, we may reuse \foreignlanguage{british}{\eqref{eq:integ_after_energy}}
and \eqref{est:term3} with the current definitions of $\widehat{k}_{j}$
and $E_{j}$, and with the cut-off functions $\eta_{j}$ and $\varphi_{j}$
chosen as in the proof of Theorem \ref{theo:local_bdd_large_bar_p}.
We thus obtain\begin{align}\label{eq:Z_jplusone_after_energy}
Z_{j+1}\,\leq\,c\,b^{j}\,\Vert u_{+}\Vert_{L^{\infty}(Q_{0})}^{m\,-\,p_{\alpha+1}}&\bigg[r^{-p}(1-\nu)^{-p}\iint_{Q_{j}}\left[(u-\widehat{k}_{j})_{+}^{p}\,+\,(u^{\frac{\alpha+1}{2}}-\widehat{k}_{j}^{\frac{\alpha+1}{2}})_{+}^{2}\right]dx\,dt\nonumber\\
&+\,\Vert f\Vert_{L^{\sigma}(Q_{0})}\Big(\iint_{Q_{j}}(u-\widehat{k}_{j})_{+}^{\sigma'}\,dx\,dt\Big)^{\frac{1}{\sigma'}}+\,\delta_{\max}^{p}\,|E_{j}|\bigg]^{\frac{N+p}{N}},
\end{align}where $c=c(N,\alpha,p,\Lambda)>0$, $b=b(N,p)>1$ and $\delta_{\max}:=\max\,\{\delta_{1},\ldots,\delta_{N}\}$.
At this point, we estimate separately the three terms involving $u$
in the integrals appearing on the right-hand of \foreignlanguage{british}{\eqref{eq:Z_jplusone_after_energy}}.
Using the definitions of $k_{j}$ and $\widehat{k}_{j},$ and recalling
that $m>\alpha+1>p$ in the current range of $p$, we have
\begin{align}
(u-\widehat{k}_{j})_{+}^{p}\, & =\,(\widehat{k}_{j}-k_{j})^{p-m}\,(\widehat{k}_{j}-k_{j})^{m-p}\,(u-\widehat{k}_{j})_{+}^{p}\nonumber \\
 & \le\,(\widehat{k}_{j}-k_{j})^{p-m}\,(u-k_{j})_{+}^{m}\,\leq\,c_{1}\,b^{j}\,k^{p-m}\,(u-k_{j})_{+}^{m}\,,\boldsymbol{}\label{eq:btut2}
\end{align}
where $c_{1}=c_{1}(p,m)>0$, while $b>1$ now also depends on $m$.
Similarly, for the other two terms we find that
\begin{equation}
(u-\widehat{k}_{j})_{+}^{\sigma'}\,\leq\,c\,b^{j}\,k^{\sigma'-m}\,(u-k_{j})_{+}^{m}\label{eq:btut3}
\end{equation}
and
\begin{align}
(u^{\frac{\alpha+1}{2}}-\widehat{k}_{j}^{\frac{\alpha+1}{2}})_{+}^{2} & \,\leq\,u^{\alpha+1}\,\chi_{\{u\,>\,\widehat{k}_{j}\}}\,\leq\,c\,(u-\widehat{k}_{j})_{+}^{\alpha+1}\,+\,c\,\widehat{k}_{j}^{\alpha+1}\,\chi_{\{u\,>\,\widehat{k}_{j}\}}\nonumber \\
 & \,\leq\,c\,b^{j}\,k^{\alpha+1-m}\,(u-k_{j})_{+}^{m}\,+\,c\,k^{\alpha+1}\,\chi_{\{u\,>\,\widehat{k}_{j}\}}\,,\label{eq:btut1}
\end{align}
where $b>1$ now also depends on $\alpha$ and $\sigma$, while $c$
additionally depends on $\sigma$ and $m$. We point out that the
derivation of (\ref{eq:btut3}) relies on the fact that $\sigma'<m$,
which follows from (\ref{sigma-prime_ineq}) combined with the estimate
$p_{\alpha+1}\leq\alpha+1<m$, valid in the present range of $p$.
Arguing as in \foreignlanguage{british}{\eqref{eq:superlevel}}, we
further obtain
\begin{align}
|E_{j}|\,\leq\,c_{1}\,2^{jm}\,k^{-m}\,Z_{j}\,.\label{est:superlevel3}
\end{align}
Combining estimates \eqref{eq:btut2}$-$\eqref{est:superlevel3}
with \foreignlanguage{british}{\eqref{eq:Z_jplusone_after_energy}},
we end up with\begin{align}\label{eq:est:Z_jplusone_almost_DG}
Z_{j+1}\,\leq\,c\,b^{j}\,\Vert u_{+}\Vert_{L^{\infty}(Q_{0})}^{m\,-\,p_{\alpha+1}}&\bigg[r^{-p}(1-\nu)^{-p}\,(k^{p-m}+k^{\alpha+1-m})\,Z_{j}\nonumber\\
&+\,\Vert f\Vert_{L^{\sigma}(Q_{0})}\,k^{1-\frac{m}{\sigma'}}\,Z_{j}^{\frac{1}{\sigma'}}\,+\,\delta_{\max}^{p}\,k^{-m}\,Z_{j}\bigg]^{\frac{N+p}{N}}.
\end{align}Since $k\geq1$ and $p<\alpha+1<m$, we may estimate
\begin{align}
k^{p-m}\,\leq\,k^{\alpha+1-m}\,,\quad\quad k^{-m}\,\leq\,k^{\alpha+1-m}\,.\label{est:k-exp1}
\end{align}
Furthermore, setting
\begin{align}
q:=\,\min\left\{ m-(\alpha+1),\,\frac{m}{\sigma'}\,-1\right\} ,\label{expr:q}
\end{align}
we have 
\begin{align}
k^{\alpha+1-m}\,\leq\,k^{-q}\,,\quad\quad k^{1-\frac{m}{\sigma'}}\,\leq\,k^{-q}\,.\label{est:k-exp2}
\end{align}
For future reference, we note that $q>0$ by virtue of \eqref{sigma-prime_ineq}
and \eqref{m-ineq}. Using \eqref{est:k-exp1} and \eqref{est:k-exp2}
in \foreignlanguage{british}{\eqref{eq:est:Z_jplusone_almost_DG}},
and recalling the definition of $\xi(p,r,\nu)$ in \eqref{eq:def:xi},
we obtain
\[
Z_{j+1}\,\leq\,c\,b^{j}\,\Vert u_{+}\Vert_{L^{\infty}(Q_{0})}^{m\,-\,p_{\alpha+1}}\,k^{-\,\frac{q(N+p)}{N}}\left[\xi(p,r,\nu)\,Z_{j}\,+\,\Vert f\Vert_{L^{\sigma}(Q_{0})}\,Z_{j}^{\frac{1}{\sigma'}}\right]^{\frac{N+p}{N}},
\]
where $c$ now also depends on $\delta_{\max}$. For the first term
inside the square brackets, we use the following estimate
\begin{align*}
Z_{j}\,=\,Z_{j}^{\frac{1}{\sigma}}\,Z_{j}^{\frac{1}{\sigma'}}\,\leq\,\Vert u_{+}\Vert_{L^{m}(Q_{0})}^{\frac{m}{\sigma}}\,Z_{j}^{\frac{1}{\sigma'}}\,,
\end{align*}
which leads to
\[
Z_{j+1}\,\leq\,C_{u,f}\,b^{j}\,Z_{j}^{1+\beta}\,,
\]
where 
\[
C_{u,f}\,=\,c\,\Vert u_{+}\Vert_{L^{\infty}(Q_{0})}^{m\,-\,p_{\alpha+1}}\,k^{-\,\frac{q(N+p)}{N}}\left[\xi(p,r,\nu)\,\Vert u_{+}\Vert_{L^{m}(Q_{0})}^{\frac{m}{\sigma}}\,+\,\Vert f\Vert_{L^{\sigma}(Q_{0})}\right]^{\frac{N+p}{N}}
\]
and $\beta$ is defined as in \eqref{eq:def:beta}. If $C_{u,f}=0$,
then $u_{+}=0$ almost everywhere in $Q_{r}(z_{o})$ and estimate
$\eqref{est:essup_small-p}$ is trivially satisfied. If instead $C_{u,f}>0$,
then, by Lemma \ref{lem:fastconvg}, the sequence $\{Z_{j}\}$ converges
to zero provided that\\
\begin{align*}
\Vert u_{+}\Vert_{L^{m}(Q_{0})}^{m}\,=\,Z_{0}\,\leq\,C_{u,f}^{-\,\frac{1}{\beta}}\,b^{-\,\frac{1}{\beta^{2}}}\,,
\end{align*}
which, in view of the positivity of $q$, can be reformulated as
\begin{equation}
k\,\geq\,c\,\Vert u_{+}\Vert_{L^{\infty}(Q_{0})}^{\theta}\,\Vert u_{+}\Vert_{L^{m}(Q_{0})}^{\frac{Nm\beta}{q(N+p)}}\left[\xi(p,r,\nu)\,\Vert u_{+}\Vert_{L^{m}(Q_{0})}^{\frac{m}{\sigma}}\,+\,\Vert f\Vert_{L^{\sigma}(Q_{0})}\right]^{\frac{1}{q}},\label{eq:T1}
\end{equation}
where we have introduced the parameter 
\[
\theta\,=\,\frac{N(m-p_{\alpha+1})}{q(N+p)}\,>\,0\,.
\]
Estimating the right-hand side of \eqref{eq:T1} from above by means
of inequality \eqref{eq:elem_ineq} with $\tau=\frac{1}{q}$, and
using the definition of $\beta$, we see that \eqref{eq:T1} holds,
for instance, if 
\[
k\geq\,c\,\Vert u_{+}\Vert_{L^{\infty}(Q_{0})}^{\theta}\left[\xi(p,r,\nu)^{\frac{1}{q}}\,\Vert u_{+}\Vert_{L^{m}(Q_{0})}^{\frac{mp}{q(N+p)}}\,+\,\Vert u_{+}\Vert_{L^{m}(Q_{0})}^{\frac{m}{q}\left(\frac{p}{N+p}\,-\,\frac{1}{\sigma}\right)}\,\Vert f\Vert_{L^{\sigma}(Q_{0})}^{\frac{1}{q}}\right].
\]
Applying Young's inequality to the last term inside the square brackets
with exponents
\[
Q:=\,\frac{\frac{p}{N+p}}{\frac{p}{N+p}-\frac{1}{\sigma}}\,\,\,\,\,\,\,\,\mathrm{and}\,\,\,\,\,\,\,\,Q'\,=\,\frac{p\sigma}{N+p}\,,
\]
we further see that \eqref{eq:T1} is satisfied if
\begin{equation}
k\,=\,c\,\Vert u_{+}\Vert_{L^{\infty}(Q_{0})}^{\theta}\left[\xi(p,r,\nu)^{\frac{1}{q}}\,\Vert u_{+}\Vert_{L^{m}(Q_{0})}^{\frac{mp}{q(N+p)}}\,+\,\Vert f\Vert_{L^{\sigma}(Q_{0})}^{\frac{p\sigma}{q(N+p)}}\right]+1\,,\label{eq:T2}
\end{equation}
where we have taken into account the fact that $k\geq1$. Thus, if
\eqref{eq:T2} holds, then
\begin{align*}
\iint_{Q_{\nu r}(z_{o})}(u-k)_{+}^{m}\,dx\,dt\,\leq\,Z_{j}\to0\,\,\,\,\,\,\,\,\mathrm{as}\,\,j\to\infty\,,
\end{align*}
and hence $u\leq k\,$ a.e. in $Q_{\nu r}(z_{o})$. We have thus proved
that
\begin{equation}
\underset{Q_{\nu r}(z_{o})}{\mathrm{ess}\,\sup}\,\,u_{+}\,\leq\,c\,\Big(\underset{Q_{r}(z_{o})}{\mathrm{ess}\,\sup}\,\,u_{+}\Big)^{\theta}\left[\xi(p,r,\nu)^{\frac{1}{q}}\,\Vert u_{+}\Vert_{L^{m}(Q_{r}(z_{o}))}^{\frac{mp}{q(N+p)}}\,+\,\Vert f\Vert_{L^{\sigma}(Q_{r}(z_{o}))}^{\frac{p\sigma}{q(N+p)}}\right]+1\,.\label{eq:T3}
\end{equation}
We now show that $\theta\in(0,1)$. If $q=m-(\alpha+1)$, we have
\begin{align*}
\theta\,=\,\frac{N(m-p_{\alpha+1})}{(m-\alpha-1)(N+p)}\,=\,\frac{mN-p(N+\alpha+1)}{mN-p(N+\alpha+1)+p\,[m-\frac{N}{p}(\alpha+1-p)]}\,<\,1\,,
\end{align*}
since the expression inside the square brackets is positive due to
\eqref{extra_integrability}. On the other hand, if $q=\tfrac{m}{\sigma'}-1$,
we have 
\begin{align*}
\theta\,=\,\frac{m-p_{\alpha+1}}{m-\sigma'}\,\cdot\,\frac{N\sigma'}{N+p}\,<\,1\,,
\end{align*}
where the last inequality follows from the fact that both fractions
above take values in $(0,1)$, in view of \eqref{assumpt:f}, \eqref{sigma-prime_ineq}
and \eqref{m-ineq}. At this point, we may apply Young's inequality
with exponents $1/\theta$ and $1/(1-\theta)$ to the first term on
the right-hand side of (\ref{eq:T3}), thus obtaining\begin{align}\label{eq:snelhest}
\underset{Q_{\nu r}(z_{o})}{\mathrm{ess}\,\sup}\,\,u_{+}\,&\leq\,\varepsilon\,\,\,\underset{Q_{r}(z_{o})}{\mathrm{ess}\,\sup}\,\,u_{+}\,+\,\tilde{c}(\varepsilon)\left[\xi(p,r,\nu)^{\frac{N+p}{\mu}}\,\Vert u_{+}\Vert_{L^{m}(Q_{r}(z_{o}))}^{\frac{mp}{\mu}}\,+\,\Vert f\Vert_{L^{\sigma}(Q_{r}(z_{o}))}^{\frac{p\sigma}{\mu}}\right]+1\nonumber\\
&\leq\,\varepsilon\,\,\,\underset{Q_{r}(z_{o})}{\mathrm{ess}\,\sup}\,\,u_{+}\,+\,\tilde{c}(\varepsilon)\left([1+(r(1-\nu))^{-(N+p)}]\iint_{Q_{r}(z_{o})}u_{+}^{m}\,dx\,dt\right)^{\frac{p}{\mu}}\nonumber\\
&\,\,\,\,\,\,\,+\,\tilde{c}(\varepsilon)\,\left(\iint_{Q_{r}(z_{o})}|f|^{\sigma}\,dx\,dt\right)^{\frac{p}{\mu}}+1\,,
\end{align}where $\varepsilon>0$ is to be chosen appropriately later, and
\begin{align*}
\mu:=\,q(1-\theta)(N+p)\,=\,q(N+p)-N(m-p_{\alpha+1})\,>\,0\,.
\end{align*}
Note that in \foreignlanguage{british}{\eqref{eq:snelhest}} we have
also used the definition of $\xi(p,r,\nu)$ in \eqref{eq:def:xi},
and that the positive constant $\tilde{c}(\varepsilon)$ depends on
$\varepsilon$, in addition to $N$, $\alpha$, $p$, $\Lambda$,
$\sigma$, $m$ and $\delta_{\max}$. In the case $q=m-(\alpha+1)$,
a direct calculation shows that 
\begin{align}
\mu\,=\,mp-N(\alpha+1-p)\,.\label{q_alt1}
\end{align}
If instead $q=\frac{m}{\sigma'}-1$, we have 
\begin{align}
\mu\,=\,mp-N(\alpha+1-p)+(N+p)\left(\alpha-\frac{m}{\sigma}\right).\label{q_alt2}
\end{align}
Comparing \eqref{q_alt1} and \eqref{q_alt2} with the definition
of $q$ in \eqref{expr:q}, we see that in any case 
\begin{align*}
\mu\,=\,mp-N(\alpha+1-p)+(N+p)\,\min\left\{ 0,\alpha-\frac{m}{\sigma}\right\} .
\end{align*}
At this stage, we can argue exactly as in the final part of the proof
of Theorem \ref{thm:loc_bdd_p_large_LP}, using \foreignlanguage{british}{\eqref{eq:snelhest}}
to start an iteration procedure which, upon a suitable choice of $\varepsilon$,
eventually leads to estimate (\ref{est:essup_small-p}).\end{proof}

\section{Properties of solutions to the Cauchy problem\label{sec:Cauchy} }

\selectlanguage{british}%
\noindent $\hspace*{1em}$\foreignlanguage{american}{In this section,
we establish contractive bounds that will later be used to derive
a global boundedness result for weak solutions to a Cauchy problem
associated with equation \eqref{eq:diffusion}. Our approach is inspired
by that of \cite{CiaVeVe}, where analogous results were proved for
doubly nonlinear anisotropic evolution equations. Here, the main novelty
with respect to \cite{CiaVeVe} is that we also allow for a nontrivial
source term $f$ and investigate its effect on the resulting estimates.
More specifically, we consider the following Cauchy problem
\begin{align}
\left\{ \begin{array}{ll}
\partial_{t}(\vert u\vert^{\alpha-1}u)-\sum_{i=1}^{N}\partial_{i}\left[a_{i}(x,t)\,(|\partial_{i}u|-\delta_{i})_{+}^{p-1}\frac{\partial_{i}u}{\vert\partial_{i}u\vert}\right]=f, & \quad\text{in }S_{T}:=\mathbb{R}^{N}\times(0,T),\\[5pt]
u(x,0)=u_{0}(x), & \quad x\in\mathbb{R}^{N},
\end{array}\right.\label{prob:cauchy}
\end{align}
where, throughout this section, we shall assume that $\alpha\in(0,\infty),$
$p\in(1,\infty)$, and that the measurable coefficients $a_{i}$ satisfy
\eqref{eq:coeff_limit} with $S_{T}$ in place of $\Omega_{T}$. To
make the whole setting more precise, we introduce the following definition.}
\selectlanguage{american}%
\begin{defn}
\noindent \label{def:Lp-integrable-sol} Let $f\in L_{\mathrm{loc}}^{1}(S_{T})$
and $u_{0}\in L_{\mathrm{loc}}^{\alpha+1}(\mathbb{R}^{N})$. We say
that a function 
\[
u\,\in\,L^{p}(0,T;W_{\mathrm{loc}}^{1,p}(\mathbb{R}^{N}))\cap C^{0}([0,T];L_{\mathrm{loc}}^{\alpha+1}(\mathbb{R}^{N}))
\]
is an $L^{p}$-integrable weak solution to problem \eqref{prob:cauchy}
if the following conditions hold:
\begin{enumerate}
\item[\foreignlanguage{american}{$\mathrm{(1)}$}] $u\in L^{p}(S_{T})$;\vspace{-2mm}
\item[\foreignlanguage{american}{$\mathrm{(2)}$}] $u$ satisfies the weak formulation \eqref{eq:weak_form}, with $S_{T}$
replacing $\Omega_{T}$, for every $\varphi\in C_{0}^{\infty}(S_{T})$;\vspace{-2mm}
\item[\foreignlanguage{american}{$\mathrm{(3)}$}] $u(\cdot,t)\rightarrow u_{0}$ in $L_{\mathrm{loc}}^{\alpha+1}(\mathbb{R}^{N})$
as $t\to0^{+}$.
\end{enumerate}
\end{defn}

\noindent \begin{brem}Since in the previous definition $u$ is assumed
to belong to $C^{0}([0,T];L_{\mathrm{loc}}^{\alpha+1}(\mathbb{R}^{N}))$,
the function $u(\cdot,t)\in L_{\mathrm{loc}}^{\alpha+1}(\mathbb{R}^{N})$
is well-defined for every $t\in[0,T]$, and thus condition $(3)$
makes sense.\end{brem}

\selectlanguage{british}%
\noindent $\hspace*{1em}$\foreignlanguage{american}{The following
result is the counterpart of Lemma 8.3 in the aforementioned work
\cite{CiaVeVe}. Here, we additionally allow for a right-hand side
$f\not\equiv0$, and establish separate estimates for the positive
and negative parts of the solution.}
\selectlanguage{american}%
\begin{lem}
\label{lem:Lalphaplusone-contractive} Let $u$ be an $L^{p}$-integrable
solution to the Cauchy problem \eqref{prob:cauchy} in the sense of
Definition \ref{def:Lp-integrable-sol}, with $u_{0}\in L^{\alpha+1}(\mathbb{R}^{N})$
and $f\in L^{\frac{1}{\alpha}+1}(S_{T})$. Then, for every $\delta>0$,
\begin{equation}
\sup_{\tau\,\in\,[0,T]}\int_{\mathbb{R}^{N}\times\{\tau\}}u_{+}^{\alpha+1}\,dx\,\leq\,(1+\delta)\int_{\mathbb{R}^{N}}(u_{0})_{+}^{\alpha+1}\,dx\,+\,(1+\delta)^{\frac{\alpha+1}{\alpha}}\delta^{-\frac{1}{\alpha}}\left(\frac{T}{\alpha}\right)^{\frac{1}{\alpha}}\iint_{S_{T}}f_{+}^{\frac{\alpha+1}{\alpha}}\,dx\,dt\,.\label{eq:est_lem5_2}
\end{equation}
An analogous estimate holds for $u_{-}$, with $(u_{0})_{-}$ and
$f_{-}$ replacing $(u_{0})_{+}$ and $f_{+}$, respectively. Moreover,
if $f\equiv0$ we have 
\begin{align}
\Vert u_{+}(\cdot,t)\Vert_{L^{\alpha+1}(\mathbb{R}^{N})}\,\leq\,\Vert(u_{0})_{+}\Vert_{L^{\alpha+1}(\mathbb{R}^{N})}\,,\label{est:contractive_Lalphaplusone_fzero}
\end{align}
for all $t\in[0,T)$, with analogous estimates holding for both $u_{-}$
and $u$.
\end{lem}

\noindent \begin{proof}[\bfseries{Proof}]Let $\eta\in C_{0}^{\infty}(B_{2}(0);[0,1])$
be such that $\eta=1$ on $B_{1}(0)$, and set
\[
\eta_{r}(x):=\,\eta\Big(\frac{x}{r}\Big)\,,\,\,\,\,\,\,\,\,r>0\,.
\]
Using Lemma \ref{lem:energy-general} with $\Omega=\mathbb{R}^{N}$,
$\tau_{1}=0$, $\tau_{2}=\tau$, $\varphi\equiv1$, $F(s)=s_{+}$
and
\[
G(s)=\int_{0}^{s}g(\omega)\,d\omega\,=\int_{0}^{s}F(\vert\omega\vert^{\frac{1}{\alpha}-1}\omega)\,d\omega\,,
\]

\noindent we then obtain
\begin{align*}
\frac{\alpha}{\alpha+1}\int_{\mathbb{R}^{N}}\eta_{r}^{p}\,u_{+}^{\alpha+1}(x,\tau)\,dx\, & \leq\,\frac{\alpha}{\alpha+1}\int_{\mathbb{R}^{N}}\eta_{r}^{p}\,(u_{0})_{+}^{\alpha+1}\,dx\,+\iint_{\mathbb{R}^{N}\times[0,\tau]}f\,\eta_{r}^{p}\,u_{+}\,dx\,dt\\
 & \quad+\,C(N,p,\Lambda)\iint_{\mathbb{R}^{N}\times[0,\tau]}u_{+}^{p}\,|\nabla\eta_{r}|^{p}\,dx\,dt\,,
\end{align*}
for all $\tau\in[0,T]$. Note that we have omitted the second term
on the left-hand side of \foreignlanguage{british}{\eqref{eq:energy-general}},
since it is non-negative. Now, on the right-hand side, we use the
bounds $0\leq\eta_{r}\leq1$ and $|\nabla\eta_{r}|\leq c/r$, and
replace $f$ by its positive part. This yields
\begin{align*}
\int_{\mathbb{R}^{N}}\eta_{r}^{p}\,u_{+}^{\alpha+1}(x,\tau)\,dx\,\leq\int_{\mathbb{R}^{N}}(u_{0})_{+}^{\alpha+1}\,dx\,+\,\frac{\alpha+1}{\alpha}\iint_{\mathbb{R}^{N}\times[0,\tau]}f_{+}\,\eta_{r}^{p}\,u_{+}\,dx\,dt\,+\,\frac{C}{r^{p}}\iint_{S_{T}}u_{+}^{p}\,dx\,dt\,,
\end{align*}
for all $\tau\in[0,T]$, where $C$ now also depends on $\alpha$.
Since $\eta_{r}$ is compactly supported in space, we have that $\eta_{r}^{p}\,u_{+}\in L^{\alpha+1}(S_{T})$.
Together with the integrability properties of $u$, $u_{0}$ and $f$,
this shows that the right-hand side of the previous estimate is finite,
and hence
\begin{align}
\sup_{\tau\,\in\,[0,T]}\int_{\mathbb{R}^{N}}\eta_{r}^{p}\,u_{+}^{\alpha+1}(x,\tau)\,dx\, & \leq\int_{\mathbb{R}^{N}}(u_{0})_{+}^{\alpha+1}\,dx\,+\,\frac{\alpha+1}{\alpha}\iint_{S_{T}}f_{+}\,\eta_{r}^{p}\,u_{+}\,dx\,dt\,+\,\frac{C}{r^{p}}\iint_{S_{T}}u_{+}^{p}\,dx\,dt\nonumber \\
 & <+\infty\,.\label{eq:towardsLalphaplusone}
\end{align}
Applying Hölder's inequality, first in space and then in time, and
subsequently Young's inequality with $\varepsilon>0$, we estimate
the second term on the right-hand side of (\ref{eq:towardsLalphaplusone})
as follows:\begin{align*}
&\frac{\alpha+1}{\alpha}\iint_{S_{T}}f_{+}\,\eta_{r}^{p}\,u_{+}\,dx\,dt\\
&\,\,\,\,\,\,\,\leq\,\frac{\alpha+1}{\alpha}\int_{0}^{T}\Big(\int_{\mathbb{R}^{N}}\eta_{r}^{p(\alpha+1)}u_{+}^{\alpha+1}\,dx\Big)^{\frac{1}{\alpha+1}}\Big(\int_{\mathbb{R}^{N}}f_{+}^{\frac{\alpha+1}{\alpha}}\,dx\Big)^{\frac{\alpha}{\alpha+1}}dt\\
&\,\,\,\,\,\,\,\leq\,\frac{\alpha+1}{\alpha}\Big(\sup_{\tau\,\in\,[0,T]}\int_{\mathbb{R}^{N}}\eta_{r}^{p}\,u_{+}^{\alpha+1}(x,\tau)\,dx\Big)^{\frac{1}{\alpha+1}}\int_{0}^{T}\Big(\int_{\mathbb{R}^{N}}f_{+}^{\frac{\alpha+1}{\alpha}}\,dx\Big)^{\frac{\alpha}{\alpha+1}}dt\\
&\,\,\,\,\,\,\,\leq\,\frac{\alpha+1}{\alpha}\Big(\sup_{\tau\,\in\,[0,T]}\int_{\mathbb{R}^{N}}\eta_{r}^{p}\,u_{+}^{\alpha+1}(x,\tau)\,dx\Big)^{\frac{1}{\alpha+1}}\,T^{\frac{1}{\alpha+1}}\Big(\iint_{S_{T}}f_{+}^{\frac{\alpha+1}{\alpha}}\,dx\,dt\Big)^{\frac{\alpha}{\alpha+1}}\\
&\,\,\,\,\,\,\,\leq\,\frac{\varepsilon^{\alpha+1}}{\alpha+1}\,\sup_{\tau\,\in\,[0,T]}\int_{\mathbb{R}^{N}}\eta_{r}^{p}\,u_{+}^{\alpha+1}(x,\tau)\,dx\,+\left(\frac{\alpha+1}{\alpha}\right)^{\frac{1}{\alpha}}\varepsilon^{-\,\frac{\alpha+1}{\alpha}}\,T^{\frac{1}{\alpha}}\iint_{S_{T}}f_{+}^{\frac{\alpha+1}{\alpha}}\,dx\,dt\,,
\end{align*}where we have also exploited the fact that $\eta_{r}^{p(\alpha+1)}\leq\eta_{r}^{p}$.
Choosing $\varepsilon=[(\alpha+1)\delta(1+\delta)^{-1}]^{\frac{1}{\alpha+1}}$,
combining the previous estimate with (\ref{eq:towardsLalphaplusone})
and observing that $\eta_{r}=1$ on $B_{r}(0)$, we obtain\begin{align*}
&\sup_{\tau\,\in\,[0,T]}\int_{B_{r}(0)}u_{+}^{\alpha+1}(x,\tau)\,dx\,\leq\sup_{\tau\,\in\,[0,T]}\int_{\mathbb{R}^{N}}\eta_{r}^{p}\,u_{+}^{\alpha+1}(x,\tau)\,dx\\
&\,\,\,\,\,\,\,\leq\,(1+\delta)\int_{\mathbb{R}^{N}}(u_{0})_{+}^{\alpha+1}\,dx\,+\,(1+\delta)^{\frac{\alpha+1}{\alpha}}\delta^{-\frac{1}{\alpha}}\left(\frac{T}{\alpha}\right)^{\frac{1}{\alpha}}\iint_{S_{T}}f_{+}^{\frac{\alpha+1}{\alpha}}\,dx\,dt\,+\,\frac{C}{r^{p}}\iint_{S_{T}}u_{+}^{p}\,dx\,dt\,,
\end{align*}where we have also used the finiteness of the supremum over $\tau\in[0,T]$.
Passing to the limit as $r\to\infty$, the last term vanishes since
$u\in L^{p}(S_{T})$, and we obtain estimate (\ref{eq:est_lem5_2}).
If $f\equiv0$, we may further let $\delta\to0$ to obtain \eqref{est:contractive_Lalphaplusone_fzero},
or alternatively pass to the limit as $r\to\infty$ directly in (\ref{eq:towardsLalphaplusone}).
The bounds for $u_{-}$ are obtained in the same way, by choosing
$F(s)=-s_{-}\,$. Finally, the bound for $u$ itself follows by observing
that $\vert u\vert^{\alpha+1}=u_{+}^{\alpha+1}+u_{-}^{\alpha+1}$
and adding the corresponding integral estimates for $u_{+}$ and $u_{-}\,$.\end{proof}

\selectlanguage{british}%
\noindent $\hspace*{1em}$\foreignlanguage{american}{In the range
$\alpha\in(0,1)$, we also obtain a contractive estimate concerning
the spatial $L^{1}$-norm, under integrability assumptions on $u_{0}$
and $f$ different from those in Lemma \ref{lem:Lalphaplusone-contractive}.
For a similar result for doubly nonlinear anisotropic equations, see
\cite[Lemma 8.4]{CiaVeVe}, where, however, $f\equiv0$.}
\selectlanguage{american}%
\begin{lem}
Let $\alpha\in(0,1)$. Let $u$ be an $L^{p}$-integrable solution
to the Cauchy problem \eqref{prob:cauchy} in the sense of Definition
\ref{def:Lp-integrable-sol}, with $u_{0}\in L^{1}(\mathbb{R}^{N})\cap L_{\mathrm{loc}}^{\alpha+1}(\mathbb{R}^{N})$
and $f\in L^{\frac{1}{\alpha}}(S_{T})$. Then, for every $\delta>0$,
\begin{equation}
\sup_{\tau\,\in\,[0,T]}\Vert u(\cdot,\tau)\Vert_{L^{1}(\mathbb{R}^{N})}\,\leq\,(1+\delta)\,\Vert u_{0}\Vert_{L^{1}(\mathbb{R}^{N})}\,+\left(\frac{1-\alpha^{2}}{\alpha\delta}\right)^{\frac{1}{\alpha}-1}(1+\delta)^{\frac{1}{\alpha}}\,T^{\frac{1}{\alpha}-1}\iint_{S_{T}}|f|^{\frac{1}{\alpha}}\,dx\,dt\,.\label{eq:Lemma5_3est1}
\end{equation}
If in addition $f\equiv0$, then we have 
\[
\sup_{\tau\,\in\,[0,T]}\Vert u(\cdot,\tau)\Vert_{L^{1}(\mathbb{R}^{N})}\,\leq\,\Vert u_{0}\Vert_{L^{1}(\mathbb{R}^{N})}\,.
\]
\end{lem}

\noindent \begin{proof}[\bfseries{Proof}]We take $\eta$ and $\eta_{r}$
as in the proof of Lemma \ref{lem:Lalphaplusone-contractive}. We
will use Lemma \ref{lem:energy-general} with $\Omega=\mathbb{R}^{N}$
and 
\begin{align*}
F(s)=(s^{2}+\varepsilon)^{-\frac{\alpha}{2}}\,s\,,\,\,\,\,\,\,\,\,\,\,g(s)=(\vert s\vert^{\frac{2}{\alpha}}+\varepsilon)^{-\frac{\alpha}{2}}\,\vert s\vert^{\frac{1}{\alpha}-1}\,s\,,
\end{align*}
where $\varepsilon>0$. Define 
\[
G(s)=\int_{0}^{s}g(\tau)\,d\tau\,.
\]
Since $\alpha\in(0,1)$, the function $F$ satisfies the assumptions
of Lemma \ref{lem:energy-general} for any fixed $\varepsilon>0$.
By a direct calculation, we see that
\[
F'(s)\,=\,(s^{2}+\varepsilon)^{-\frac{\alpha}{2}-1}\,[(1-\alpha)s^{2}+\varepsilon]\,\geq\,(1-\alpha)(s^{2}+\varepsilon)^{-\frac{\alpha}{2}}\,,
\]
and hence\begin{align*}
|F(u)|^{p}\,(F'(u))^{1-p}\,&\le\,(1-\alpha)^{1-p}\,(u^{2}+\varepsilon)^{-\frac{\alpha}{2}}\,\vert u\vert^{p}\\
&\leq\,C(\alpha,p)\,\varepsilon^{-\frac{\alpha}{2}}\,|u|^{p}\,.
\end{align*}

\noindent Therefore, using \foreignlanguage{british}{\eqref{eq:energy-general}}
with $\tau_{1}=0$, $\tau_{2}=\tau$ and $\varphi\equiv1$, we have
\begin{align}
\int_{\mathbb{R}^{N}}\eta_{r}^{p}\,G(\vert u\vert^{\alpha-1}u)(x,\tau)\,dx\, & \leq\int_{\mathbb{R}^{N}}\eta_{r}^{p}\,G(\vert u_{0}\vert^{\alpha-1}u_{0})\,dx\,+\int_{0}^{\tau}\int_{\mathbb{R}^{N}}\eta_{r}^{p}\,(u^{2}+\varepsilon)^{-\frac{\alpha}{2}}\,uf\,dx\,dt\nonumber \\
 & \quad+\,C\varepsilon^{-\frac{\alpha}{2}}\int_{0}^{\tau}\int_{\mathbb{R}^{N}}|u|^{p}\,|\nabla\eta_{r}|^{p}\,dx\,dt\,,\label{eq:est:towardsL1}
\end{align}
for all $\tau\in[0,T],$ where $C$ now also depends on $N$. Note
that 
\[
G(\vert u\vert^{\alpha-1}u)=\int_{0}^{\vert u\vert^{\alpha-1}u}(\vert s\vert^{\frac{2}{\alpha}}+\varepsilon)^{-\frac{\alpha}{2}}\,\vert s\vert^{\frac{1}{\alpha}-1}\,s\,ds\,\leq\int_{0}^{\vert u\vert^{\alpha-1}u}\vert s\vert^{\frac{1}{\alpha}-2}\,s\,ds\,=\,\alpha|u|\,,
\]
which, combined with the $L^{1}$-integrability of $u_{0}$, implies
that the first term on the right-hand side of (\ref{eq:est:towardsL1})
is finite. Moreover, by Hölder's inequality, we have 
\begin{align*}
\int_{0}^{\tau}\int_{\mathbb{R}^{N}}\eta_{r}^{p}\,(u^{2}+\varepsilon)^{-\frac{\alpha}{2}}\,uf\,dx\,dt\, & \leq\int_{0}^{T}\int_{B_{2r}(0)}|u|^{1-\alpha}\,|f|\,dx\,dt\\
 & \leq\Big[\int_{0}^{T}\int_{B_{2r}(0)}|u|\,dx\,dt\Big]^{1-\alpha}\Big[\int_{0}^{T}\int_{B_{2r}(0)}|f|^{\frac{1}{\alpha}}\,dx\,dt\Big]^{\alpha}<+\infty\,.
\end{align*}
The last term on the right-hand side of (\ref{eq:est:towardsL1})
is also finite, since $u\in L^{p}(S_{T})$. Therefore, we see that
\begin{align}
\sup_{\tau\,\in\,[0,T]}\int_{\mathbb{R}^{N}}\eta_{r}^{p}\,G(\vert u\vert^{\alpha-1}u)(x,\tau)\,dx\, & \leq\int_{\mathbb{R}^{N}}\eta_{r}^{p}\,G(\vert u_{0}\vert^{\alpha-1}u_{0})\,dx\,+\iint_{S_{T}}\eta_{r}^{p}\,(u^{2}+\varepsilon)^{-\frac{\alpha}{2}}\,uf\,dx\,dt\nonumber \\
 & \quad+\,C\varepsilon^{-\frac{\alpha}{2}}\,r^{-p}\iint_{S_{T}}|u|^{p}\,dx\,dt\,<+\infty\,,\label{eq:timesup-fininte}
\end{align}
where we have also used the fact that $\vert\nabla\eta_{r}\vert\leq c/r$.
To estimate the second term on the right-hand side of \eqref{eq:timesup-fininte}
in a more convenient way, we first apply Hölder's inequality in space:\begin{align}\label{eq:2ndtermbetterest}
&\int_{0}^{T}\int_{\mathbb{R}^{N}}\eta_{r}^{p}\,(u^{2}+\varepsilon)^{-\frac{\alpha}{2}}\,uf\,dx\,dt\nonumber\\
&\,\,\,\,\,\,\,\leq\int_{0}^{T}\Big(\int_{\mathbb{R}^{N}}\eta_{r}^{\frac{p}{1-\alpha}}\,(u^{2}+\varepsilon)^{-\,\frac{\alpha}{2(1-\alpha)}}\,|u|^{\frac{1}{1-\alpha}}\,dx\Big)^{1-\alpha}\Big(\int_{\mathbb{R}^{N}}\vert f|^{\frac{1}{\alpha}}\,dx\Big)^{\alpha}dt\,.
\end{align}In order to control the first integrand in the last line, we observe
that\begin{align*}
G(\vert u\vert^{\alpha-1}u)\,&=\int_{0}^{\vert u\vert^{\alpha-1}u}(|s|^{\frac{2}{\alpha}}+\varepsilon)^{-\frac{\alpha}{2}}\,|s|^{\frac{1}{\alpha}-1}\,s\,ds\\
&\geq\int_{0}^{\vert u\vert^{\alpha-1}u}(u^{2}+\varepsilon)^{-\frac{\alpha}{2}}\,|s|^{\frac{1}{\alpha}-1}\,s\,ds\,=\,\frac{\alpha}{\alpha+1}\,(u^{2}+\varepsilon)^{-\frac{\alpha}{2}}\,|u|^{\alpha+1},
\end{align*}which allows us to conclude that\begin{align}\label{eq:Gestimate}
(u^{2}+\varepsilon)^{-\,\frac{\alpha}{2(1-\alpha)}}\,|u|^{\frac{1}{1-\alpha}}\,&=\,\frac{(u^{2}+\varepsilon)^{-\,\frac{\alpha}{2(1-\alpha)}}\,|u|^{\frac{1}{1-\alpha}}}{G(\vert u\vert^{\alpha-1}u)}\,G(\vert u\vert^{\alpha-1}u)\nonumber\\
&\leq\,\frac{\alpha+1}{\alpha}\,(u^{2}+\varepsilon)^{-\,\frac{\alpha^{2}}{2(1-\alpha)}}\,|u|^{\frac{\alpha^{2}}{1-\alpha}}\,G(\vert u\vert^{\alpha-1}u)\nonumber\\
&\leq\,\frac{\alpha+1}{\alpha}\,G(\vert u\vert^{\alpha-1}u)\,,
\end{align}whenever $u\neq0$. Combining \eqref{eq:2ndtermbetterest} with \eqref{eq:Gestimate},
applying Hölder's inequality in time, and then Young's inequality
with $\theta>0$, we obtain\begin{align}\label{eq:trippa_bolognese}
&\int_{0}^{T}\int_{\mathbb{R}^{N}}\eta_{r}^{p}\,(u^{2}+\varepsilon)^{-\frac{\alpha}{2}}\,uf\,dx\,dt\\
&\,\,\,\,\leq\left(\frac{\alpha+1}{\alpha}\right)^{1-\alpha}\int_{0}^{T}\Big(\int_{\mathbb{R}^{N}}\eta_{r}^{p}\,G(\vert u\vert^{\alpha-1}u)\,dx\Big)^{1-\alpha}\Big(\int_{\mathbb{R}^{N}}|f|^{\frac{1}{\alpha}}\,dx\Big)^{\alpha}dt\nonumber\\
&\,\,\,\,\leq\left(\frac{\alpha+1}{\alpha}\right)^{1-\alpha}\Big(\sup_{\tau\,\in\,[0,T]}\int_{\mathbb{R}^{N}}\eta_{r}^{p}\,G(\vert u\vert^{\alpha-1}u)\,(x,\tau)\,dx\Big)^{1-\alpha}\int_{0}^{T}\Big(\int_{\mathbb{R}^{N}}|f|^{\frac{1}{\alpha}}\,dx\Big)^{\alpha}dt\nonumber\\
&\,\,\,\,\leq\left(\frac{\alpha+1}{\alpha}\right)^{1-\alpha}\Big(\sup_{\tau\,\in\,[0,T]}\int_{\mathbb{R}^{N}}\eta_{r}^{p}\,G(\vert u\vert^{\alpha-1}u)\,(x,\tau)\,dx\Big)^{1-\alpha}\,T^{1-\alpha}\Big(\iint_{S_{T}}|f|^{\frac{1}{\alpha}}\,dx\,dt\Big)^{\alpha}\nonumber\\
&\,\,\,\,\leq\,(1-\alpha)\,\theta\sup_{\tau\,\in\,[0,T]}\int_{\mathbb{R}^{N}}\eta_{r}^{p}\,G(\vert u\vert^{\alpha-1}u)\,(x,\tau)\,dx\,+\,\alpha\left(\frac{\alpha+1}{\alpha\theta}\right)^{\frac{1}{\alpha}-1}T^{\frac{1}{\alpha}-1}\iint_{S_{T}}|f|^{\frac{1}{\alpha}}\,dx\,dt\,.\nonumber
\end{align} Joining \eqref{eq:timesup-fininte} and \eqref{eq:trippa_bolognese},
and recalling that the supremum over $\tau\in[0,T]$ is finite, we
end up with\begin{align*}
&[1-(1-\alpha)\,\theta]\sup_{\tau\,\in\,[0,T]}\int_{\mathbb{R}^{N}}\eta_{r}^{p}\,G(\vert u\vert^{\alpha-1}u)(x,\tau)\,dx\\
&\,\,\,\,\,\,\,\leq\int_{\mathbb{R}^{N}}G(\vert u_{0}\vert^{\alpha-1}u_{0})\,dx\,+\,C\varepsilon^{-\frac{\alpha}{2}}\,r^{-p}\iint_{S_{T}}|u|^{p}\,dx\,dt\,+\,\alpha\left(\frac{\alpha+1}{\alpha\theta}\right)^{\frac{1}{\alpha}-1}T^{\frac{1}{\alpha}-1}\iint_{S_{T}}|f|^{\frac{1}{\alpha}}\,dx\,dt\,.
\end{align*}Passing to the limit as $r\to\infty$ and arguing as in the final
part of the proof of Lemma \ref{lem:Lalphaplusone-contractive}, we
obtain from the previous estimate that\begin{align}\label{eq:after_rlimit}
&[1-(1-\alpha)\,\theta]\sup_{\tau\,\in\,[0,T]}\int_{\mathbb{R}^{N}}G(\vert u\vert^{\alpha-1}u)(x,\tau)\,dx\nonumber\\
&\,\,\,\,\,\,\,\leq\int_{\mathbb{R}^{N}}G(\vert u_{0}\vert^{\alpha-1}u_{0})\,dx\,+\,\alpha\left(\frac{(\alpha+1)T}{\alpha\theta}\right)^{\frac{1}{\alpha}-1}\iint_{S_{T}}|f|^{\frac{1}{\alpha}}\,dx\,dt\,.
\end{align}Exploiting the definition of $G$ we see that, for any $\tau\in[0,T]$,\begin{align*}
\int_{\mathbb{R}^{N}}G(\vert u\vert^{\alpha-1}u)(x,\tau)\,dx\,&=\int_{\mathbb{R}^{N}}\int_{0}^{\vert u\vert^{\alpha-1}u(x,\tau)}(|s|^{\frac{2}{\alpha}}+\varepsilon)^{-\frac{\alpha}{2}}\,|s|^{\frac{1}{\alpha}-1}\,s\,ds\,dx\\
&=\int_{\mathbb{R}^{N}}\int_{0}^{\vert u(x,\tau)\vert^{\alpha}}(s^{\frac{2}{\alpha}}+\varepsilon)^{-\frac{\alpha}{2}}\,s^{\frac{1}{\alpha}}\,ds\,dx\\
&\xrightarrow[\varepsilon\to0]{}\int_{\mathbb{R}^{N}}\int_{0}^{\vert u(x,\tau)\vert^{\alpha}}s^{\frac{1}{\alpha}-1}\,ds\,dx\,=\,\alpha\int_{\mathbb{R}^{N}}\vert u(x,\tau)\vert\,dx\,,
\end{align*}where we have used the Monotone Convergence Theorem in order to pass
to the limit as $\varepsilon\rightarrow0$. Thus, letting $\varepsilon\rightarrow0$
in \eqref{eq:after_rlimit}, we obtain
\begin{align}
[1-(1-\alpha)\,\theta]\sup_{\tau\,\in\,[0,T]}\Vert u(\cdot,\tau)\Vert_{L^{1}(\mathbb{R}^{N})}\,\leq\,\Vert u_{0}\Vert_{L^{1}(\mathbb{R}^{N})}\,+\left(\frac{(\alpha+1)T}{\alpha\theta}\right)^{\frac{1}{\alpha}-1}\iint_{S_{T}}|f|^{\frac{1}{\alpha}}\,dx\,dt\,.\label{eq:presque-fin}
\end{align}
Finally, taking $\theta=\frac{\delta}{(1-\alpha)(1+\delta)}$ with
$\delta>0$ in (\ref{eq:presque-fin}), we arrive at estimate (\ref{eq:Lemma5_3est1}).
In the case $f\equiv0$, one can furthermore let $\delta\to0$ in
(\ref{eq:Lemma5_3est1}), thereby obtaining the second estimate in
the statement of the lemma.\end{proof}

\selectlanguage{british}%
\noindent $\hspace*{1em}$\foreignlanguage{american}{The next lemma
will be instrumental in the proof of the global boundedness result
stated in Theorem \ref{thm:limitatezza-globale} below.}
\selectlanguage{american}%
\begin{lem}
\label{lem:global_Lalphaplusone} Let $f\in L^{\frac{1}{\alpha}+1}(S_{T})$
and $u_{0}\in L^{\alpha+1}(\mathbb{R}^{N})$, and assume that $u$
is an $L^{p}$-integrable solution to problem \eqref{prob:cauchy}
in the sense of Definition \ref{def:Lp-integrable-sol}. Then $u\in L^{p_{\alpha+1}}(S_{T})$.
\end{lem}

\noindent \begin{proof}[\bfseries{Proof}]Let $r>0$ and let $\eta_{r}$
be as in the proof of Lemma \ref{lem:Lalphaplusone-contractive}.
Using the fact that $\eta_{r}=1$ on $B_{r}(0)$ and applying the
Sobolev embedding of Lemma \ref{lem:parabolic-sobolev} to the function
$(u-1)_{+}\,\eta_{r}$, we obtain\begin{align}\label{eq:est:uminusone}
&\int_{0}^{T}\int_{B_{r}(0)}(u-1)_{+}^{p_{\alpha+1}}\,dx\,dt\,\leq\iint_{S_{T}}[(u-1)_{+}\,\eta_{r}]^{p_{\alpha+1}}\,dx\,dt\nonumber\\
&\,\,\,\,\,\,\,\leq\, C\Big[\underset{\tau\,\in\,[0,T]}{\mathrm{ess}\,\sup}\int_{\mathbb{R}^{N}\times\{\tau\}}|(u-1)_{+}\,\eta_{r}|^{\alpha+1}\,dx\Big]^{\frac{p}{N}}\iint_{S_{T}}|\nabla[(u-1)_{+}\,\eta_{r}]|^{p}\,dx\,dt\nonumber\\
&\,\,\,\,\,\,\,\leq\, C\Big[\underset{\tau\,\in\,[0,T]}{\mathrm{ess}\,\sup}\int_{\mathbb{R}^{N}\times\{\tau\}}u_{+}^{\alpha+1}\,dx\Big]^{\frac{p}{N}}\iint_{S_{T}}|\nabla[(u-1)_{+}\,\eta_{r}]|^{p}\,dx\,dt\,,
\end{align}where $C$ is a positive constant depending only on $N$, $\alpha$
and $p$. We now estimate the last integral using (\ref{est:basic_delta})
with $\xi=\nabla u$. This gives \begin{align}\label{eq:aggiunta1}
&\iint_{S_{T}}|\nabla[(u-1)_{+}\,\eta_{r}]|^{p}\,dx\,dt\nonumber\\
&\,\,\,\,\,\,\,\leq C\iint_{S_{T}}|\nabla u\vert^{p}\,\chi_{\{u\,>\,1\}}\,\eta_{r}^{p}\,dx\,dt\,+\,C\iint_{S_{T}}(u-1)_{+}^{p}\,|\nabla\eta_{r}|^{p}\,dx\,dt\nonumber\\
&\,\,\,\,\,\,\,\leq C\iint_{S_{T}}\left[\sum_{i=1}^{N}(\vert\partial_{i}u\vert-\delta_{i})_{+}^{p}\,\chi_{\{u\,>\,1\}}\,\eta_{r}^{p}\,+\,\delta_{\mathrm{max}}^{p}\,\chi_{\{u\,>\,1\}}\,\eta_{r}^{p}\right]dx\,dt\,+\,C\iint_{S_{T}}(u-1)_{+}^{p}\,|\nabla\eta_{r}|^{p}\,dx\,dt\nonumber\\
&\,\,\,\,\,\,\,\leq C\iint_{S_{T}}\sum_{i=1}^{N}(\vert\partial_{i}u\vert-\delta_{i})_{+}^{p}\,\chi_{\{u\,>\,1\}}\,\eta_{r}^{p}\,dx\,dt\,+\,C\,\delta_{\mathrm{max}}^{p}\,\vert S_{T}\cap\{u>1\}\vert\nonumber\\
&\,\,\,\,\,\,\,\,\,\,\,\,\,\,+\,C\iint_{S_{T}}(u-1)_{+}^{p}\,|\nabla\eta_{r}|^{p}\,dx\,dt\,,
\end{align}where $\delta_{\mathrm{max}}:=\max\,\{\delta_{1},...,\delta_{N}\}$.
Using Lemma \ref{lem:energy-general} with $\Omega=B_{2r}(0)$, $F(s)=(s-1)_{+}$,
$\lambda=\alpha+1$, $\tau_{1}=0$, $\tau_{2}=\tau$ and $\varphi\equiv1$,
and proceeding as in the proof of Lemma \ref{lem:energy-classical},
we obtain, in place of estimate \eqref{eq:ultima-est-energy},\begin{align}\label{eq:aggiunta2}
&\int_{\mathbb{R}^{N}\times\{\tau\}}\mathfrak{b}_{\alpha}[u,1]\,\chi_{\{u\,>\,1\}}\,\eta_{r}^{p}\,dx\,+\,\sum_{i=1}^{N}\int_{0}^{\tau}\int_{\mathbb{R}^{N}}(\vert\partial_{i}u\vert-\delta_{i})_{+}^{p}\,\eta_{r}^{p}\,\chi_{\{u\,>\,1\}}\,dx\,dt\nonumber\\
&\,\,\,\,\,\,\,\leq\,C\int_{\mathbb{R}^{N}\cap\,\{u_{0}\,>\,1\}}\mathfrak{b}_{\alpha}[u_{0},1]\,\eta_{r}^{p}\,dx\,+\,C\int_{0}^{\tau}\int_{\mathbb{R}^{N}}(u-1)_{+}^{p}\,\vert\nabla\eta_{r}\vert^{p}\,dx\,dt\nonumber\\
&\,\,\,\,\,\,\,\,\,\,\,\,\,\,+\,C\int_{0}^{\tau}\int_{\mathbb{R}^{N}}(u-1)_{+}\,\vert f\vert\,\eta_{r}\,dx\,dt
\end{align}for all $\tau\in(0,T]$, where the constant $C$ now also depends
on $\Lambda$. Discarding the first term on the left-hand side of
\eqref{eq:aggiunta2}, using the right-hand inequality in (\ref{est:b-all-alpha})
with $v=u$ and $w=1$, and taking the supremum over $\tau\in(0,T]$,
we infer from the previous estimate that\begin{align}\label{eq:aggiunta3}
&\iint_{S_{T}}\sum_{i=1}^{N}(\vert\partial_{i}u\vert-\delta_{i})_{+}^{p}\,\eta_{r}^{p}\,\chi_{\{u\,>\,1\}}\,dx\,dt\nonumber\\
&\,\,\,\,\,\,\,\leq\,C\int_{\mathbb{R}^{N}}(u_{0}^{\frac{\alpha+1}{2}}-1)_{+}^{2}\,\eta_{r}^{p}\,dx\,+\,C\iint_{S_{T}}(u-1)_{+}^{p}\,\vert\nabla\eta_{r}\vert^{p}\,dx\,dt\,+\,C\iint_{S_{T}}(u-1)_{+}\,\vert f\vert\,\eta_{r}\,dx\,dt\,.
\end{align}Combining \eqref{eq:aggiunta1} with \eqref{eq:aggiunta3}, we then
have 
\begin{align*}
\iint_{S_{T}}|\nabla[(u-1)_{+}\,\eta_{r}]|^{p}\,dx\,dt\, & \leq\,C\int_{\mathbb{R}^{N}}(u_{0}^{\frac{\alpha+1}{2}}-1)_{+}^{2}\,\eta_{r}^{p}\,dx\,+\,C\iint_{S_{T}}(u-1)_{+}^{p}\,|\nabla\eta_{r}|^{p}\,dx\,dt\\
 & \quad+\,C\iint_{S_{T}}(u-1)_{+}\,\vert f\vert\,\eta_{r}\,dx\,dt\,+\,C\,\delta_{\mathrm{max}}^{p}\,|S_{T}\cap\{u>1\}|\,.
\end{align*}
We next show that each term on the right-hand side is bounded uniformly
for all $r\geq r_{0}$, where $r_{0}>0$ is sufficiently large. First,
we have
\[
\int_{\mathbb{R}^{N}}(u_{0}^{\frac{\alpha+1}{2}}-1)_{+}^{2}\,\eta_{r}^{p}\,dx\,\leq\int_{\mathbb{R}^{N}}(u_{0})_{+}^{\alpha+1}\,dx\,<+\infty\,.
\]
Moreover, recalling that $\vert\nabla\eta_{r}\vert\leq c/r$, we obtain\\
\begin{align*}
\iint_{S_{T}}(u-1)_{+}^{p}\,|\nabla\eta_{r}|^{p}\,dx\,dt\,\leq\,c^{p}r^{-p}\iint_{S_{T}}(u-1)_{+}^{p}\,dx\,dt\longrightarrow0\,\,\,\,\,\,\,\,\,\,\mathrm{as}\,\,\,r\to\infty\,,
\end{align*}
since $u\in L^{p}(S_{T}).$ Similarly,
\begin{align}
|S_{T}\cap\{u>1\}|\,\leq\iint_{S_{T}\,\cap\,\{u\,>\,1\}}u^{p}\,dx\,dt\,\leq\iint_{S_{T}}u_{+}^{p}\,dx\,dt\,<+\infty\,.\label{est:measure-u-largerthanone}
\end{align}
Finally, Young's inequality yields 
\begin{align*}
\iint_{S_{T}}(u-1)_{+}\,\vert f\vert\,\eta_{r}\,dx\,dt\,\leq\iint_{S_{T}}u_{+}^{\alpha+1}\,dx\,dt\,+\iint_{S_{T}}|f|^{\frac{1}{\alpha}+1}\,dx\,dt\,,
\end{align*}
where the first term on the right-hand side is finite by Lemma \ref{lem:Lalphaplusone-contractive},
while the second one is finite since $f\in L^{\frac{1}{\alpha}+1}(S_{T})$.
Combining \eqref{eq:est:uminusone} with the previous estimates, we
infer that the right-hand side of \eqref{eq:est:uminusone} is bounded
uniformly for all $r\geq r_{0}$, where $r_{0}>0$ is sufficiently
large. Hence, letting $r\rightarrow\infty$ and applying the Monotone
Convergence Theorem to the left-hand side of \eqref{eq:est:uminusone},
we conclude that
\begin{align}
\iint_{S_{T}}(u-1)_{+}^{p_{\alpha+1}}\,dx\,dt\,<+\infty\,.\label{est:uminusone-finite}
\end{align}
Now, recalling that $p_{\alpha+1}>p$, we may estimate 
\begin{align}
\iint_{S_{T}}u_{+}^{p_{\alpha+1}}\,dx\,dt & \,=\iint_{S_{T}\,\cap\,\{u\,>\,1\}}(u-1+1)^{p_{\alpha+1}}\,dx\,dt\,+\iint_{S_{T}\,\cap\,\{u\,\leq\,1\}}u_{+}^{p_{\alpha+1}}\,dx\,dt\nonumber \\
 & \,\leq\,C\iint_{S_{T}\,\cap\,\{u\,>\,1\}}[(u-1)_{+}^{p_{\alpha+1}}+1]\,dx\,dt\,+\iint_{S_{T}\,\cap\,\{u\,\leq\,1\}}u_{+}^{p}\,dx\,dt\nonumber \\
 & \,\leq\,C\iint_{S_{T}}(u-1)_{+}^{p_{\alpha+1}}\,dx\,dt\,+\,C\,|S_{T}\cap\{u>1\}|\,+\iint_{S_{T}}u_{+}^{p}\,dx\,dt\,.\label{eq:aggiunta_4}
\end{align}
The first term on the right-hand side of (\ref{eq:aggiunta_4}) is
finite by \eqref{est:uminusone-finite}, while the last two terms
are finite by \eqref{est:measure-u-largerthanone}. We have thus proved
that $u_{+}\in L^{p_{\alpha+1}}(S_{T})$. Replacing $(u-1)_{+}$ with
$-(u+1)_{-}$ and arguing as above, we also obtain that $u_{-}\in L^{p_{\alpha+1}}(S_{T})$.
This concludes the proof.\end{proof}

\noindent We are now in a position to prove the following global boundedness
result.
\begin{thm}
\label{thm:limitatezza-globale}Let $f\in L^{\frac{1}{\alpha}+1}(S_{T})\cap L^{\widehat{\sigma}}(S_{T})$
for some $\widehat{\sigma}>\frac{N+p}{p}$, let $u_{0}\in L^{\alpha+1}(\mathbb{R}^{N})$
and assume that $p_{\alpha+1}>\alpha+1$.\footnote{We recall that the lower bound for $p_{\alpha+1}$ in the statement
of Theorem \ref{thm:limitatezza-globale} is in fact equivalent to
the condition \eqref{p_lower_bnd}.} Suppose that $u$ is an $L^{p}$-integrable solution to problem \eqref{prob:cauchy}
in the sense of Definition \ref{def:Lp-integrable-sol}. Then, for
every $\theta\in(0,T)$ and every $q\in[\alpha+1,p_{\alpha+1}]$,
we have the explicit upper bound
\begin{equation}
\underset{\mathbb{\mathbb{R}}^{N}\times(\theta,T)}{\mathrm{ess}\,\sup}\,\,u\,\leq\,c\left[(\theta^{-1}+1)^{\frac{N+p}{p}}\int_{\theta/2}^{T}\int_{\mathbb{R}^{N}}u_{+}^{q}\,dx\,dt\,+\int_{\theta/2}^{T}\int_{\mathbb{R}^{N}}|f|^{\sigma}\,dx\,dt\right]^{\mu}+1\,,\label{eq:est:global_sup}
\end{equation}
where
\[
\sigma:=\begin{cases}
\begin{array}{cc}
\frac{\alpha+1}{\alpha} & \mathit{if}\,\,\frac{\alpha+1}{\alpha}>\frac{N+p}{p}\vspace{1mm}\\
\widehat{\sigma} & \mathit{otherwise\,\,\,\,\,\,\,\,\,\,}
\end{array}\end{cases},\,\,\,\,\,\,\,\,\,\,\,\,\,\mu:=\,\frac{p}{qp+N(p-\alpha-1)-\frac{(N+p)}{\sigma}(q-\alpha-1)}\,>0\,,
\]
and $c$ is a positive constant depending only on $N$, $\alpha$,
$p$, $\Lambda$, $\sigma$ and $\max\,\{\delta_{1},\ldots,\delta_{N}\}$.
An analogous lower bound holds for the essential infimum, with $u_{-}$
replacing $u_{+}$ on the right-hand side of $(\ref{eq:est:global_sup})$.
In particular, if $\frac{\alpha+1}{\alpha}>\frac{N+p}{p}$, the assumption
on $f$ reduces to $f\in L^{\frac{1}{\alpha}+1}(S_{T})$.
\end{thm}

\noindent \begin{proof}[\bfseries{Proof}]Let $\theta\in(0,T)$ and
$q\in[\alpha+1,p_{\alpha+1}]$. For $j\in\mathbb{N}_{0}$ we define
\begin{align*}
\theta_{j} & :=\,\theta(1-2^{-(j+1)})\,,\hspace{7mm}\varphi_{j}(t):=\,\min\,\{1,2^{j+2}\,\theta^{-1}(t-\theta_{j})_{+}\}\,,\\
S^{j} & :=\,\mathbb{R}^{N}\times(\theta_{j},T)\,.
\end{align*}
Furthermore, we define the sequences 
\begin{align*}
Y_{j}:=\iint_{S^{j}}(u-k_{j})_{+}^{q}\,dx\,dt\,,\,\,\,\,\,\,\,\,k_{j}:=\,k(1-2^{-j})\,,\,\,\,\,\,\,\,\,\widehat{k}_{j}:=\,\frac{1}{2}(k_{j}+k_{j+1})\,,\,\,\,\,\,\,\,\,\tilde{k}_{j}:=\,\frac{1}{2}(k_{j}+\widehat{k}_{j})\,,
\end{align*}
where $k\geq1$ is a number to be chosen later. From Lemma \ref{lem:Lalphaplusone-contractive}
it follows that $u\in L^{\alpha+1}(S_{T})$. Moreover, Lemma \ref{lem:global_Lalphaplusone}
ensures that $u\in L^{p_{\alpha+1}}(S_{T})$. Hence, by interpolation
in Lebesgue spaces, $u\in L^{q}(S_{T})$, and therefore $Y_{j}$ is
finite for every $j\in\mathbb{N}_{0}$. Letting $r>0$ and taking
$\eta_{r}$ as in the proof of Lemma \ref{lem:Lalphaplusone-contractive},
we may argue as in \eqref{eq:notsolongcalc}$-$\eqref{eq:integ_after_energy}
to obtain

\begin{align*}
 & \int_{\theta_{j+1}}^{T}\int_{B_{r}(0)}(u-k_{j+1})_{+}^{p_{\alpha+1}}\,dx\,dt\\
 & \quad\leq\,c\,b^{j}\Big[\int_{\theta_{j}}^{T}\int_{B_{2r}(0)}(u-\widehat{k}_{j})_{+}^{p}\,|\nabla\eta_{r}|^{p}\,dx\,dt\,+\int_{\theta_{j}}^{T}\int_{B_{2r}(0)}(u^{\frac{\alpha+1}{2}}-\widehat{k}_{j}^{\frac{\alpha+1}{2}})_{+}^{2}\,(\partial_{t}\varphi_{j})_{+}\,dx\,dt\\
 & \quad\qquad\,\,\,\,\,\,+\int_{\theta_{j}}^{T}\int_{B_{2r}(0)}|f|\,(u-\widehat{k}_{j})_{+}\,\eta_{r}^{p}\,\varphi_{j}\,dx\,dt\,+\,\delta_{\mathrm{max}}^{p}\,|S^{j}\cap\{u>\widehat{k}_{j}\}|\Big]^{\frac{N+p}{N}}\\
 & \quad\leq\,c\,b^{j}\Big[r^{-p}\int_{\theta_{j}}^{T}\int_{B_{2r}(0)}(u-\widehat{k}_{j})_{+}^{p}\,dx\,dt\,+\,2^{j+2}\,\theta^{-1}\int_{\theta_{j}}^{T}\int_{B_{2r}(0)}(u^{\frac{\alpha+1}{2}}-\widehat{k}_{j}^{\frac{\alpha+1}{2}})_{+}^{2}\,dx\,dt\\
 & \quad\qquad\,\,\,\,\,\,+\int_{\theta_{j}}^{T}\int_{B_{2r}(0)}|f|\,(u-\widehat{k}_{j})_{+}\,dx\,dt\,+\,\delta_{\mathrm{max}}^{p}\,|S^{j}\cap\{u>\widehat{k}_{j}\}|\Big]^{\frac{N+p}{N}},
\end{align*}
where $c=c(N,\alpha,p,\Lambda)>0$, $b=b(N,p)>1$ and $\delta_{\mathrm{max}}:=\max\,\{\delta_{1},...,\delta_{N}\}$.
Passing to the limit as $r\to\infty$ and recalling that $u\in L^{p}(S_{T})$,
we end up with 
\begin{align}
 & \iint_{S^{j+1}}(u-k_{j+1})_{+}^{p_{\alpha+1}}\,dx\,dt\label{est:Sjplusone-itegral}\\
 & \quad\leq\,c\,b^{j}\Big[\theta^{-1}\iint_{S^{j}}(u^{\frac{\alpha+1}{2}}-\widehat{k}_{j}^{\frac{\alpha+1}{2}})_{+}^{2}\,dx\,dt\,+\iint_{S^{j}}|f|\,(u-\widehat{k}_{j})_{+}\,dx\,dt\,+\,\delta_{\mathrm{max}}^{p}\,|S^{j}\cap\{u>\widehat{k}_{j}\}|\Big]^{\frac{N+p}{N}}.\nonumber 
\end{align}
If $\frac{\alpha+1}{\alpha}>\frac{N+p}{p}$, we take $\sigma:=\frac{\alpha+1}{\alpha}$.
Otherwise, we define $\sigma:=\widehat{\sigma}>\frac{N+p}{p}\geq\frac{\alpha+1}{\alpha}$,
and in both cases we have $1<\sigma'\leq\alpha+1$. We now turn to
the estimation of the right-hand side of (\ref{est:Sjplusone-itegral}).
First, note that 
\begin{align*}
(u-\widehat{k}_{j})_{+}^{\sigma'} & \,=\,(\widehat{k}_{j}-\tilde{k}_{j})^{\sigma'-(\alpha+1)}\,(\widehat{k}_{j}-\tilde{k}_{j})^{\alpha+1-\sigma'}\,(u-\widehat{k}_{j})_{+}^{\sigma'}\\
 & \,\leq\,(\widehat{k}_{j}-\tilde{k}_{j})^{\sigma'-(\alpha+1)}\,(u-\tilde{k}_{j})_{+}^{\alpha+1}\\
 & \,=\,k^{\sigma'-(\alpha+1)}\,2^{(j+3)(\alpha+1-\sigma')}\,(u-\tilde{k}_{j})_{+}^{\alpha+1}\,.
\end{align*}
Combining this estimate with Hölder's inequality, we get 
\begin{align}
\iint_{S^{j}}|f|\,(u-\widehat{k}_{j})_{+}\,dx\,dt & \,\leq\,\Vert f\Vert_{L^{\sigma}(S^{0})}\Big(\iint_{S^{j}}(u-\widehat{k}_{j})_{+}^{\sigma'}\,dx\,dt\Big)^{\frac{1}{\sigma'}}\nonumber \\
 & \,\leq\,8^{\alpha}\,\Vert f\Vert_{L^{\sigma}(S^{0})}\,k^{1\,-\,\frac{\alpha+1}{\sigma'}}\,2^{j\alpha}\Big(\iint_{S^{j}}(u-\tilde{k}_{j})_{+}^{\alpha+1}\,dx\,dt\Big)^{\frac{1}{\sigma'}},\label{eq:est:f-term}
\end{align}
where, in the last line, we have also used that $\sigma'>1$. Furthermore,
as already seen in \eqref{eq:est:term2}, we have 
\begin{align}
(u^{\frac{\alpha+1}{2}}-\widehat{k}_{j}^{\frac{\alpha+1}{2}})_{+}^{2}\,\leq\,c\,(u-\widehat{k}_{j})_{+}^{\alpha+1}\,+\,c\,\widehat{k}_{j}^{\alpha+1}\,\chi_{\{u\,>\,\widehat{k}_{j}\}}\,.\label{aaterigen}
\end{align}
In addition, proceeding as in (\ref{eq:superlevel}), we find that
\begin{align}
|S^{j}\cap\{u>\widehat{k}_{j}\}|\,\leq\iint_{S^{j}\,\cap\,\{u\,>\,\widehat{k}_{j}\}}\frac{(u-\tilde{k}_{j})_{+}^{\alpha+1}}{(\widehat{k}_{j}-\tilde{k}_{j})_{+}^{\alpha+1}}\,dx\,dt\,\leq\,\frac{2^{(j+3)(\alpha+1)}}{k^{\alpha+1}}\iint_{S^{j}}(u-\tilde{k}_{j})_{+}^{\alpha+1}\,dx\,dt\,.\label{est:measure-global}
\end{align}
Inserting \eqref{eq:est:f-term}, \eqref{aaterigen} and \eqref{est:measure-global}
into \eqref{est:Sjplusone-itegral}, we obtain 
\begin{align}
 & \iint_{S^{j+1}}(u-k_{j+1})_{+}^{p_{\alpha+1}}\,dx\,dt\label{eq:aather_btut}\\
 & \,\,\leq\,c\,b^{j}\Big[(\theta^{-1}+1+\delta_{\mathrm{max}}^{p})\iint_{S^{j}}(u-\tilde{k}_{j})_{+}^{\alpha+1}\,dx\,dt\,+\,\Vert f\Vert_{L^{\sigma}(S^{0})}\Big(\iint_{S^{j}}(u-\tilde{k}_{j})_{+}^{\alpha+1}\,dx\,dt\Big)^{\frac{1}{\sigma'}}\Big]^{\frac{N+p}{N}}\nonumber \\
 & \,\,\leq\,c\,b^{j}\Big[(\theta^{-1}+1+\delta_{\mathrm{max}}^{p})\Big(\iint_{S^{j}}(u-\tilde{k}_{j})_{+}^{\alpha+1}\,dx\,dt\Big)^{\frac{1}{\sigma}}+\,\Vert f\Vert_{L^{\sigma}(S^{0})}\Big]^{\frac{N+p}{N}}\Big(\iint_{S^{j}}(u-\tilde{k}_{j})_{+}^{\alpha+1}\,dx\,dt\Big)^{\frac{N+p}{N\sigma'}},\nonumber 
\end{align}
where $b>1$ now also depends on $\alpha$. Note that, in the previous
estimate, we have also used the fact that $k\geq1$ to bound from
above all negative powers of $k$ by $1$, and the inequality $\tilde{k}_{j}<\widehat{k}_{j}$
to replace $\widehat{k}_{j}$ with $\tilde{k}_{j}$. Exploiting the
fact that $\tilde{k}_{j}\geq\tilde{k}_{0}=\tfrac{k}{8}\geq\tfrac{1}{8}$
and that $\alpha+1\leq q\leq p_{\alpha+1}$, we may write\begin{align}\label{eq:seafoid1}
\iint_{S^{j}}(u-\tilde{k}_{j})_{+}^{\alpha+1}\,dx\,dt\,&\leq\iint_{S^{0}\,\cap\,\{u\,>\,\tilde{k}_{0}\}}u^{\alpha+1}\,dx\,dt\,\leq\,\tilde{k}_{0}^{\alpha+1-q}\iint_{S^{0}\,\cap\,\{u\,>\,\tilde{k}_{0}\}}u^{q}\,dx\,dt\nonumber\\
&\leq\,8^{p_{\alpha+1}-\alpha-1}\iint_{S^{0}}u_{+}^{q}\,dx\,dt\,.
\end{align}By Hölder's inequality, we can estimate the same integral also as
follows: 
\begin{align}
\iint_{S^{j}}(u-\tilde{k}_{j})_{+}^{\alpha+1}\,dx\,dt\, & \leq\Big(\iint_{S^{j}}(u-\tilde{k}_{j})_{+}^{q}\,dx\,dt\Big)^{\frac{\alpha+1}{q}}\,|S^{j}\cap\{u>\tilde{k}_{j}\}|^{1\,-\,\frac{\alpha+1}{q}}\nonumber \\
 & \leq\,Y_{j}^{\frac{\alpha+1}{q}}\,|S^{j}\cap\{u>\tilde{k}_{j}\}|^{1\,-\,\frac{\alpha+1}{q}}\,.\label{eq:seafoid2}
\end{align}
Using \eqref{eq:seafoid1} and \eqref{eq:seafoid2} to estimate the
two occurrences of the integral on the right-hand side of \eqref{eq:aather_btut},
we obtain

\begin{align*}
\iint_{S^{j+1}}(u-k_{j+1})_{+}^{p_{\alpha+1}}\,dx\,dt\, & \leq\,c\,b^{j}\Big[(\theta^{-1}+1)\,\Vert u_{+}\Vert_{L^{q}(S^{0})}^{\frac{q}{\sigma}}\,+\,\Vert f\Vert_{L^{\sigma}(S^{0})}\Big]^{\frac{N+p}{N}}\\
 & \quad\,\times Y_{j}^{\frac{(N+p)(\alpha+1)}{qN\sigma'}}\,|S^{j}\cap\{u>\tilde{k}_{j}\}|^{\frac{N+p}{N\sigma'}(1\,-\,\frac{\alpha+1}{q})}\,,
\end{align*}
where the constant $c$ now also depends on $\delta_{\mathrm{max}}$.
Applying Hölder's inequality together with the previous estimate,
we get 
\begin{align}
Y_{j+1} & \,\leq\Big[\iint_{S^{j+1}}(u-k_{j+1})_{+}^{p_{\alpha+1}}\,dx\,dt\Big]^{\frac{q}{p_{\alpha+1}}}\,|S^{j+1}\cap\{u>k_{j+1}\}|^{1\,-\,\frac{q}{p_{\alpha+1}}}\nonumber \\
 & \,\leq\,c\,b^{j}\Big[(\theta^{-1}+1)\,\Vert u_{+}\Vert_{L^{q}(S^{0})}^{\frac{q}{\sigma}}\,+\,\Vert f\Vert_{L^{\sigma}(S^{0})}\Big]^{\frac{q(N+p)}{Np_{\alpha+1}}}\,Y_{j}^{\frac{(N+p)(\alpha+1)}{p_{\alpha+1}N\sigma'}}\nonumber \\
 & \,\quad\,\times|S^{j}\cap\{u>\tilde{k}_{j}\}|^{\frac{N+p}{p_{\alpha+1}N\sigma'}\,(q-\alpha-1)\,+1\,-\,\frac{q}{p_{\alpha+1}}}\,,\label{eq:almostRecursive}
\end{align}
where the constants $c>0$ and $b>1$ now also depend on $q$. Furthermore,
arguing as in (\ref{eq:superlevel}), we have
\begin{align*}
|S^{j}\cap\{u>\tilde{k}_{j}\}|\,\leq\,8^{q}\,2^{jq}\,k^{-q}\,Y_{j}\,.
\end{align*}
Combining this estimate with \eqref{eq:almostRecursive}, we obtain
\begin{equation}
Y_{j+1}\,\leq\,C_{u,f}\,b^{j}\,Y_{j}^{1+\beta}\,,\label{eq:passaggio-chiave}
\end{equation}
where $b>1$ now also depends on $\sigma$ and
\begin{align}
C_{u,f}\, & =\,c\Big[(\theta^{-1}+1)\,\Vert u_{+}\Vert_{L^{q}(S^{0})}^{\frac{q}{\sigma}}\,+\,\Vert f\Vert_{L^{\sigma}(S^{0})}\Big]^{\frac{q(N+p)}{Np_{\alpha+1}}}\,k^{-H}\,,\label{eq:key-constant}\\
H & \,=\,q\Big[\,\frac{N+p}{p_{\alpha+1}N\sigma'}\,(q-\alpha-1)+1-\frac{q}{p_{\alpha+1}}\,\Big]\,,\nonumber \\
\beta & \,=\,\frac{q}{Np_{\alpha+1}}\Big(p-\frac{N+p}{\sigma}\Big)\,.\nonumber 
\end{align}
Although the constants $b$ and $c$ appearing in \eqref{eq:passaggio-chiave},
\eqref{eq:key-constant} and in the sequel depend on $q$, the admissible
range of $q$ enables us to choose them so that they depend only on
the data. Moreover, this range also implies that $H$ is positive,
while the range of $\sigma$ guarantees that $\beta>0$. If $C_{u,f}=0$,
we immediately deduce that $u\leq0$ almost everywhere in $S^{0}=\mathbb{R}^{N}\times\left(\frac{\theta}{2},T\right)$.
If instead $C_{u,f}>0$, then, by Lemma \ref{lem:fastconvg}, the
sequence $\{Y_{j}\}$ converges to zero provided that
\begin{align*}
\Vert u_{+}\Vert_{L^{q}(S^{0})}^{q}\,=\,Y_{0}\,\leq\,C_{u,f}^{-\,\frac{1}{\beta}}\,b^{-\,\frac{1}{\beta^{2}}}\,,
\end{align*}
which can equivalently be stated as
\begin{align*}
k\,\geq\,c\,\Vert u_{+}\Vert_{L^{q}(S^{0})}^{\frac{q\beta}{H}}\Big[(\theta^{-1}+1)\,\Vert u_{+}\Vert_{L^{q}(S^{0})}^{\frac{q}{\sigma}}\,+\,\Vert f\Vert_{L^{\sigma}(S^{0})}\Big]^{\frac{q(N+p)}{NH\,p_{\alpha+1}}}\,=:\,\tilde{k}\,.
\end{align*}
Estimating $\tilde{k}$ from above by means of \eqref{eq:elem_ineq}
with $\tau=\frac{q(N+p)}{NH\,p_{\alpha+1}}$, and applying Young's
inequality, we obtain 
\begin{align*}
\tilde{k} & \,\leq\,c\,(\theta^{-1}+1)^{\frac{q(N+p)}{NH\,p_{\alpha+1}}}\,\Vert u_{+}\Vert_{L^{q}(S^{0})}^{\frac{q\beta}{H}\,+\,\frac{q^{2}(N+p)}{\sigma NH\,p_{\alpha+1}}}\,+\,c\,\Vert u_{+}\Vert_{L^{q}(S^{0})}^{\frac{q\beta}{H}}\,\Vert f\Vert_{L^{\sigma}(S^{0})}^{\frac{q(N+p)}{NH\,p_{\alpha+1}}}\\
 & \,\leq\,c\,(\theta^{-1}+1)^{\frac{q(N+p)}{NH\,p_{\alpha+1}}}\,\Vert u_{+}\Vert_{L^{q}(S^{0})}^{\frac{q\beta}{H}\,+\,\frac{q^{2}(N+p)}{\sigma NH\,p_{\alpha+1}}}\,+\,c\,\Vert f\Vert_{L^{\sigma}(S^{0})}^{\frac{\sigma\beta}{H}\,+\,\frac{q(N+p)}{NH\,p_{\alpha+1}}}\\
 & \,=\,c\,(\theta^{-1}+1)^{\sigma\mu\,-\,\frac{\sigma\beta}{H}}\Big(\iint_{S^{0}}u_{+}^{q}\,dx\,dt\Big)^{\mu}+\,c\,\Big(\iint_{S^{0}}|f|^{\sigma}\,dx\,dt\Big)^{\mu}\\
 & \,\leq\,c\Big[(\theta^{-1}+1)^{\sigma\,-\,\frac{\sigma\beta}{\mu H}}\iint_{S^{0}}u_{+}^{q}\,dx\,dt\,+\,\iint_{S^{0}}|f|^{\sigma}\,dx\,dt\Big]^{\mu}+1\\
 & \,=:\,\ell\,,
\end{align*}
where we denote
\[
\mu\,:=\,\frac{\beta}{H}\,+\,\frac{q(N+p)}{\sigma NH\,p_{\alpha+1}}\,=\,\frac{qp}{NH\,p_{\alpha+1}}\,.
\]
Therefore, if $k=\ell$, we have 
\[
\int_{\theta}^{T}\int_{\mathbb{R}^{N}}(u-k)_{+}^{q}\,dx\,dt\,\leq\,Y_{j}\to0\,\,\,\,\,\,\,\,\mathrm{as}\,\,j\to\infty\,,
\]
and hence $u\leq\ell\,$ a.e. in $\mathbb{R}^{N}\times(\theta,T)$.
A direct calculation shows that
\begin{align*}
\sigma-\frac{\sigma\beta}{\mu H}\,=\,\frac{N+p}{p}\,\,\,\,\,\,\,\,\,\,\,\,\,\,\,\,\mathrm{and}\,\,\,\,\,\,\,\,\,\,\,\,\,\,\,\,\mu\,=\,\frac{p}{qp+N(p-\alpha-1)-\frac{(N+p)}{\sigma}(q-\alpha-1)}\,,
\end{align*}
thus proving that the upper bound in \eqref{eq:est:global_sup} holds.\\
\foreignlanguage{british}{$\hspace*{1em}$}Finally, noting that $-u$
solves the Cauchy problem obtained from \eqref{prob:cauchy} by replacing
$f$ and $u_{0}$ with $-f$ and $-u_{0}$, respectively, we also
obtain the analogous lower bound for the essential infimum of $u$
stated in the theorem.\end{proof}

\selectlanguage{british}%
\noindent \appendix
\titleformat{\section}   
{\normalfont\Large\bfseries}   
{Appendix.}  
{0.3em}
{}
\section{Continuity in time}\label{sec:app:time-cont}$\hspace*{1em}$\foreignlanguage{american}{In this appendix, we show
that time-continuity does not have to be assumed explicitly in the
definition of a weak solution to (\ref{eq:diffusion}), provided that
both the solution itself and the source term $f$ on the right-hand
side of the equation satisfy some other integrability conditions.
In fact, in many cases the continuity in time of a solution $u$ as
a map into $L_{\mathrm{loc}}^{\alpha+1}(\Omega)$ can be deduced from
weaker assumptions. As a consequence, the local boundedness results
established in Section \ref{sec:Local-boundedness} remain valid for
a somewhat larger class of solutions. To make this precise, we begin
with the following definition.}
\selectlanguage{american}%
\begin{defn}
\label{def:somewhat-weaker} Let $f\in L_{\mathrm{loc}}^{1}(\Omega_{T})$.
We say that a function $u:\Omega_{T}\to\mathbb{R}$ is an \textit{$L^{\alpha+1}$-integrable
solution} to \eqref{eq:diffusion} if $u\in L^{p}(0,T;W^{1,p}(\Omega))\cap L^{\alpha+1}(\Omega_{T})$
and 
\begin{align}
 & \iint_{\Omega_{T}}\left(\langle A(x,t,\nabla u),\nabla\varphi\rangle-\vert u\vert^{\alpha-1}u\,\partial_{t}\varphi\right)dx\,dt\,=\,\iint_{\Omega_{T}}f\varphi\,dx\,dt\label{eq:weak_form-again}
\end{align}
for all $\varphi\in C_{\textnormal{0}}^{\infty}(\Omega_{T})$, where
the vector field $A:\Omega_{T}\times\mathbb{R}^{N}\to\mathbb{R}^{N}$
is defined by (\ref{eq:def:F}) and (\ref{eq:vector_field}).
\end{defn}

\selectlanguage{british}%
\noindent $\hspace*{1em}$\foreignlanguage{american}{Note that the
only difference between this notion of solutions and the weak solutions
introduced in Definition \ref{def:weaksol} is that the time-continuity
condition $u\in C^{0}([0,T];L^{\alpha+1}(\Omega))$ has now been replaced
by the weaker requirement that $u\in L^{\alpha+1}(\Omega_{T})$. Below,
we shall prove that every $L^{\alpha+1}$-integrable solution $u$
of (\ref{eq:diffusion}) belongs to $C^{0}([0,T];L_{\mathrm{loc}}^{\alpha+1}(\Omega))$,
provided that the following additional assumption holds:
\begin{align}
u\in L^{\mathfrak{q}}(\Omega_{T})\,\,\,\textnormal{ and \,\,\,}f\in L^{\mathfrak{q}'}(\Omega_{T})\,\,\,\textnormal{ for some }\mathfrak{q}>1\,.\label{assumpt:extra}
\end{align}
This condition shows that there is an interplay between the regularity
assumptions on $u$ and $f$. More precisely, the main result of this
appendix is the following.}
\selectlanguage{american}%
\begin{thm}
\noindent \label{thm:cont_into_Lbetaplusone} Assume that $(\ref{eq:coeff_limit})$
holds, and let $u$ be an $L^{\alpha+1}$-integrable solution to \eqref{eq:diffusion}
in the sense of Definition \ref{def:somewhat-weaker}, under the additional
assumption $(\mathrm{\ref{assumpt:extra}})$. Then 
\[
\vert u\vert^{\alpha-1}u\,\in\,C^{0}([0,T];L_{\mathrm{loc}}^{\frac{1}{\alpha}+1}(\Omega))\,\,\,\,\,\,\,\,\,\,\,\mathrm{\mathit{and}}\,\,\,\,\,\,\,\,\,\,\,u\,\in\,C^{0}([0,T];L_{\mathrm{loc}}^{\alpha+1}(\Omega))\,.
\]
\end{thm}

\selectlanguage{british}%
\noindent \begin{brem}[\textbf{Applicability of Theorem~\ref{thm:cont_into_Lbetaplusone}}]\foreignlanguage{american}{Since
any solution in the sense of Definition \ref{def:somewhat-weaker}
belongs to $L^{P}(\Omega_{T})$, with $P:=\max\,\{\alpha+1,p\}$,
condition (\ref{assumpt:extra}) is guaranteed, for instance, by assuming
that $f\in L^{P'}(\Omega_{T})$. However, for solutions with higher
integrability, we see that it is possible to weaken the assumption
on $f$. For source terms satisfying \eqref{assumpt:f}, the condition
\eqref{assumpt:extra} holds, for example, if
\begin{align}
u\in L^{\frac{N+p}{N}}(\Omega_{T})\,.\label{bokachuta2}
\end{align}
In the so-called \textit{fast diffusion regime} $p_{\alpha+1}\leq\alpha+1$,
all solutions in the sense of Definition \ref{def:somewhat-weaker}
satisfy \eqref{bokachuta2}, since 
\begin{align}
\frac{N+p}{N}\,<\,\frac{pN+p(\alpha+1)}{N}\,=\,p_{\alpha+1}\,\leq\,\alpha+1\,.\label{purearumenica}
\end{align}
Therefore, in the fast diffusion regime, every $L^{\alpha+1}$-integrable
solution to \eqref{eq:diffusion} is continuous in time as a map $[0,T]\to L_{\mathrm{loc}}^{\alpha+1}(\Omega)$,
provided, for instance, that $f$ satisfies (\ref{assumpt:f}), by
Theorem \ref{thm:cont_into_Lbetaplusone} applied with $\mathfrak{q}=\sigma'$.
Consequently, the local boundedness results of Theorems \ref{thm:p_small1}
and \ref{thm:p_small2} also apply to $L^{\alpha+1}$-integrable solutions.}\\
$\hspace*{1em}$\foreignlanguage{american}{In the \textit{slow diffusion
regime} $p_{\alpha+1}>\alpha+1$, by Lemma \ref{lem:parabolic-sobolev}
and the first inequality in \eqref{purearumenica}, we have the inclusions
\begin{align*}
L^{p}(0,T;W^{1,p}(\Omega))\cap L^{\infty}(0,T;L^{\alpha+1}(\Omega))\,\subset\,L_{\mathrm{loc}}^{p_{\alpha+1}}(\Omega_{T})\,\subset L_{\mathrm{loc}}^{\frac{N+p}{N}}(\Omega_{T})\,.
\end{align*}
Hence, in the slow diffusion regime, a local application of Theorem
\ref{thm:cont_into_Lbetaplusone} shows that every solution $u$ in
the sense of Definition \ref{def:somewhat-weaker} is actually in
$C_{\mathrm{loc}}^{0}((0,T);L_{\mathrm{loc}}^{\alpha+1}(\Omega))$,
provided that $u$ also belongs to $L^{\infty}(0,T;L^{\alpha+1}(\Omega))$
and that $f$ satisfies (\ref{assumpt:f}). Consequently, the local
boundedness results of Theorems \ref{theo:local_bdd_large_bar_p},
\ref{thm:loc_bdd_p_large_LP} and \ref{thm:better_est} remain valid
for all \textit{$L^{\alpha+1}$-}integrable solutions to \eqref{eq:diffusion}
that are also in $L^{\infty}(0,T;L^{\alpha+1}(\Omega))$.\end{brem}}

\noindent \begin{brem}\foreignlanguage{american}{For convenience,
in Definition \ref{def:somewhat-weaker} and assumption (\ref{assumpt:extra})
we have required global integrability properties. However, the local
boundedness results in Section \ref{sec:Local-boundedness} clearly
remain valid also for solutions $u$ and source terms $f$ satisfying
only analogous \textit{local} integrability properties, as revealed
by a careful inspection of the arguments that follow in this appendix,
as well as by the fact that the proofs in Section \ref{sec:Local-boundedness}
use the energy estimate }\eqref{eq:energy-classical}\foreignlanguage{american}{
only on cylinders compactly contained in $\Omega_{T}$.}\\
$\hspace*{1em}$\foreignlanguage{american}{In what follows, in order
to prove the desired time-continuity on the whole interval $[0,T]$,
i.e. including time zero, it seems necessary to have global integrability
properties in time. The arguments are also clearer when this type
of integrability is assumed. The local boundedness results obtained
under our assumptions, on the other hand, also hold for solutions
satisfying only the corresponding local integrability properties in
time, since a simple translation in time leads to the setting considered
in this appendix.\end{brem}\medskip{}
}

\noindent $\hspace*{1em}$\foreignlanguage{american}{Below we present
the arguments leading to the result of Theorem \ref{thm:cont_into_Lbetaplusone}.
The strategy of the proof is adapted from \cite{CiaVeVe} and \cite{St}.
We start with the following lemma.}
\selectlanguage{american}%
\begin{lem}
\label{lem:time-cont} Assume that $(\ref{eq:coeff_limit})$ holds,
and let $u$ be an $L^{\alpha+1}$-integrable solution to \eqref{eq:diffusion}
in the sense of Definition \ref{def:somewhat-weaker}, under the additional
assumption $(\mathrm{\ref{assumpt:extra}})$. Define 
\begin{align*}
\mathcal{V}:=\big\{ w\in L^{\max\,\{\alpha\,+\,1,\,\mathfrak{q}\}}(\Omega_{T})\,|\,w\in L^{p}(0,T;W^{1,p}(\Omega)),\,\,\partial_{t}w\in L^{\alpha+1}(0,T;L_{\mathrm{loc}}^{\alpha+1}(\Omega))\big\}\,.
\end{align*}
Then, for every $\zeta\in C_{0}^{\infty}(\Omega_{T};[0,\infty))$
and every $w\in\mathcal{V}$, we have 
\begin{align}
\iint_{\Omega_{T}}\partial_{t}\zeta\,\mathfrak{b}_{\alpha}[u,w]\,dx\,dt\, & =\iint_{\Omega_{T}}\left[\langle A(x,t,\nabla u),\nabla[\zeta(u-w)]\rangle+\zeta(\vert u\vert^{\alpha-1}u-\vert w\vert^{\alpha-1}w)\,\partial_{t}w\right]dx\,dt\nonumber \\
 & \quad-\iint_{\Omega_{T}}f\zeta(u-w)\,dx\,dt\,.\label{eq:time_1}
\end{align}
\end{lem}

\noindent \begin{proof}[\bfseries{Proof}]Let $w\in\mathcal{V}$,
$\zeta\in C_{0}^{\infty}(\Omega_{T};[0,\infty))$, $h\in(0,T)$ and
choose 
\[
\varphi\,=\,\zeta\,(w-u_{h})
\]
as a test function in \eqref{eq:weak_form-again}. The integrability
properties of $u$ and $w$, together with the $L^{p}$-integrability
of each $\partial_{i}u$ and $\partial_{i}w$ and the structure conditions
in (\ref{eq:structure_conditions}) for the vector field $A$ ensure
that the test function can be justified by approximation with smooth
compactly supported test functions. Our goal is to pass to the limit
as $h\rightarrow0$. To this end, we observe that $f$ and $\varphi$
have Hölder-conjugate integrability exponents. It follows from Lemma
\ref{lem:expmolproperties}, Remark \ref{enu:expmol_local_integ}
and the aforementioned integrability properties that
\begin{align*}
\iint_{\Omega_{T}}\langle A(x,t,\nabla u),\nabla\varphi\rangle\,dx\,dt\, & \xrightarrow[h\to0]{}\iint_{\Omega_{T}}\langle A(x,t,\nabla u),\nabla[\zeta(w-u)]\rangle\,dx\,dt\,,\\
\iint_{\Omega_{T}}f\varphi\,dx\,dt\, & \xrightarrow[h\to0]{}\iint_{\Omega_{T}}f\zeta(w-u)\,dx\,dt\,.
\end{align*}
Note that Lemma \ref{lem:expmolproperties} (ii) implies 
\[
(\vert u_{h}\vert^{\alpha-1}u_{h}-\vert u\vert^{\alpha-1}u)\,\partial_{t}u_{h}\,\leq\,0\,,
\]
and hence the integral involving the time derivative can be treated
as follows:\begin{align*}
&\iint_{\Omega_{T}}\vert u\vert^{\alpha-1}u\,\partial_{t}\varphi\,dx\,dt\,=\iint_{\Omega_{T}}\zeta\,\vert u\vert^{\alpha-1}u\,\partial_{t}w\,dx\,dt\,-\iint_{\Omega_{T}}\zeta\,\vert u_{h}\vert^{\alpha-1}u_{h}\,\partial_{t}u_{h}\,dx\,dt\\
&\,\,\,\,\,\,\,+\iint_{\Omega_{T}}\zeta(\vert u_{h}\vert^{\alpha-1}u_{h}-\vert u\vert^{\alpha-1}u)\,\partial_{t}u_{h}\,dx\,dt\,+\iint_{\Omega_{T}}\partial_{t}\zeta\,\vert u\vert^{\alpha-1}u\,(w-u_{h})\,dx\,dt\\
&\leq\iint_{\Omega_{T}}\zeta\,\vert u\vert^{\alpha-1}u\,\partial_{t}w\,dx\,dt\,+\iint_{\Omega_{T}}\tfrac{1}{\alpha+1}\,\partial_{t}\zeta\,\vert u_{h}\vert^{\alpha+1}\,dx\,dt\,+\iint_{\Omega_{T}}\partial_{t}\zeta\,\vert u\vert^{\alpha-1}u\,(w-u_{h})\,dx\,dt\\
&\xrightarrow[h\to0]{}\iint_{\Omega_{T}}\zeta\,\vert u\vert^{\alpha-1}u\,\partial_{t}w\,dx\,dt\,+\iint_{\Omega_{T}}\partial_{t}\zeta\left(\tfrac{1}{\alpha+1}\,\vert u\vert^{\alpha+1}\,+\,\vert u\vert^{\alpha-1}u\,(w-u)\right)dx\,dt\\
&\,\,\,\,\,\,\,\,\,\,\,\,\,\,=\iint_{\Omega_{T}}\zeta(\vert u\vert^{\alpha-1}u-\vert w\vert^{\alpha-1}w)\,\partial_{t}w\,dx\,dt\,-\iint_{\Omega_{T}}\partial_{t}\zeta\,\mathfrak{b}_{\alpha}[u,w]\,dx\,dt\,.
\end{align*}This proves the inequality with ``$\leq$'' in \eqref{eq:time_1}.
The reverse inequality is obtained in the same way by taking $\varphi=\zeta(w-u_{\bar{h}})$
as the test function.\end{proof}

\noindent Now we can prove the main result of this appendix.

\noindent \begin{proof}[\bfseries{Proof of Theorem~\ref{thm:cont_into_Lbetaplusone}}]By
Lemma \ref{lem:equivalent_time_cont} it is sufficient to prove the
time-continuity for either $\vert u\vert^{\alpha-1}u$ or $u$. We
prove continuity on the interval $[0,\tfrac{1}{2}T]$ and later explain
how the argument can be modified to show continuity also on $[\tfrac{1}{2}T,T]$,
thus completing the proof. We first note that due to Lemma \ref{lem:expmolproperties},
$w:=u_{\bar{h}}$ belongs to the set of admissible comparison functions
$\mathcal{V}$ of Lemma \ref{lem:time-cont}. Furthermore, Lemma \ref{lem:expmolproperties}
(iv) and Remark \ref{enu:expmol_local_integ} guarantee that $w$
belongs to $C^{0}([0,T];L_{\mathrm{loc}}^{\alpha+1}(\Omega))$. Thus,
by Lemma \ref{lem:equivalent_time_cont} we have that $\vert w\vert^{\alpha-1}w$
belongs to $C^{0}([0,T];L_{\mathrm{loc}}^{1+1/\alpha}(\Omega))$.
For a compact set $K\subset\Omega$ we take $\eta\in C_{0}^{\infty}(\Omega;[0,1])$
such that $\eta=1$ on $K$ and $|\nabla\eta|\leq C_{K}$. Moreover,
we take $\psi\in C^{\infty}([0,T];[0,1])$ with $\psi=1$ on $[0,\tfrac{1}{2}T]$,
$\psi=0$ on $[\tfrac{3}{4}T,T]$ and $|\psi'|\leq\tfrac{8}{T}$.
For $\tau\in(0,\tfrac{1}{2}T)$ and $\varepsilon>0$ so small that
$\tau+\varepsilon<\tfrac{T}{2}$, we define 
\begin{align*}
\chi_{\varepsilon}^{\tau}(t)=\begin{cases}
\,0 & \mathrm{if}\,\,t<\tau\,,\\
\,\varepsilon^{-1}(t-\tau) & \mathrm{if}\,\,t\in[\tau,\tau+\varepsilon]\,,\\
\,1 & \mathrm{if}\,\,t>\tau+\varepsilon\,.
\end{cases}
\end{align*}
We use \eqref{eq:time_1} with $\zeta=\eta\,\chi_{\varepsilon}^{\tau}\,\psi$
and $w=u_{\bar{h}}$ to obtain\begin{align*}
&\varepsilon^{-1}\int_{\tau}^{\tau+\varepsilon}\int_{\Omega}\mathfrak{b}_{\alpha}[u,u_{\bar{h}}]\,\eta\,dx\,dt\,=\iint_{\Omega_{T}}\langle A(x,t,\nabla u),\nabla[\eta(u-u_{\bar{h}})]\rangle\,\chi_{\varepsilon}^{\tau}\,\psi\,dx\,dt\\
&\,\,\,\,\,\,\,+\iint_{\Omega_{T}}\eta\,\chi_{\varepsilon}^{\tau}\,\psi\,(\vert u\vert^{\alpha-1}u-\vert u_{\bar{h}}\vert^{\alpha-1}u_{\bar{h}})\,\partial_{t}u_{\bar{h}}\,dx\,dt\,-\iint_{\Omega_{T}}\mathfrak{b}_{\alpha}[u,u_{\bar{h}}]\,\eta\,\chi_{\varepsilon}^{\tau}\,\psi'\,dx\,dt\\
&\,\,\,\,\,\,\,-\iint_{\Omega_{T}}f\,\eta\,\chi_{\varepsilon}^{\tau}\,\psi\,(u-u_{\bar{h}})\,dx\,dt\\
&\leq\,\sum_{i=1}^{N}\iint_{\mathrm{supp}\,\eta\,\times\,(0,T)}a_{i}\,(|\partial_{i}u|-\delta_{i})_{+}^{p-1}\,(|\partial_{i}u-(\partial_{i}u)_{\bar{h}}|\,+\,|\partial_{i}\eta|\,|u-u_{\bar{h}}|)\,dx\,dt\\
&\,\,\,\,\,\,\,+\,\frac{8}{T}\iint_{\mathrm{supp}\,\eta\,\times\,(\frac{1}{2}T,\frac{3}{4}T)}\mathfrak{b}_{\alpha}[u,u_{\bar{h}}]\,dx\,dt\,+\iint_{\mathrm{supp}\,\eta\,\times\,(0,T)}|f|\,|u-u_{\bar{h}}|\,dx\,dt\,.
\end{align*}Here we were able to drop the term involving $\partial_{t}u_{\bar{h}}$
since Lemma \ref{lem:expmolproperties} (ii) shows that the factors
$\partial_{t}u_{\bar{h}}$ and $(\vert u\vert^{\alpha-1}u-\vert u_{\bar{h}}\vert^{\alpha-1}u_{\bar{h}})$
have opposite signs, and hence their product is non-positive. Passing
to the limit as $\varepsilon\to0$, we see that\begin{align}\label{eq:gs}
&\int_{K}\mathfrak{b}_{\alpha}[u,u_{\bar{h}}](x,\tau)\,dx\nonumber\\
&\,\,\,\,\,\,\,\leq\,C_{K}\,\sum_{i=1}^{N}\iint_{\mathrm{supp}\,\eta\,\times\,(0,T)}a_{i}\,(|\partial_{i}u|-\delta_{i})_{+}^{p-1}\,(|\partial_{i}u-(\partial_{i}u)_{\bar{h}}|\,+\,|u-u_{\bar{h}}|)\,dx\,dt\nonumber\\
&\,\,\,\,\,\,\,\,\,\,\,\,\,\,+\,\frac{8}{T}\iint_{\mathrm{supp}\,\eta\,\times\,(\frac{1}{2}T,\frac{3}{4}T)}\mathfrak{b}_{\alpha}[u,u_{\bar{h}}]\,dx\,dt\,+\iint_{\mathrm{supp}\,\eta\,\times\,(0,T)}|f|\,|u-u_{\bar{h}}|\,dx\,dt
\end{align} for all $\tau\in[0,\tfrac{1}{2}T]\setminus N_{h}$, where $N_{h}$
is a set of measure zero. Our goal is to investigate the limit as
$h\to0$. Assumption (\ref{eq:coeff_limit}) on the coefficients $a_{i}$,
the integrability properties of $u$ and $\partial_{i}u$, together
with Lemma \ref{lem:expmolproperties} (i) and (iii), imply that the
first integral on the right-hand side of \foreignlanguage{british}{\eqref{eq:gs}}
converges to zero as $h\to0$. The integral involving $f$ also vanishes
in the limit by Lemma \ref{lem:expmolproperties} and assumption \eqref{assumpt:extra}.
In order to treat the terms involving $\mathfrak{b}_{\alpha}$, we
need to distinguish between the cases $\alpha\in(0,1)$ and $\alpha\geq1$.\\
\foreignlanguage{british}{$\hspace*{1em}$}If $\alpha\in(0,1)$, the
integrand on the left-hand side of \foreignlanguage{british}{\eqref{eq:gs}}
can be estimated using the left inequality of \eqref{est:exponent_inside}
with $\gamma=\frac{\alpha+1}{2\alpha}>1$ and \eqref{est:b-all-alpha}
as follows:\begin{align*}
&|\vert u\vert^{\alpha-1}u-\vert u_{\bar{h}}\vert^{\alpha-1}u_{\bar{h}}|^{\frac{\alpha+1}{\alpha}}\\
&\,\,\,\,\,\,\,=\,|\vert u\vert^{\alpha-1}u-\vert u_{\bar{h}}\vert^{\alpha-1}u_{\bar{h}}|^{2\,\frac{\alpha+1}{2\alpha}}\,\leq\,c(\alpha)\,\vert\vert u\vert^{\frac{\alpha-1}{2}}u-\vert u_{\bar{h}}\vert^{\frac{\alpha-1}{2}}u_{\bar{h}}\vert^{2}\,\leq\,c(\alpha)\,\mathfrak{b}_{\alpha}[u,u_{\bar{h}}]\,.
\end{align*}For the term in the last line of \foreignlanguage{british}{\eqref{eq:gs}}
we can use Lemma \ref{lem:bdry_term_estimates_alpha_small} (ii) to
obtain the estimate
\begin{align*}
\mathfrak{b}_{\alpha}[u,u_{\bar{h}}]\, & \leq\,c(\alpha)\,\vert u-u_{\bar{h}}\vert^{1+\alpha}\,=\,c(\alpha)\,\vert u-u_{\bar{h}}\vert^{\alpha}\,\vert u-u_{\bar{h}}\vert\,\leq\,c(\alpha)\,(\vert u\vert^{\alpha}+\vert u_{\bar{h}}\vert^{\alpha})\,|u-u_{\bar{h}}|\,.
\end{align*}
The factor $(\vert u\vert^{\alpha}+\vert u_{\bar{h}}\vert^{\alpha})$
remains bounded in $L^{\frac{\alpha+1}{\alpha}}$ as $h\to0$, while
the term $|u-u_{\bar{h}}|$ converges to zero in $L^{\alpha+1}$ as
$h\to0$. Choosing now a sequence $h_{j}\to0$, setting $w_{j}=u_{\overline{h_{j}}}\,$,
and defining $N:=\cup N_{h_{j}}$ (which has measure zero), we see
that \foreignlanguage{british}{\eqref{eq:gs}} combined with the previous
observations implies
\begin{align}
\lim_{j\to\infty}\,\sup_{\tau\,\in\,[0,\frac{1}{2}T]\setminus N}\int_{K}|\vert u\vert^{\alpha-1}u-\vert w_{j}\vert^{\alpha-1}w_{j}|^{\frac{\alpha+1}{\alpha}}(x,\tau)\,dx\,=\,0\,.\label{unif_limit}
\end{align}
As noted earlier, each $\vert w_{j}\vert^{\alpha-1}w_{j}$ is continuous
as a map $[0,T]\to L^{\frac{1}{\alpha}+1}(K)$. This fact, together
with the uniform limit \eqref{unif_limit} on the dense set $[0,\frac{1}{2}T]\setminus N$
and the completeness of $L^{\frac{1}{\alpha}+1}(K)$, shows that $\vert w_{j}\vert^{\alpha-1}w_{j}$
converges uniformly on $[0,\frac{1}{2}T]$ to a limit function which
is continuous into $L^{\frac{1}{\alpha}+1}(K)$. Due to \eqref{unif_limit},
this limit is a representative of $\vert u\vert^{\alpha-1}u$.

\selectlanguage{british}%
\noindent $\hspace*{1em}$\foreignlanguage{american}{In the case $\alpha\geq1$,
we instead estimate the left-hand side of }\eqref{eq:gs}\foreignlanguage{american}{
using Lemma \ref{lem:bdry_term_estimates_alpha_large} (ii): 
\begin{align*}
|u-u_{\bar{h}}|^{\alpha+1}\,\leq\,c(\alpha)\,\mathfrak{b}_{\alpha}[u,u_{\bar{h}}]\,.
\end{align*}
On the right-hand side of }\eqref{eq:gs}\foreignlanguage{american}{,
we denote $L=\mathrm{supp}\,\eta\times(\frac{1}{2}T,\frac{3}{4}T)$
and estimate the term involving $\mathfrak{b}_{\alpha}$ using Lemma
\ref{lem:bdry_term_estimates_alpha_large} (i), the right-hand inequality
in \eqref{est:exponent_inside} with $\gamma=\alpha$ and, when $\alpha>1$,
Hölder's inequality, thus obtaining:
\begin{align*}
\iint_{L}\mathfrak{b}_{\alpha}[u,u_{\bar{h}}]\,dx\,dt\, & \leq\,c(\alpha)\iint_{L}|\vert u\vert^{\alpha-1}u-\vert u_{\bar{h}}\vert^{\alpha-1}u_{\bar{h}}|^{\frac{\alpha+1}{\alpha}}\,dx\,dt\\
 & \leq\,c(\alpha)\iint_{L}|\vert u\vert^{\alpha-1}+\vert u_{\bar{h}}\vert^{\alpha-1}|^{\frac{\alpha+1}{\alpha}}\,|u-u_{\bar{h}}|^{\frac{\alpha+1}{\alpha}}\,dx\,dt\\
 & \leq\,c(\alpha)\,\Big[\iint_{L}(\vert u\vert^{\alpha+1}+\vert u_{\bar{h}}\vert^{\alpha+1})\,dx\,dt\Big]^{\frac{\alpha-1}{\alpha}}\Big[\iint_{L}|u-u_{\bar{h}}|^{\alpha+1}\,\,dx\,dt\Big]^{\frac{1}{\alpha}}.
\end{align*}
The first integral in the last line remains bounded as $h\to0$, while
the second integral vanishes in the limit by Lemma \ref{lem:expmolproperties}
(i). Therefore, we can argue as in the case $\alpha\in(0,1)$ and
obtain a sequence $h_{j}\to0$ and a set $N\subset[0,\frac{1}{2}T]$
of measure zero such that 
\begin{align*}
\lim_{j\to\infty}\,\sup_{\tau\,\in\,[0,\frac{1}{2}T]\setminus N}\int_{K}|u-w_{j}|^{\alpha+1}(x,\tau)\,dx\,=\,0\,,
\end{align*}
which implies the time-continuity of $u$ on the interval $[0,\frac{1}{2}T]$.}

\noindent $\hspace*{1em}$\foreignlanguage{american}{The continuity
on $[\tfrac{1}{2}T,T]$ follows in both ranges of $\alpha$ from similar
arguments with $w=u_{h}$ and with $\psi$ and $\chi_{\varepsilon}^{\tau}$
mirrored on the interval $[0,T]$ under the map $t\mapsto T-t$.\end{proof}}

\selectlanguage{american}%
\noindent \begin{brem}We observe that the arguments presented in
this appendix extend, more generally, to doubly nonlinear diffusion
equations in divergence form whose diffusion flux is given by a measurable
vector field $\mathbf{A}(x,t,u,\nabla u)$ satisfying a standard $p$-growth
condition with respect to the gradient variable. In fact, these arguments
remain valid even if the vector field in the weak formulation (\ref{eq:weak_form-again})
is replaced by any function belonging to $L^{p'}(\Omega_{T};\mathbb{R}^{N})$.
This observation is useful in the existence theory for doubly nonlinear
evolution equations; see, for instance, \cite{Vestb}, where analogous
arguments are employed in the fully anisotropic setting.\end{brem}

\selectlanguage{british}%
\medskip{}

\begin{singlespace}
\noindent \textbf{Acknowledgments. }This work has been partially supported
by the Wallenberg AI, Autonomous Systems and Software Program (WASP)
funded by the Knut and Alice Wallenberg Foundation. Pasquale Ambrosio
is a member of the GNAMPA group of INdAM, which partially supported
his research through the INdAM--GNAMPA 2026 Project ``Esistenza
e regolarità per soluzioni di equazioni ellittiche e paraboliche anisotrope''
(CUP E53C25002010001).\bigskip{}

\noindent \textbf{Declarations.} On behalf of all authors, the corresponding
author states that there is no conflict of interest.\bigskip{}

\noindent \textbf{Data availability.} This manuscript has no associated
data.\addcontentsline{toc}{section}{References}
\end{singlespace}

\noindent \textbf{$\quad$}
\end{document}